\documentclass[10pt,a4paper]{amsart}

\usepackage{color}

\usepackage[utf8]{inputenc}
\usepackage{stmaryrd}

\newcommand{\drpullback}{\arrow[phantom]{dr}[very near start,description]{\lrcorner}}
\newcommand{\dlpullback}{\arrow[phantom]{dl}[very near start,description]{\llcorner}}
\newcommand{\drpullbackS}{\arrow[phantom, start anchor=south]{dr}[very near start,description]{\lrcorner}}

\newcommand{\fib}[2]{{#1}_{#2}}

\newcommand{\freemonad}{\overline}

\newcommand{\act}{\mathrm{act}}

\newcommand{\bottomnode}{\mathrm{bot}}

\newcommand{\leaves}{\mathrm{leaves}}

\newcommand{\el}{\operatorname{el}}
\newcommand{\corol}{\operatorname{cor}}
\newcommand{\nTr}{n\textrm{-}\mathrm{tr}}

\newcommand{\xint}{\mathrm{int}}

\usepackage{amsmath,amssymb,amsthm}
\usepackage[utf8]{inputenc}
\usepackage{url}
\usepackage[shortlabels]{enumitem}
  \setitemize[1]{leftmargin=2em}
  \setenumerate[1]{leftmargin=*}
\usepackage[colorlinks=true,linktocpage=true,citecolor=blue]{hyperref}
\usepackage[capitalise]{cleveref}
\crefformat{equation}{(#2#1#3)}

\newtheorem{thm}{Theorem}[subsection]
\newtheorem{lemma}[thm]{Lemma}
\newtheorem{propn}[thm]{Proposition}
\newtheorem{cor}[thm]{Corollary}

\newtheorem*{thm*}{Theorem}
\newtheorem*{conjecture*}{Conjecture}
\newtheorem*{goal*}{Goal}
\newtheorem*{question*}{Question}
\newtheorem*{prethm*}{Pretheorem}
\theoremstyle{definition}
\newtheorem{defn}[thm]{Definition}

\newtheorem{notation}[thm]{Notation}

\newtheorem{warning}[thm]{Warning}

\newtheorem{remark}[thm]{Remark}

\newtheorem{construction}[thm]{Construction}

\theoremstyle{remark}

\usepackage[alphabetic]{amsrefs}

\theoremstyle{definition}

\newcommand{\isoto}{\xrightarrow{\sim}}
\newcommand{\isofrom}{\xleftarrow{\sim}}

\usepackage{eucal}

\newcommand{\IFF}{if and only if}

\newcommand{\catname}[1]{\ensuremath{\text{\textup{#1}}}}
\newcommand{\txt}[1]{\ensuremath{\text{\textup{#1}}}}
\newcommand{\blank}{\txt{\textendash}}
\newcommand{\Set}{\catname{Set}}

\newcommand{\Cat}{\catname{Cat}}
\newcommand{\CatI}{\catname{Cat}_\infty}
\newcommand{\LCatI}{\widehat{\catname{Cat}}_\infty}

\newcommand{\Fun}{\txt{Fun}}

\newcommand{\Map}{\txt{Map}}

\newcommand{\Aut}{\txt{Aut}}

\newcommand{\op}{\txt{op}}

\newcommand{\icat}{$\infty$-category}
\newcommand{\icats}{$\infty$-categories}
\newcommand{\icatl}{$\infty$-categorical}
\newcommand{\itcat}{$(\infty,2)$-category}
\newcommand{\itcats}{$(\infty,2)$-categories}

\newcommand{\xto}[1]{\xrightarrow{#1}}

\newcommand{\from}{\leftarrow}
\newcommand{\xfrom}[1]{\xleftarrow{#1}}

\usepackage{color}

\newcommand{\tr}{\txt{tr}}

\usepackage{tikz}
\usetikzlibrary{matrix,arrows}
\usepackage{tikz-cd}

\newcommand{\csquare}[8]{ %
\[ %
\begin{tikzpicture} %
\matrix (m) [matrix of math nodes,row sep=3em,column sep=2.5em,text height=1.5ex,text depth=0.25ex] %
{ #1 \pgfmatrixnextcell #2 \\ %
  #3 \pgfmatrixnextcell #4 \\ }; %
\path[->,font=\footnotesize] %
(m-1-1) edge node[auto] {$#5$} (m-1-2)%
(m-1-1) edge node[left] {$#6$} (m-2-1)%
(m-1-2) edge node[auto] {$#7$} (m-2-2)%
(m-2-1) edge node[below] {$#8$} (m-2-2);%
\end{tikzpicture}%
\]%
}

\newcommand{\nolabelcsquare}[4]{\csquare{#1}{#2}{#3}{#4}{}{}{}{}}

\newcommand{\opctriangle}[6]{ %
\[ %
\begin{tikzpicture} %
\matrix (m) [matrix of math nodes,row sep=3em,column sep=1.2em,text height=1.5ex,text depth=0.25ex] %
{  #1 \pgfmatrixnextcell \pgfmatrixnextcell #2 \\ %
  \pgfmatrixnextcell #3 \pgfmatrixnextcell \\ %
}; %
\path[->,font=\footnotesize] %
(m-1-1) edge node[above] {$#4$} (m-1-3)%
(m-1-1) edge node[below left] {$#5$} (m-2-2)%
(m-1-3) edge node[below right] {$#6$} (m-2-2);%
\end{tikzpicture}%
\]%
}

\newcommand{\id}{\txt{id}}

\DeclareMathOperator{\colimP}{colim}
\newcommand{\colim}{\mathop{\colimP}}

\newcommand{\CATI}{\txt{CAT}_{\infty}}

\newcommand{\bbDelta}{\boldsymbol{\Delta}}

\newcommand{\bbTheta}{\boldsymbol{\Theta}}

\newcommand{\simp}{\bbDelta}

\newcommand{\Mnd}{\txt{Mnd}}

\newcommand{\Alg}{\catname{Alg}}

\newcommand{\FUN}{\txt{FUN}}

\newcommand{\iopd}{$\infty$-operad}
\newcommand{\iopds}{$\infty$-operads}

\newcommand{\Sq}{\txt{Sq}}

\makeatletter
\def\@tocline#1#2#3#4#5#6#7{\relax
  \ifnum #1>\c@tocdepth % then omit
  \else
    \par \addpenalty\@secpenalty\addvspace{#2}%
    \begingroup \hyphenpenalty\@M
    \@ifempty{#4}{%
      \@tempdima\csname r@tocindent\number#1\endcsname\relax
    }{%
      \@tempdima#4\relax
    }%
    \parindent\z@ \leftskip#3\relax \advance\leftskip\@tempdima\relax
    \rightskip\@pnumwidth plus4em \parfillskip-\@pnumwidth
    #5\leavevmode\hskip-\@tempdima
      \ifcase #1
       \or \hskip -1em \or \hskip 1em \or \hskip 3em \else \hskip 5em \fi%
      #6\nobreak\relax
    \hfill\hbox to\@pnumwidth{\@tocpagenum{#7}}
      \par
    \nobreak
    \endgroup
  \fi}
\makeatother

\newcommand{\Dop}{\simp^{\op}}
\newcommand{\Fin}{\txt{Fin}}
\newcommand{\bbO}{\boldsymbol{\Omega}}

\newcommand{\PSeg}{\mathcal{P}_{\txt{Seg}}}

\newcommand{\bbOint}{\bbO_{\txt{int}}}
\newcommand{\bbOel}{\bbO_{\txt{el}}}
\newcommand{\bbOelT}{\bbO_{\txt{el}/T}}

\newcommand{\alg}{\txt{alg}}
\newcommand{\coalg}{\operatorname{coalg}}
\newcommand{\Tw}{\operatorname{Tw}}

\newcommand{\End}{\txt{End}}
\newcommand{\END}{\txt{END}}
\newcommand{\LEnd}{\widehat{\End}}
\newcommand{\LMnd}{\widehat{\Mnd}}

\newcommand{\AnEnd}{\txt{AnEnd}}

\newcommand{\AnMnd}{\txt{AnMnd}}

\newcommand{\Poly}{\txt{Poly}}

\newcommand{\PolyFun}{\txt{PolyFun}}

\newcommand{\PolyEnd}{\txt{PolyEnd}}

\newcommand{\AnFun}{\txt{AnFun}}

\newcommand{\Act}{\txt{Act}}

\newcommand{\SqL}{\txt{Sq}^{\txt{lax}}}

\newcommand{\SqCL}{\txt{Sq}^{\txt{colax}}}

\newcommand{\lax}{\txt{lax}}

\newcommand{\colax}{\txt{colax}}
\newcommand{\ladj}{\txt{ladj}}
\newcommand{\radj}{\txt{radj}}
\newcommand{\mndradj}{\txt{mndradj}}

\newcommand{\LMod}{\txt{LMod}}

\newcommand{\LCATI}{\widehat{\txt{CAT}}_{\infty}}

\newcommand{\fmnd}{\mathfrak{mnd}}

\newcommand{\LMndLSR}{\LMnd_{\lax}^{\sigma,\radj}}
\newcommand{\LEndLSR}{\LEnd_{\lax}^{\sigma,\radj}}
\newcommand{\LCatISR}{\LCatI^{\sigma,\radj}}
\newcommand{\LMndLSRop}{\LMnd_{\lax}^{\sigma,\radj,\op}}
\newcommand{\LEndLSRop}{\LEnd_{\lax}^{\sigma,\radj,\op}}
\newcommand{\LCatISRop}{\LCatI^{\sigma,\radj,\op}}

\newcommand{\LMndLop}{\LMnd_{\lax}^{\op}}

\newcommand{\Sym}{\mathrm{Sym}}

\renewcommand{\opctriangle}[6]{
\[%
\begin{tikzcd}%
  #1 \arrow[swap]{dr}{#5} \arrow{rr}{#4} \pgfmatrixnextcell \pgfmatrixnextcell #2 \arrow{dl}{#6} \\%
  {} \pgfmatrixnextcell #3 %
\end{tikzcd}%
\] %
}

\newcommand{\ie}{i.e.\@}

\newcommand{\CatIT}{\Cat_{(\infty,2)}}
\newcommand{\pcolax}{\txt{(co)lax}}

\newcommand{\MND}{\txt{MND}}

\newcommand{\fend}{\mathfrak{end}}

\newcommand{\bTt}{\bbTheta_{2}}

\usepackage{xstring}
\renewcommand{\eprint}[1]{\IfBeginWith{#1}{arXiv}{\href{https://arxiv.org/abs/#1}{#1}}{\href{#1}{#1}}}

\author{David Gepner}
\address{University of Illinois, Chicago, USA}
\email{gepner@uic.edu}
\urladdr{http://sites.google.com/view/dgepner/home}

\author{Rune Haugseng}
\address{Norwegian University of Science and Technology (NTNU), Trondheim, Norway}
\email{rune.haugseng@ntnu.no}
\urladdr{http://folk.ntnu.no/runegha/}

\author{Joachim Kock}
\address{Universitat Aut\`{o}noma de Barcelona \& Centre de Recerca 
Matem\`atica, Bellaterra, Spain}
\email{kock@mat.uab.cat}
\urladdr{http://mat.uab.cat/~kock/}

\date{\today}

\title{$\infty$-Operads as Analytic Monads}

\begin{document}

\begin{abstract}
  We develop an $\infty$-categorical version of the classical theory
  of polynomial and analytic functors, initial algebras, and free
  monads. Using this machinery, we provide a new model for \iopds{},
  namely $\infty$-operads as analytic monads.  We justify this
  definition by proving that the $\infty$-category of analytic monads
  is equivalent to that of dendroidal Segal spaces, known to
  be equivalent to the other existing models for \iopds{}.
\end{abstract}
\maketitle

\enlargethispage{1em}
\tableofcontents

\section{Introduction}
Operads are a powerful formalism for encoding algebraic
operations.  They were first introduced in the early seventies
for the purpose of describing up-to-homotopy algebraic
structures on topological spaces~\cite{MayGeomIter,
BoardmanVogt}, and have since become a standard tool also in
algebra, geometry, combinatorics, and mathematical physics.
Operads are closely related to monads, which were introduced
some 10 years earlier, implicitly with Godement's ``standard
construction'' of flasque resolutions~\cite{Godement},%
\footnote{Godement used (co)monads to construct (co)simplicial
resolutions, an essential tool for computations in algebra, geometry and topology.}
and explicitly by Huber~\cite{Huber}.
The notion soon spread from algebraic geometry
and homological algebra to universal algebra, logic, and computer science.
The relationship between operads and monads was
exploited from the very beginning of operad theory~\cite{MayGeomIter}, and is a major
theme of the present contribution.

Classically, an operad $\mathcal{O}$ consists of a sequence
$\mathcal{O}(n)$ of topological spaces, where $\mathcal{O}(n)$ is
equipped with an action of the symmetric group $\Sigma_{n}$ (this data
is called a \emph{symmetric sequence}), together with a unital and
associative composition law. The object $\mathcal{O}(n)$ describes the
$n$-ary operations of the operad.  Every symmetric sequence
$\mathcal{O}$ gives rise to an endofunctor
\[
F(X) = \coprod_{n} (\mathcal{O}(n) \times X^{\times n})_{\Sigma_{n}};
\]
endofunctors of this form are sometimes called {\em analytic functors} due
to their resemblance to power series.\footnote{This should not be
confused with the notion of analytic functor used in the setting of
Goodwillie calculus.}
When $\mathcal{O}$ is an operad, this
endofunctor acquires the structure of a monad, and the algebras for
the operad are canonically identified with the algebras for this
monad.

From a homotopical viewpoint, these topological operads (and their
associated algebras) have certain shortcomings, analogous to those
afflicting topological categories when viewed as a model for
``categories weakly enriched in spaces'' (or
$\infty$-categories). Just as these issues can be avoided by using a
better-behaved model for \icats{}, it is often convenient to work with
less rigid notions of ``operads weakly enriched in spaces'' or
\emph{$\infty$-operads}.
Indeed, even for well-known topological operads there are advantages
to viewing them as $\infty$-operads. For example, if $\mathbb{E}_{n}$ denotes the $\infty$-operad
corresponding to the classical operad of little $n$-discs (introduced
by May and Boardman--Vogt to study the algebraic structure of $n$-fold
loop spaces), then Lurie has proved a homotopically meaningful version of
Dunn's additivity theorem,
$\mathbb{E}_{n} \otimes \mathbb{E}_{m} \simeq \mathbb{E}_{n+m}$,
where $\otimes$ is 
the Boardman--Vogt tensor product of
  $\infty$-operads (which is well behaved, in contrast to the
  classical Boardman--Vogt tensor product of topological operads,
  which is not homotopy-invariant).
There are various
models for $\infty$-operads, the approach of Lurie~\cite{HA} being the
most well-developed at the moment.

In this work we introduce a new model for $\infty$-operads, in terms
of monads, and show that it is equivalent to the existing models.
As a consequence, we shall see that an $\infty$-operad can be
recovered from its free algebra monad, and obtain a characterization
of the monads that arise in this way.  One such characterization is
expressed by the following slogan:
\begin{center}
  \emph{$\infty$-operads are monads cartesian over the symmetric
  monad.}
\end{center}
Here the symmetric monad means the monad $\Sym$ associated to the
terminal \iopd{}; its underlying 
endofunctor on the 
$\infty$-category of spaces $\mathcal{S}$
is given by
$\Sym(X)\simeq \coprod_{n} X^{n}_{h\Sigma_{n}}$. If $T$ is a monad
over $\Sym$, then evaluating the natural transformation at the point
we get a space $T(*)$ over $\Sym(*) \simeq \coprod_n
B\Sigma_n$. 
This is precisely the same data as a symmetric sequence: the fibre of
$T(*)$ at the point of $B\Sigma_{n}$ gives the space of $n$-ary
operations with its $\Sigma_{n}$-action. The operad structure on this
symmetric sequence is encoded by the monad structure on the
endofunctor.

Being cartesian over $\Sym$ means we have a map of monads whose
underlying natural transformation is cartesian, i.e.~its naturality
squares are pullbacks. It turns out that such a natural transformation
to $\Sym$ is unique if it exists, so that being cartesian over $\Sym$
is a \emph{property} of a monad. We will see that this property has an
inherent characterization as the monad being \emph{analytic}, by which
we mean that it is cartesian (i.e.~its multiplication and unit
transformations are cartesian) and the underlying endofunctor
preserves sifted colimits and wide pullbacks (or equivalently all
weakly contractible limits). We can thus reformulate our slogan still more succinctly:
\begin{center}
  \emph{$\infty$-operads are analytic monads.}
\end{center}
This generalizes a classical description of operads in sets: by a
result of Joyal~\cite{JoyalAnalytique}, these are also equivalent to
analytic monads.

So far we have only discussed one-object operads, but it is quite often
useful to work with the more general notion of operads with many
objects (commonly called \emph{coloured operads} or \emph{symmetric
  multicategories}), which generalizes categories by allowing arrows
(operations) with multiple inputs instead of just one input. The term
$\infty$-operad usually denotes the higher-categorical version of this
more general notion of operad, and our slogan remains true with this
interpretation, provided we consider analytic monads on slices of
$\mathcal{S}$:
\begin{center}
\emph{\iopds{} with
space of objects $I$ are analytic monads on $\mathcal{S}_{/I}$.}
\end{center}
More precisely, in this paper we set up an \icat{} of analytic monads
(on all slices of $\mathcal{S}$ simultaneously) and prove that this is
equivalent to an existing model of \iopds{}, namely the dendroidal
Segal spaces of Cisinski and
Moerdijk~\cite{CisinskiMoerdijkDendSeg}. This model is known
  to be equivalent to other models of \iopds{}, including those of
  Lurie~\cite{HA} and Barwick~\cite{BarwickOpCat}, as well as to
  simplicial operads, thanks to results of
  Cisinski--Moerdijk~\cite{CisinskiMoerdijkDendSeg,
    CisinskiMoerdijkSimplOpd},
  Heuts--Hinich--Moerdijk~\cite{HeutsHinichMoerdijkDendrComp},
  Barwick~\cite{BarwickOpCat}, and
  Chu--Haugseng--Heuts~\cite{iopdcomp}.

  In order to study analytic monads, we first develop a theory
    of analytic functors between slices of $\mathcal{S}$, which can be
    viewed as a categorification of power series in many variables.
    In fact, these analytic functors turn out to be a special case of
    a more general notion of \emph{polynomial} functors, and our
	first task
	is to set up an \icatl{} framework for polynomial
    functors.
    This is
    in contrast to the situation in ordinary categories, where 
    analytic functors are {\em not} in general polynomial.

To make sense of this, let us explain what we actually
  mean by a polynomial functor.
 To any map of
spaces $f \colon I \to J$ there is associated a string of three
adjoint functors
$f_! \dashv f^* \dashv f_*$,
\[
\begin{tikzcd}
\mathcal{S}_{/I} \arrow[bend left=40]{r}{f_{!}} 
\arrow[bend right=40]{r}{f_{*}} & \arrow[swap]{l}{f^{*}} \mathcal{S}_{/J},
\end{tikzcd}
\]
where $f_{!}$ is given by composition with $f$ and $f^{*}$ by pullback
along $f$.\footnote{The fundamental nature of these three operations is
  witnessed by the fact that they correspond precisely to
  {\em substitution}, {\em dependent sums}, and {\em dependent  
  products}, the most basic building blocks of type   
  theory~\cite{HoTT-book}.}
Alternatively, we may identify
$\mathcal{S}_{/I}$ with $\Fun(I, \mathcal{S})$; then $f^{*}$ is given
by precomposition with $f$, and $f_{!}$ and $f_{*}$ are respectively
the left and right Kan extension functors along $f$. A
\emph{polynomial functor}\footnote{The notion of ``polynomial functor'' we consider
  here should not be confused with the notion of ``polynomial
  functor'' introduced by Eilenberg and Mac Lane and subsequently used
  in the study of functor homology, nor with the notion occurring in
  Goodwillie's calculus of functors.} is 
  a functor that is built as a composite of functors of these three
  kinds.
  The terminology ``polynomial functor'' has some drawbacks, such as the
   odd fact that  ``analytic'' is a special case of ``polynomial'', but we
   keep this terminology due to its long history in logic, computer
   science and category theory (see \S\ref{subsec:related} for some
   pointers).

The description of polynomial functors in terms of these fundamental
adjoints can be formulated as a representability property: a
polynomial functor $P \colon \mathcal{S}_{/I} \to \mathcal{S}_{/J}$
has a unique description as $t_{!}p_{*}s^{*}$ for a diagram of spaces
\[I \xfrom{s} E \xto{p} B \xto{t} J;\]
moreover, $P$ is analytic precisely when the homotopy fibres of the
map $p$ are
finite discrete spaces. Many questions about polynomial functors can
be handled by manipulating these representing diagrams, and our
description of \iopds{} as analytic monads allows us to leverage this
calculus of polynomial functors in the setting of \iopds{}. This
combinatorial interpretation of \iopds{} does
not have a direct analogue in the 1-categorical setting. In sets, the
endofunctors corresponding to most operads are not polynomial --- this
is only true for the so-called \emph{$\Sigma$-free} operads,
i.e.~those for which the actions of the symmetric groups are all
free.

In this paper, for the sake of emphasizing the key ideas, we consider
polynomial and analytic monads over (slices of) $\mathcal{S}$ only,
but it is an attractive feature of the polynomial formalism that it is
readily adaptable to more general contexts.  In particular, it
would seem to be a natural setting for notions of operads with non-discrete
arities, as required in certain situations beyond spaces.

\subsection{Overview of Results}

\subsubsection*{Polynomial Functors}
To carry out our programme, we first develop the higher-categorical version of
the basic theory of polynomial functors, roughly corresponding to the results of
Gambino--Kock~\cite{GambinoKock} in the case of ordinary categories.  In
view of the broad spectrum of applications of ordinary polynomial
functors, we expect that this theory will be of independent interest,
and hope that it can serve as a starting point for further
developments.

Our first main result is the following classification of polynomial
functors:
\begin{thm*}
  The following are equivalent for a functor
  $F \colon \mathcal{S}_{/I} \to \mathcal{S}_{/J}$:
  \begin{enumerate}[(i)]
  \item $F$ is a polynomial functor.
  \item $F$ is of the form $t_{!}p_{*}s^{*}$ for a diagram of spaces
    \[ I \xfrom{s} E \xto{p} B \xto{s} J.\]
  \item $F$ is accessible and preserves weakly contractible limits.
  \item $F$ is a local right adjoint.
  \end{enumerate}
\end{thm*}
\noindent{} Here a weakly contractible limit means a limit of a
diagram indexed by an \icat{} whose classifying space is contractible,
and a local right adjoint functor is a functor
$F \colon \mathcal{C} \to \mathcal{D}$ such that for every object
$x \in \mathcal{C}$ the induced functor
$\mathcal{C}_{/x} \to \mathcal{D}_{/Fx}$ is a right adjoint.

This characterization is the higher-categorical version of classical theorems
due to Lamarche, Taylor, Johnstone--Carboni, and Weber (as 
synthesized in ~\cite{GambinoKock}).
Its proof takes up \S\ref{subsec:polyfun}--\S\ref{subsec:lra}.

For our purposes, the relevant morphisms between polynomial functors
$\mathcal{S}_{/I} \to \mathcal{S}_{/J}$ are the cartesian natural
transformations. We show in \S\ref{subsec:morpolyfun} that these are
represented by diagrams of the form
\[
  \begin{tikzcd}
    {} & E' \drpullback
    \arrow{dd} \arrow{dl} \arrow{r} & B' 
    \arrow{dd} \arrow{dr} \\
    I  & & {} & J. \\
    & E \arrow{ul} \arrow[swap]{r} & B \arrow[swap]{ur} 
  \end{tikzcd}
\]

The interplay between the polynomial functors and the diagrams that
represent them (called polynomial diagrams) is a key aspect of the
theory: some features are most easily handled in terms of functors and
some more easily in terms of representing diagrams. To exploit this
fully we need to describe polynomial functors with varying source and
target in terms of diagrams.  To define such a general \icat{} of
polynomial functors we start by constructing a double \icat{} of
``colax squares'' in which the vertical arrows are right adjoints. We
have delegated its definition, in terms of lax natural transformations
as studied in \cite{adjmnd}, to Appendix~\ref{AppA}, where we also
discuss the naturality of the procedure of taking ``mates''; the aim
is to ensure coherence of all the Beck-Chevalley transformations once
and for all in a uniform way.  With this in place, we can define an
\icat{} $\PolyFun$ of polynomial functors and cartesian
transformations, and a (much simpler) \icat{} $\Poly$ of polynomial
diagrams (a subcategory of the \icat{} of diagrams in $\mathcal{S}$ of
shape $\bullet \from \bullet \to \bullet \to \bullet$). The main
result of \S\ref{subsec:icatpoly} is then that there is an equivalence
of \icats{}
   \[ \Poly \isoto \PolyFun\]
   over the source and target projections to $\mathcal{S} \times \mathcal{S}$.

In \S\ref{sec:colimpoly} we exploit this equivalence to show that
the colimit of a diagram of polynomial functors and cartesian
transformations is a polynomial functor, corresponding to the
(pointwise) colimit of the corresponding diagrams 
(Proposition~\ref{propn:colimpolyispoly}). We end the section
in \S\ref{subsec:slicepoly} by studying slices $\PolyFun_{/P}$ for $P$
a polynomial functor; we prove that these \icats{} are all
$\infty$-topoi (Theorem~\ref{thm:PolyFun/P=topos}).
Note that $\PolyFun$ itself is not even accessible (see
Remark~\ref{rmk:PolyFunnotacc}).

\subsubsection*{Analytic Functors}
In \S\ref{sec:anal} we study the special case of \emph{analytic}
functors, which we characterize by the equivalent conditions of the following theorem:
\begin{thm*}
  Let $\mathbf{E} \colon \mathcal{S} \to \mathcal{S}$ denote the
  polynomial functor $X \mapsto \coprod_{n = 0}^{\infty} X^{\times
    n}_{h\Sigma_{n}}$, represented by the diagram
  \[ * \from \coprod_{n} n_{h\Sigma_{n}} \to \coprod_{n} B\Sigma_{n}
  \to *.\] The following are equivalent for a functor $F \colon
  \mathcal{S}_{/I} \to \mathcal{S}_{/J}$:
  \begin{enumerate}[(i)]
  \item $F$ is a polynomial functor with a morphism to $\mathbf{E}$ (which
    is unique if it exists).
  \item $F$ is a polynomial functor, represented by a diagram 
    \[ I \xfrom{s} E \xto{p} B \xto{s} J\]
    where the map $p$ has finite discrete fibres.
  \item $F$ preserves sifted colimits and weakly contractible limits.
  \end{enumerate}
\end{thm*}
We are mostly interested in endofunctors.
For a functor $F \colon \mathcal{S}
\to \mathcal{S}$, condition (ii) implies that $F$ is of the form 
\[ X \mapsto \coprod_{n} (B_{n} \times X^{\times n})_{h\Sigma_{n}},\]
so our notion of analytic functors does indeed generalize the standard
definition for endofunctors of $\Set$. 
We observe that analytic endofunctors of $\mathcal{S}$ are
equivalent to symmetric sequences, and that they can also
be characterized as the left Kan extensions of homotopical species, 
meaning functors $\iota\Fin\to\mathcal{S}$ where $\iota \Fin$ denotes
the groupoid of finite sets and bijections --- this is an \icatl{}
version of a theorem of Joyal.
More generally, analytic
endofunctors of $\mathcal{S}_{/I}$ are equivalent to $I$-coloured
symmetric sequences (or symmetric $I$-collections), defined as
functors $\mathbf{E}(I)\times I \to \mathcal{S}$.

The combinatorics of trees enter all approaches to operads, explicitly
or otherwise. In the polynomial formalism, the interplay between trees
and operads is particularly intimate, since following \cite{KockTree}
we can define trees as certain polynomial endofunctors
\[
A\leftarrow N' \to N \to A,
\]
where $A$ is the set of edges, $N$ is the set of nodes, and $N'$ is
the set of nodes with a marked incoming edge. We thus have a full
subcategory $\bbOint$ of trees inside the \icat{} $\AnEnd$ of analytic
endofunctors. In \S\ref{subsec:treesanalend} we use this to show that
analytic endofunctors can be described in terms of trees --- more
precisely, we prove that
the restricted Yoneda embeddings give equivalences of \icats{}
  \[
  \AnEnd  \simeq \mathcal{P}(\bbOel) \simeq \PSeg(\bbOint).
  \]
Here $\bbOel$ is the full subcategory of \emph{elementary} trees,
which are the corollas and the trivial tree (the edge without nodes),
and $\PSeg(\bbOint)$ is the full subcategory of presheaves on
$\bbOint$ that satisfy a Segal condition, which can be interpreted as
a sheaf condition for the covers of trees by elementary subtrees.

\subsubsection*{Initial Algebras and Free Monads}
Our comparison result relies on understanding the free monad on an
analytic endofunctor. As a first step, we need to know that these
free monads actually exist, which is the main result of
\S\ref{sec:IAFM}. 
We follow the classical
approach using initial Lambek algebras, which goes back to
Ad\'amek~\cite{AdamekFreeAlgebras}; a standard reference for the
classical case is Kelly~\cite{Kellyunified}. For a finitary endofunctor $P$,
i.e.\ an endofunctor that preserves filtered colimits, we show in
\S\ref{subsec:lambek} that the \icat{} of Lambek algebras has an
initial object, constructed inductively; we present the construction
in terms of a bar-cobar adjunction for Lambek algebras, which appears
to be new. 

In \S\ref{subsec:freemndexist} we use the initial
algebra construction to exhibit a left adjoint to the forgetful functor
$\alg_P(\mathcal{C}) \to \mathcal{C}$, where $\alg_P(\mathcal{C})$ is the
\icat{} of Lambek algebras for $P$, and show that the resulting adjunction is
monadic (Proposition~\ref{propn:algPmonadic}).  The monad induced by the adjunction 
is the free monad on $P$, i.e.~characterized by a universal property (see 
Proposition~\ref{propn:freemnd}).
The forgetful functor from finitary monads on $\mathcal{C}$ to
finitary endofunctors thus has a left adjoint, taking an endofunctor
to its free monad.  We observe that, at least if we restrict to
endofunctors that preserve sifted colimits, this adjunction is itself
monadic (Corollary~\ref{cor:freesiftedmonad}).

Then, in \S\ref{subsec:freecolim} we give a more explicit description
of the underlying endofunctor of the free monad as the colimit of a
sequence of functors, which we will later exploit to understand
the free monad on an analytic endofunctor in terms of trees.

In \S\ref{subsec:freemndfamily} we extend our results to obtain a
monadic left adjoint to the forgetful functor from monads that
preserve sifted colimits to endofunctors of varying \icats{}. This
requires an \icat{} of monads over varying \icats{}, which is studied
in \cite{adjmnd}; we recall the results we need from there in
\S\ref{subsec:monad}.

\subsubsection*{Analytic Monads}
In 
\S\ref{sec:analytic-trees} we apply our results on free
monads in the special case of analytic monads. In
\S\ref{subsec:anmnd} we show that the free monad on an analytic
endofunctor exists and is again analytic, and the natural
transformations of the monad structure are cartesian. In
\S\ref{subsec:freeanalytictrees} we then show that the free monad on an
analytic endofunctor has an explicit description in terms of trees, giving:
\begin{thm*}
  The forgetful functor $\AnMnd \to \AnEnd$ from analytic monads to
  analytic endo\-functors has a left adjoint, taking an analytic
  endofunctor to its free monad, and the resulting adjunction is
  monadic. If $P$ is an analytic endofunctor given by the diagram
  \[ I \leftarrow E \to B \to I \]
  then the underlying endofunctor of the free monad on $P$ is
  represented by
  $$
  I \leftarrow \tr'(P) \to \tr(P) \to I   ,
  $$
  where $\tr(P)$ is the $\infty$-groupoid of $P$-trees, i.e.~trees with
  a morphism to $P$ in $\AnEnd$, and $\tr'(P)$ is the $\infty$-groupoid
  of $P$-trees with a marked leaf.
\end{thm*}
This is an \icatl{} version of a result from \cite{KockTree}.

\subsubsection*{Comparison with $\infty$-Operads}
We are now ready to establish the main result of the paper, namely the
equivalence between analytic monads and \iopds{}.

Let $\bbO$ be the full subcategory of $\AnMnd$ on the free monads on
trees; this is the polynomial description (cf.~\cite{KockTree})
of the
dendroidal category of Moerdijk and Weiss~\cite{MoerdijkWeiss}.
\begin{thm*}
  The restricted Yoneda functor $\AnMnd \to \mathcal{P}(\bbO)$ is fully
  faithful, and its essential image is $\PSeg(\bbO)$.  We thus have
  an equivalence of $\infty$-categories
  \[ \AnMnd \simeq \PSeg(\bbO).\]
\end{thm*}
\noindent
Here $\PSeg(\bbO)$ is the \icat{} of presheaves whose restriction to
$\bbOint$ lies in $\PSeg(\bbOint)$; these are precisely the dendroidal
Segal spaces.

The proof is inspired by the Nerve Theorem of Weber~\cite{WeberFamilial}.
The main ingredients are the  monadicity of the free monad adjunction,
the interpretation of analytic endofunctors as presheaves on
$\bbOel$, and the explicit description of the free monad in terms of trees.

\subsection{Related Work}
\label{subsec:related}

\subsubsection*{Operads}
A number of categorical descriptions exist for operads in
$\Set$.  While it is well known that {\em non-symmetric} operads
are equivalent to monads cartesian over the free-monoid monad
(and are hence automatically polynomial), it is not true that a
non-symmetric operad can be recovered from is monad
alone~\cite{LeinsterHigherOpds} --- the cartesian natural
transformation is a structure, not a property.  Leinster took
this as the starting point for a theory of generalized operads,
defined as monads cartesian over a fixed cartesian monad.  
Symmetric operads are not an instance of this notion,
though: they ought to be cartesian over the
free-commutative-monoid monad $\Sym$, but while the free-algebra
monad of a symmetric operad does admit a canonical monad map to
$\Sym$, neither the monads nor the map are cartesian in general.

The symmetric case can be handled with the notion of weakly
cartesian natural transformation, introduced by
Joyal~\cite{JoyalAnalytique}; see Weber~\cite{WeberGeneric} for
a systematic treatment.  Joyal proved that an endofunctor is
analytic if and only if it admits a weakly cartesian natural
transformation to $\Sym$, and showed that the category of
analytic functors and weakly cartesian natural transformations
is equivalent to that of symmetric sequences (or species).  This
equivalence is monoidal: composition of analytic functors
corresponds to the composition product of symmetric sequences,
which goes back to Kelly~\cite{KellyOnOperads}.  Kelly had
observed that operads are monoids in symmetric sequences, so it
follows that operads are analytic monads.  The characterization
of operads as weakly cartesian over $\Sym$ also follows.

An alternative way of overcoming the subtleties consists in
observing that while $\Sym$ is not cartesian on $\Set$, it {\em
is} cartesian as a $2$-monad on $\Cat$, as is important in
Kelly's theory of clubs~\cite{KellyClub}.
This was exploited by Weber~\cite{WeberOpdsPoly} to give a
characterization of symmetric $\Set$-operads as polynomial
monads cartesian over $\Sym$ in a certain $2$-categorical sense.
Weber's work was an important starting point for us.

It is a pleasant feature of the $\infty$-categorical setting
that these various approaches are unified in clean statements,
as expressed in the slogans of the introduction. These hold true
already over $1$-groupoids, but to get
a good description of analytic functors in $1$-groupoids one is
forced to pass to $2$-groupoids, etc., giving the usual infinite
ladder --- only for $\infty$-groupoids do we get a nice
self-contained theory.

\subsubsection*{$\infty$-Operads}
As we mentioned above, our description of \iopds{} as analytic monads
can be interpreted as an \icatl{} version of the classical description
of ($I$-coloured) operads as associative algebras in ($I$-coloured)
symmetric sequences. Another version of such a description of
\iopds{}, which also works for enriched \iopds{}, was recently
obtained by the second author \cite{HaugsengDayConv} by describing the
composition product using an extension of Day convolution to double
\icats{}. Alternatively, the composition product can be constructed
using free symmetric monoidal \icats{} (extending to \icats{} the
construction of \cite{TrimbleLie}); this approach is implemented in
the thesis of Brantner~\cite{BrantnerThesis}, though it has not yet
been compared to other models for \iopds{}.

\subsubsection*{Polynomial Functors}
The theory of polynomial functors has 
roots in
topology, representation theory, combinatorics,
logic and computer science.
For instance, the $1$-categorical version of our
Theorem~\ref{thm:pnchar} grew out of work on Girard's linear logic and
domain theory, and some of the basic results on polynomial functors
were first established in connection with semantics for generic data
types and polymorphic functions~\cite{AbbottAltenkirchGhani} (see
\cite{GambinoKock} for further background and references, and
\cite{KockData} for analytic functors in that context).  Moerdijk and
Palmgren~\cite{MoerdijkPalmgren} showed that initial algebras for
polynomial functors are semantics for W-types in (extensional)
Martin-L\"of type theory, the fundamental example being the natural
numbers as initial algebra for $X \mapsto 1+X$,
cf.~\cite{LawvereETCS,LambekFixpoint}.  With
the homotopy interpretation of type theory~\cite{HoTT-book}, a
full-blown intentional interpretation has recently been given by
Awodey--Gambino--Sojakova~\cite{AwodeyGambinoSojakovaJACM} as {\em
  homotopy} initial algebras.  (Generalized) $\infty$-operads are
expected to serve as semantics for the so-called higher inductive
types (see \cite{LumsdaineShulman}).  The
polynomial approach to $\infty$-operads might play some role in
fleshing out the semantics side of those ideas.

For their role in encoding both substitution and induction/recursion,
polynomial functors have also become an important tool for handling
the intricate combinatorial structures that arise in higher category
theory.  For example, polynomial monads were used to give a purely
combinatorial description of opetopes~\cite{KockJoyalBataninMascari},
and Batanin and Berger~\cite{BataninBergerPolynomial} have exploited
polynomial monads to give unified constructions of Quillen model
structures on categories of algebras.  Their paper has many references
to related developments.

\subsection{Acknowledgments}

We wish to thank
first of all Mark Weber, but also 
Pierre-Louis Curien, Nicola Gambino, Andr\'e Joyal,
Thomas Nikolaus, and Dimitri Zaganidis for helpful discussions.  
D.G.~was partially supported by NSF grants DMS-1406529 and DMS-1714273.
J.K.~was partially supported by grants MTM2016-80439-P (AEI/FEDER, UE) 
of Spain and 2017-SGR-1725 of Catalonia.

\section{Polynomial Functors}\label{sec:poly}

\subsection{Polynomial Functors}\label{subsec:polyfun}
We write $\mathcal{S}$ for the $\infty$-category of spaces 
(also known as $\infty$-groupoids or homotopy types); 
in the
model of \icats{} as quasicategories this can be explicitly defined as
the coherent nerve of the simplicial category of Kan complexes.

If $f \colon I \to J$ is a map of spaces, then $f$ induces three adjoint
functors between the slice \icats{} $\mathcal{S}_{/I}$ and
$\mathcal{S}_{/J}$: Composition with $f$ gives a functor 
\[f_{!} \colon \mathcal{S}_{/I} \to \mathcal{S}_{/J}\]
which is left adjoint to the functor 
\[ f^{*} \colon \mathcal{S}_{/J} \to \mathcal{S}_{/I}\]
given by pullback along $f$. The functor $f^{*}$ also has a right
adjoint
\[ f_{*} \colon \mathcal{S}_{/I} \to \mathcal{S}_{/J}\]
since $\mathcal{S}$ is locally cartesian closed. If we interpret the
slice \icats{} $\mathcal{S}_{/I}$ as functor \icats{} $\Fun(I, \mathcal{S})$ using the straightening
equivalence, then the functor $f^{*} \colon \Fun(J, \mathcal{S}) \to
\Fun(I, \mathcal{S})$ is given by precomposition with $f$, and $f_{!}$
and $f_{*}$ are given by left and right Kan extension along $f$.

\begin{defn} 
  A \emph{polynomial functor} is a functor $\mathcal{S}_{/I}\to
  \mathcal{S}_{/J}$ of the form $t_{!}p_{*}s^{*}$ corresponding to a
  diagram of spaces
  \[ I\xfrom{s} E \xto{p} B \xto{t} J.\]
\end{defn}

\begin{remark}
  In this paper we only consider polynomial functors in the context of
  the \icat{} of spaces, since this is the appropriate setting for
  \iopds{}. It is possible to consider polynomial functors in the more
  general setting of an arbitrary $\infty$-topos
  (or a locally cartesian closed \icat{},
  as treated in \cite{GepnerKock}), and we expect that most of our
  results can be generalized to this context. However, this would
  require working in the setting of internal \icats{}, which has not
  yet been sufficiently developed. For example, instead of natural transformations
  between polynomial functors we must use the analogue of so-called
  \emph{strong} natural transformations (cf.~\cite{GambinoKock}), or
  equivalently fibred natural transformations
  (cf.~\cite{KockKock}). In ordinary category theory, polynomial
  functors have also been considered~\cite{WeberPn} in general
  categories with pullbacks, at the price of having to impose
  an exponentiability condition separately on the middle maps in the
  diagrams.
\end{remark}

A basic fact about polynomial functors is that they compose 
(cf.~Theorem~\ref{thm:comp} below).  This 
result amounts to being able to rewrite any composite of upper-star,
lower-star and lower-shriek functors in the normal form of the definition.
This is achieved through Beck-Chevalley transformations and 
distributivity, which we proceed to discuss.
This works essentially as in the 1-categorical
case~\cite{GambinoKock}.  Our treatment follows \cite{WeberPn}*{\S
2.2}.

\begin{defn}
  A natural transformation $\phi \colon F \to G$ of functors $F,G
  \colon \mathcal{C} \to \mathcal{D}$ is \emph{cartesian} if for every
  morphism $f \colon C \to C'$ in $\mathcal{C}$ the commutative square
  \csquare{FC}{FC'}{GC}{GC'}{Ff}{\phi_{C}}{\phi_{C'}}{Gf}
  is cartesian.
\end{defn}

\begin{remark}\label{rmk:cart*}
  If $\mathcal{C}$ has a terminal object $*$, then by the 2-of-3
  property of pullback squares a natural
  transformation $\phi$ as above is cartesian \IFF{} for every object
  $c \in \mathcal{C}$ the naturality square
  \csquare{Fc}{F*}{Gc}{G*}{}{\phi_{c}}{\phi_{*}}{}
  is cartesian.
\end{remark}

\begin{lemma}\label{lem:!Cart}
  Suppose $\mathcal{C}$ is an \icat{} with pullbacks. For any morphism
  $f \colon S \to T$ in $\mathcal{C}$ we have a functor $f_{!}\colon
  \mathcal{C}_{/S} \to \mathcal{C}_{/T}$ given by composition with
  $f$, with right adjoint $f^{*}$ given by pullback along $f$.
  The counit and unit transformations $f_{!}f^{*} \to
  \id$ and $\id \to f^{*}f_{!}$ for the adjunction $f_{!} \dashv
  f^{*}$ are cartesian.
\end{lemma}
\begin{proof}
  It suffices to check that the naturality squares for the map to the
  terminal object is cartesian in both cases. For the counit
  transformation at $q \colon Y \to
  T$ this is obvious, since the naturality square is 
  \csquare{S \times_{T} Y}{Y}{S}{T.}{}{}{q}{f}
 For $p \colon X \to S$ in $\mathcal{C}_{/S}$, consider the diagram
  \[
  \begin{tikzcd}
    X \arrow[swap]{d}{p} \arrow{r} \arrow[bend left=20,equal]{rr} & S
    \times_{T} X \arrow{r} \arrow{d} & X \arrow{d}{p} \\
    S \arrow{r} \arrow[equal]{dr} & S \times_{T}S \arrow{r} \arrow{d}
    & S \arrow{d}{f} \\
    {} & S \arrow[swap]{r}{f} & T.
  \end{tikzcd}
\]
  Here the top left square is the naturality square for the unit at
  $p$.
  In the right column the bottom square and the
  composite square are cartesian, hence so is the top right
  square. The composite in the top row is also cartesian, whence the
  top left square is cartesian, as required.
\end{proof}

\begin{lemma}\label{lem:BCpbk}
  For a commutative diagram of spaces
  \csquare{A}{B}{C}{D}{u}{g}{f}{v}
  the following are equivalent:
  \begin{enumerate}[(i)]
	  \item The square is cartesian.
	  
	  \item The Beck-Chevalley transformation
  \[ u_{!}g^{*} \to u_{!}g^{*}v^{*}v_{!} \simeq
  u_{!}u^{*}f^{*}v_{!}\to f^{*}v_{!}\]
  is an equivalence.

	  \item
  The Beck-Chevalley transformation
  \[ v^{*}f_{*} \to g_{*}g^{*}v^{*}f_{*} \simeq g_{*}u^{*}f^{*}f_{*}
  \to g_{*}u^{*}\]
  is an equivalence.  
  \end{enumerate}
\end{lemma}
\begin{proof}
  (i)$\Leftrightarrow$(ii): To see that the Beck-Chevalley
  transformation in (ii) is cartesian we use Lemma~\ref{lem:!Cart}:
  this implies that the unit $\id \to v^{*}v_{!}$ and counit
  $u_{!}u^{*} \to \id$ are cartesian transformations, and the functors
  $u_{!}$, $g^{*}$, $f^{*}$,$v_{!}$ all preserve pullbacks (for the
  left adjoints this follows for instance from
  Lemma~\ref{lem:slicecreates}). By Remark~\ref{rmk:cart*} it is
  therefore a natural equivalence \IFF{} the map
  $u_{!}g^{*}(\id_{C}) \to f^{*}v_{!}(\id_{C})$ in $\mathcal{S}_{/B}$
  is an equivalence.  Here
  $u_{!}g^{*}(\id_{C}) \simeq u_{!}(\id_{A}) \simeq u$ and
  $f^{*}v_{!}(\id_{C}) \simeq f^{*}(v)$, and the map $u \to f^{*}v$ is
  given by the natural map from $A$ to the pullback of $v$ along
  $f$. Since the forgetful functor $\mathcal{S}_{/B} \to \mathcal{S}$
  is conservative, we see that the square is indeed cartesian \IFF{}
  this map is an equivalence.

  (ii)$\Leftrightarrow$(iii) follows since the two transformations are
  mates: (iii) is obtained from (ii) by taking right adjoints, and
  (ii) from (iii) by taking left adjoints.
\end{proof}

\begin{propn}
  Given maps of spaces  $f
  \colon X \to Y$  and $g \colon E \to
  X$, put $r := f_{*}g \colon E' \to Y$ and $h := f^{*}r \colon E'' \to X$,
  to get a commutative diagram
  \[
  \begin{tikzcd}
   {} & E'' \drpullback
   \arrow[swap]{dl}{\epsilon} \arrow{r}{q} \arrow{dd}{h} & E' \arrow{dd}{r} \\
   E \arrow[swap]{dr}{g} & & {} \\
    & X \arrow[swap]{r}{f} & Y,
  \end{tikzcd}
  \]
  where $\epsilon$ is the counit for the adjunction $f^{*} \dashv
  f_{*}$ and $q$ is the pullback of $f$ along $r$.
  Then the natural transformation $\delta \colon r_{!}q_{*}\epsilon^{*} \to
  f_{*}g_{!}$, defined as the composite
  \[ r_{!}q_{*}\epsilon^{*} \to r_{!}q_{*}\epsilon^{*}g^{*}g_{!} \simeq
  r_{!}q_{*}h^{*}g_{!} \isofrom r_{!}r^{*}f_{*}g_{!} \to f_{*}g_{!}\]
  is an equivalence.
\end{propn}

\begin{proof}
  The natural transformation $\delta$ is cartesian, since by
  Lemma~\ref{lem:!Cart} it is a composite of cartesian
  transformations. It therefore suffices to show that the component of
  $\delta$ at $\id_{E}$ is an equivalence. We have
  $r_{!}q_{*}\epsilon^{*}(\id_{E}) \simeq r_{!}(\id_{E'}) \simeq r$ (since
  $q_{*}\epsilon^{*}$ preserves the terminal object) and $f_{*}g_{!}(\id_{E})
  \simeq f_{*}(g)$ which by definition is also $r$.  Tracing through 
  the maps constituting the transformation reveals that the actual 
  map from $r$ to $r$ is the diagonal followed by a projection, 
  which is the identity.
\end{proof}

We can now give an explicit description of the composite of
polynomial functors:
\begin{thm}\label{thm:comp} 
  Suppose $P \colon \mathcal{S}_{/I} \to \mathcal{S}_{/J}$ and $Q
  \colon \mathcal{S}_{/J} \to \mathcal{S}_{/K}$ are polynomial
  functors, represented by diagrams of spaces
  \[ I \xfrom{s} E \xto{p} B \xto{t} J,\]
  \[ J \xfrom{u} F \xto{q} C \xto{v} K,\]
  respectively. Consider the commutative diagram of spaces
  \[
  \begin{tikzcd}
    {} &   &   & G \arrow[bend left=32]{rrr}{w} 
	\arrow[bend right=20]{dddlll}[above]{r} \arrow{rr}{p''} \arrow{dl}{\epsilon'} 
	\drpullback &                
	& X \drpullback
	\arrow{r}{q'} 
	\arrow{dl}{\epsilon}
    \arrow{dd}{q^{*}q_{*}\pi}& D \arrow{dd}{q_{*}\pi} \arrow[bend left=15]{dddr}{x} \\
    {} &   & Y \drpullback \arrow{rr}{p'} \arrow{dl}{u''} &   &  B \times_{J} F 
	\arrow[phantom]{dd}[very near start]{\rotatebox{-45}{$\lrcorner$}} 
	\arrow{dl}{u'} \arrow{dr}{\pi} && {} \\
    {} & E\arrow[swap]{rr}{p} \arrow{dl}{s} &   & B \arrow[swap]{dr}{t} &     { }           &
    F \arrow{dl}{u}  \arrow[swap]{r}{q} & C \arrow[swap]{dr}{v} \\
    I  &   &   &   & J              &   &   & K
  \end{tikzcd}
  \]
  where $\epsilon$ is the counit map $q^{*}q_{*}\pi \to \pi$ for the
  adjunction $q^{*} \dashv q_{*}$, and the squares are all pullbacks.
  Then the composite $Q \circ P \colon \mathcal{S}_{/I}
  \to \mathcal{S}_{/K}$ is the polynomial functor represented by the
  diagram 
  \[ I \xfrom{r} G \xto{w} D \xto{x} K.\]
\end{thm}

\begin{proof}
  We have natural equivalences
  \[
  \begin{split}
    v_{!}q_{*}u^{*}t_{!}p_{*}s^{*} & \simeq
    v_{!}q_{*}\pi_{!}u'^{*}p_{*}s^{*} \qquad \txt{(using the
      Beck-Chevalley equivalence $u^{*}t_{!} \simeq
      \pi_{!}(u')^{*}$)} \\ 
    & \simeq v_{!}(q_{*}\pi)_{!}
    q'_{*}\epsilon^{*}u'^{*}p_{*}s^{*} \qquad \txt{(using the
      distributivity equivalence $q_{*}\pi_{!} \simeq
      (q_{*}\pi)_{!}q'_{*}\epsilon^{*}$)} \\
    & \simeq x_{!}q'_{*}p''_{*}\epsilon'^{*}u''^{*}s^{*} \qquad \txt{(using the
      Beck-Chevalley equivalence $\epsilon^{*}u'^{*}p_{*} \simeq
      p''_{*}\epsilon'^{*}u''^{*}$)} \\
     & \simeq x_{!}w_{*}r^{*}. \qedhere
  \end{split}
\]
\end{proof}

\begin{cor}\label{cor:compose}
  The composite of two polynomial functors is again a polynomial
  functor. \qed
\end{cor}

\begin{remark}
  This corollary also follows from the characterization of polynomial
  functors we will prove below in Theorem~\ref{thm:pnchar}, but we
  will also need the explicit formula for the composition given by Theorem~\ref{thm:comp}.
\end{remark}

\subsection{Local Right Adjoints}\label{subsec:lra}
In this subsection we will prove an alternative characterization of
polynomial functors. To state this we must first introduce some
terminology:
\begin{defn}
  A functor $F \colon \mathcal{C} \to \mathcal{D}$ between
  $\infty$-categories is a \emph{local right adjoint} if for every $x
  \in \mathcal{C}$ the induced functor $\mathcal{C}_{/x} \to
  \mathcal{D}_{/Fx}$ is a right adjoint.
\end{defn}

\begin{defn}
  The inclusion $\mathcal{S} \hookrightarrow \CatI$ of spaces into the $\infty$-category of small $\infty$-categories has a left
  adjoint, which takes an \icat{} $\mathcal{C}$ to the space obtained by inverting
  all morphisms in $\mathcal{C}$, which we denote
  $\|\mathcal{C}\|$. We say that $\mathcal{C}$ is \emph{weakly
    contractible} if $\|\mathcal{C}\|$ is a contractible space.
\end{defn}

\begin{thm}\label{thm:pnchar}
  The following are equivalent for a functor $F \colon
  \mathcal{S}_{/I} \to \mathcal{S}_{/J}$:
  \begin{enumerate}[(i)]
  \item $F$ is a polynomial functor.
  \item $F$ is accessible and preserves weakly contractible limits.
  \item $F$ is a local right adjoint.
  \end{enumerate}
\end{thm}

\begin{remark}
  Theorem~\ref{thm:pnchar} is specific to the $\infty$-category
  $\mathcal{S}$ (and its truncations, such as the category of sets).
  Even over a presheaf topos it is not true in general that a local
  right adjoint is always polynomial, as exemplified by the
  free-category monad on directed graphs \cite{WeberFamilial}.  The
  corresponding theorem in ordinary category theory has a long
  history, see \cite{GambinoKock}.  It can be extended to general
  locally cartesian closed categories with a terminal object by
  considering local \emph{fibred} right adjoints instead of just
  local right adjoints, cf.~\cite{KockKock}. Presumably the fibred
  viewpoint can be upgraded to the $\infty$-categorical setting to get
  a version of Theorem~\ref{thm:pnchar} for presentable locally
  cartesian closed $\infty$-categories, but we will not pursue this
  here.
\end{remark}

Before proving Theorem~\ref{thm:pnchar},
we need some observations on weakly contractible limits.

\begin{defn}
  A \emph{conical limit} is a limit indexed by an \icat{}
  of the form
  $\mathcal{I}^{\triangleright}$ for some \icat{} $\mathcal{I}$.
\end{defn}

\begin{lemma}\label{lem:conical=wc}
  Suppose $\mathcal{C}$ has a terminal object. Then a functor $F
  \colon \mathcal{C} \to \mathcal{D}$ preserves conical limits \IFF{}
  it preserves all weakly contractible limits.
\end{lemma}
\begin{proof}
  Conical limits are in particular indexed by weakly contractible
  \icats{}, so suppose $F$ preserves conical limits and let $\phi \colon \mathcal{I} \to \mathcal{C}$ be a
  diagram with $\mathcal{I}$ weakly contractible. Since $\mathcal{C}$
  has a terminal object, the right Kan extension $\phi'$ of $\phi$
  along the inclusion $i \colon \mathcal{I} \hookrightarrow
  \mathcal{I}^{\triangleright}$ exists, and $\phi \simeq
  \phi'|_{\mathcal{I}}$. Moreover, if $\phi$ has a limit then so does
  $\phi'$ and the limit of $\phi$ is equivalent
  to that of $\phi'$. Since $\mathcal{I}^{\triangleright}$ is
  conical, $F$ preserves the limit of $\phi'$. But as $\mathcal{I}$ is
  weakly contractible, the inclusion $i$ is coinitial, hence the limit
  of $F \circ \phi$ exists and is equivalent to the limit of $F \circ
  \phi'$. In other words, $F$ preserves the limit of $\phi$, as required.
\end{proof}

\begin{lemma}\label{lem:slicecreates}
  For any object $x$ in an \icat{} $\mathcal{C}$, the forgetful
  functor $P\colon \mathcal{C}_{/x} \to \mathcal{C}$ preserves and reflects weakly contractible limits.
\end{lemma}
\begin{proof}
  The limit of a diagram $f \colon \mathcal{I} \to \mathcal{C}_{/x}$ is the
  limit of the corresponding diagram $f' \colon \mathcal{I}^{\triangleright} \to
  \mathcal{C}$. If $\mathcal{I}$ is weakly contractible, then the
  inclusion $\mathcal{I} \hookrightarrow \mathcal{I}^{\triangleright}$
  is coinitial, so the limit of $f'$ is the same as the limit of
  $f'|_{\mathcal{I}}$, which is the image of $f$ under the forgetful functor.
\end{proof}

\begin{lemma}\label{lem:igpdcolimwclim}
  In the \icat{} $\mathcal{S}$, weakly contractible limits commute
  with colimits indexed by $\infty$-groupoids.
\end{lemma}
\begin{proof}
  For $X \in \mathcal{S}$ we have the straightening equivalence
  $\Fun(X, \mathcal{S}) \simeq \mathcal{S}_{/X}$, under which the
  constant diagram functor $\mathcal{S} \to \Fun(X, \mathcal{S})$
  corresponds to taking products with $X$. Passing to left adjoints,
  this means that taking $X$-indexed colimits corresponds under the
  equivalence to the forgetful functor $\mathcal{S}_{/X}\to
  \mathcal{S}$, which preserves weakly contractible limits by
  Lemma~\ref{lem:slicecreates}.
\end{proof}

\begin{propn}\label{propn:lrachar}
  Suppose $F \colon \mathcal{C} \to \mathcal{D}$ is an accessible
  functor between presentable \icats{}. Then $F$ is a local right
  adjoint \IFF{} it preserves weakly contractible limits.
\end{propn}
\begin{proof}
  For every $x \in \mathcal{C}$, the induced functor
  $F_{/x} \colon \mathcal{C}_{/x} \to \mathcal{D}_{/F(x)}$ is
  accessible, so by the adjoint functor theorem it is a right adjoint
  \IFF{} it preserves limits. A limit of a diagram
  $\mathcal{I} \to \mathcal{C}_{/x}$ is the limit in $\mathcal{C}$ of
  the associated diagram
  $\mathcal{I}^{\triangleright} \to \mathcal{C}$, so the functors
  $F_{/x}$ preserve limits for all $x$ \IFF{} $F$ preserves all
  conical limits. By Lemma~\ref{lem:conical=wc} this is equivalent to
  $F$ preserving weakly contractible limits, since $\mathcal{C}$ has a
  terminal object.
\end{proof}

\begin{lemma}\label{lem:shriekconlim}
  For any map $f \colon S \to T$ in an \icat{} $\mathcal{C}$, the functor $f_{!}
  \colon \mathcal{C}_{/S}\to \mathcal{C}_{/T}$ preserves and reflects weakly
  contractible limits.
\end{lemma}
\begin{proof}
  In the commutative triangle
  \opctriangle{\mathcal{C}_{/S}}{\mathcal{C}_{/T}}{\mathcal{C},}{f_{!}}{}{}
  both forgetful functors to $\mathcal{C}$ preserve and reflect weakly contractible
  limits. Therefore so does $f_{!}$.
\end{proof}

\begin{lemma}\label{lem:leftadjspan}
  Suppose $F \colon \mathcal{S}_{/I} \to \mathcal{S}_{/J}$ is a
  functor that preserves colimits (equivalently, by the adjoint
  functor theorem, it is a left
  adjoint). Then $F$ is of the form $s_{!}p^{*}$ for some span
  \[ I \xfrom{p} U \xto{s} J.\]
\end{lemma}
\begin{proof}
  Using equivalences of the form $\mathcal{S}_{/I} \simeq \Fun(I,
  \mathcal{S})$ we get an equivalence
  \[ \Fun^{\mathrm{L}}(\mathcal{S}_{/I}, \mathcal{S}_{/J}) \simeq
  \Fun^{\mathrm{L}}(\mathcal{P}(I), \mathcal{P}(J)) \simeq \Fun(I,
  \Fun(J, \mathcal{S})) \simeq \Fun(I \times J, \mathcal{S}) \simeq
  \mathcal{S}_{/I \times J}.\]
  Thus every colimit-preserving functor corresponds to a span, and
  under this equivalence the span  \[ I \xfrom{p} U \xto{s} J\] is
  sent to $s_{!}p^{*}$. 
\end{proof}

\begin{proof}[Proof of Theorem~\ref{thm:pnchar}]
  The equivalence of (ii) and (iii) is a special case of
  Proposition~\ref{propn:lrachar}. To see that (i) implies (ii),
  suppose $F \simeq t_{!}p_{*}s^{*}$. The functors $t_{!}$, $p_{*}$
  and $s^{*}$ are all accessible, and $p_{*}$ and $s^{*}$ preserve all
  limits, being right adjoints; by Lemma~\ref{lem:shriekconlim} the
  functor $t_{!}$ also preserves weakly contractible limits, which gives (ii). 
  Finally we show that (iii) implies (i). Observe that $F$ factors as
  \[ \mathcal{S}_{/I} \xto{F_{/I}} \mathcal{S}_{/F(I)} \xto{t_{!}}
  \mathcal{S}_{/J} \]
  where $t$ is the map $F(I) \to J$. By assumption $F_{/I}$ is a right
  adjoint, so it follows from Lemma~\ref{lem:leftadjspan} that it is
  of the form $p_{*}s^{*}$ for some span
  \[ I \xfrom{s} U \xto{p} F(I).\qedhere\]
\end{proof}

\subsection{Morphisms of Polynomial Functors}\label{subsec:morpolyfun}
For our purposes the appropriate type of morphism between
polynomial functors is a cartesian natural
transformation, so we make the following definition:
\begin{defn}
  The \icat{} $\PolyFun(I,J)$ of polynomial functors is the
  subcategory of $\Fun(\mathcal{S}_{/I}, \mathcal{S}_{/J})$ with objects
  the polynomial functors and morphisms the cartesian natural transformations
  between them.
\end{defn}

We now wish to identify the cartesian natural transformations with
certain diagrams.
\begin{defn}\label{defn:carttransdiag}
  Suppose given a commutative diagram of spaces
  \[
    \begin{tikzcd}
      {} & E' \drpullback 
      \arrow{dd}{\epsilon} \arrow[swap]{dl}{s'} \arrow{r}{p'} & B' \arrow{dd}{\beta} \arrow{dr}{t'} \\
      I  & & {} & J \\
      & E \arrow{ul}{s} \arrow[swap]{r}{p} & B, \arrow[swap]{ur}{t} 
    \end{tikzcd}
  \]
  where the middle square is cartesian.
  Let $F' := t'_{!}p'_{*}s'^{*}$ and $F := t_{!}p_{*}s^{*}$ denote the
  corresponding polynomial functors. Then we define a cartesian
  natural transformation $\phi \colon F' \to F$ as the composite
  \[ t'_{!}p'_{*}s'^{*} \simeq t'_{!}p'_{*}\epsilon^{*}s^{*} \simeq
  t'_{!}\beta^{*}p_{*}s^{*} \simeq t_{!}\beta_{!}\beta^{*}p_{*}s^{*}
  \to t_{!}p_{*}s^{*}\]
  using the Beck-Chevalley transformation for the cartesian square and
  the counit for $\beta_{!} \dashv \beta^{*}$.  Observe that the
  component of $\phi$ on the terminal object $\id_I$ is essentially
  $\beta$ itself (as follows since right adjoints preserve terminal
  objects): $\phi_{\id_I} \colon F'(\id_I) \to F(\id_I)$
  is canonically identified with $\beta\colon B' \to B$ as maps over
  $J$.
\end{defn}

Our goal in this subsection is to show that in the $\infty$-category of 
spaces, this construction gives an
equivalence between the space of such diagrams and the mapping space
$\Map_{\txt{PolyFun}(I,J)}(F', F)$. (Later, we will show that this
extends to an equivalence of \icats{}.) We first prove that every
cartesian natural transformation is of this form, 
which is a consequence of the following observation: 
\begin{lemma}\label{lem:carttrcartfib}
  Suppose $\mathcal{C}$ is an \icat{} with a terminal object $*$ and
  $\mathcal{D}$ is an \icat{} with pullbacks. Then the functor
  \[ \txt{ev}_{*} \colon \Fun(\mathcal{C},\mathcal{D}) \to
  \mathcal{D} \]
  is a cartesian fibration, and the cartesian morphisms are precisely
  the cartesian natural transformations.
\end{lemma}
\begin{proof}
  Observe that $\txt{ev}_{*}$ has a right adjoint $r$, taking
  $d \in \mathcal{D}$ to the constant functor $r(d) \colon \mathcal{C}
  \to \mathcal{D}$ with value $d$.  Let $\eta$ denote the unit and 
  $\epsilon$ the counit.  
  We can apply the criterion of (the dual
  of) \cite{nmorita}*{Corollary 4.52}: $\txt{ev}_{*}$ is cartesian
  \IFF{} for every functor $F \in \Fun(\mathcal{C}, \mathcal{D})$
  and every morphism $d \to F(*)$ in $\mathcal{D}$, in the pullback square
  \[
  \begin{tikzcd}
	  F' \drpullback \arrow{d} \arrow{r} & F \arrow{d}{\eta_F} \\
	  r(d) \arrow{r} & r(F(*)) ,
  \end{tikzcd}
  \]
  the composite
  \[ F'(*) \longrightarrow r(d)(*) 
  \stackrel{\epsilon_d}\longrightarrow d \]
  is an equivalence.  But this is obvious since pullbacks in
  $\Fun(\mathcal{C}, \mathcal{D})$ are computed objectwise, so we have a
  pullback square
  \[
  \begin{tikzcd}
	  F'(*) \drpullback \arrow{d} \arrow{r} & F(*) \arrow{d} \\
	  d \arrow{r} & F(*).
  \end{tikzcd}
  \]
  Moreover, by \cite{nmorita}*{Proposition 4.51} a morphism
  $\phi \colon F \to G$ in $\Fun(\mathcal{C}, \mathcal{D})$ is
  $\txt{ev}_{*}$-cartesian \IFF{} the commutative square
  \csquare{F}{G}{rF(*)}{rG(*)}{\phi}{\eta_F}{\eta_G}{r(\phi_{*})} is
  cartesian, i.e.~\IFF{} for every $x \in \mathcal{C}$ the square
  \csquare{F(x)}{G(x)}{F(*)}{G(*)}{\phi_{x}}{}{}{\phi_{*}} is
  cartesian, which is equivalent to $\phi$ being cartesian by
  Remark~\ref{rmk:cart*}.
\end{proof}

\begin{defn}
  Let $\Fun(\mathcal{C}, \mathcal{D})^{\txt{cart}}$ denote the
  subcategory of $\Fun(\mathcal{C}, \mathcal{D})$ containing only the
  cartesian natural transformations. 
\end{defn}

\begin{remark}
  By Lemma~\ref{lem:carttrcartfib} the \icat{} $\Fun(\mathcal{C},
  \mathcal{D})^{\txt{cart}}$ is precisely the subcategory of
  $\Fun(\mathcal{C}, \mathcal{D})$ containing only the cartesian
  morphisms for the cartesian fibration $\txt{ev}_{*}$. The
  restriction of this to functor $\txt{ev}_{*} \colon
  \Fun(\mathcal{C}, \mathcal{D})^{\txt{cart}} \to \mathcal{D}$ is hence a
  right fibration.
\end{remark}

\begin{lemma}\label{lem:polypoly-ess}
  Suppose $F \colon \mathcal{S}_{/I} \to \mathcal{S}_{/J}$ is a
  polynomial functor, represented by a diagram
  \[ I \xfrom{s} E \xto{p} B \xto{t} J.\]
  If $\phi \colon F' \to F$ is a cartesian natural transformation,
  then $F'$ is also a polynomial functor, and $\phi$ is equivalent to
  the natural transformation associated to a diagram
  \[
  \begin{tikzcd}
    {} & E' \drpullback
	\arrow{dd}{\epsilon} \arrow[swap]{dl}{s'} \arrow{r}{p'} & B' 
	\arrow{dd}{\beta} \arrow{dr}{t'} \\
    I  & & {} & J   . \\
     & E \arrow{ul}{s} \arrow[swap]{r}{p} & B \arrow[swap]{ur}{t} 
  \end{tikzcd}
  \]
\end{lemma}
\begin{proof}
  Let $(t' \colon B' \to J) := F'(\id_{I})$, then the map $\phi_{\id_{I}}$ gives
 a map $\beta \colon B' \to B$ over $J$. We can then define $\epsilon$ as the
 pullback of $\beta$ along $p$ and put $s' := s \circ \epsilon$ to get a
 diagram of this form. The construction of
 Definition~\ref{defn:carttransdiag} then gives a natural
 transformation $\phi' \colon F'' \to F$. Thus we have two cartesian
 natural transformations to $F$ with the same image in
 $\mathcal{S}_{/J}$ under evaluation at the terminal object. Thus by
 Lemma~\ref{lem:carttrcartfib}, $\phi$ and $\phi'$
 are both cartesian morphisms to $F$ with the same image in
 $\mathcal{S}_{/J}$, and so they must be equivalent --- in particular
 $F' \simeq F''$, which implies that $F'$ is indeed polynomial.
\end{proof}

As a consequence, we get:
\begin{lemma}
  The projection $\txt{ev}_{\id_{I}} \colon \txt{PolyFun}(I,J) \to
  \mathcal{S}_{/J}$ is a right fibration.
  \qed
\end{lemma}
It remains to understand the fibres of this fibration.  Note that, for
$\mathcal{C}$ and $\mathcal{D}$ \icats{} where $\mathcal{C}$ has a terminal
object $*$, the fibre of $\txt{ev}_{*} \colon \Fun(\mathcal{C},
\mathcal{D}) \to \mathcal{D}$ at $d \in \mathcal{D}$ is the \icat{} $\Fun_{*}(\mathcal{C},
\mathcal{D}_{/d})$ of functors that preserve the terminal object.
Restricting to polynomial functors from $\mathcal{S}_{/I}$ to
$\mathcal{S}_{/J}$ (with all natural transformations allowed), the fibre at
$t \colon B \to J$ is $\Fun^{\mathrm{R}}(\mathcal{S}_{/I},
\mathcal{S}_{/B})$, since polynomial functors are local right adjoints by
Theorem~\ref{thm:pnchar}.  This fibre can be described explicitly:

\begin{lemma}\label{lem:SIB}
  There is a natural equivalence
  \[
  (\mathcal{S}_{/I \times B})^{\op} \isoto \Fun^{\mathrm{R}}(\mathcal{S}_{/I},
  \mathcal{S}_{/B}) ,
  \]
  which sends a span $I \xfrom{s} X \xto{p} B$ to the functor
  $p_{*}s^{*}$, and sends a map of spans
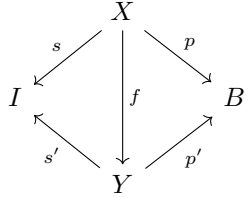
\[
\begin{tikzcd}
  {} & X \arrow[swap]{dl}{s} \arrow{dr}{p}  \arrow{dd}{f}\\
  I  & & B \\
  & Y \arrow{ul}{s'} \arrow[swap]{ur}{p'}
\end{tikzcd}
\]
to the natural transformation
\[ p'_{*}s'^{*} \to p'_{*}f_{*}f^{*}s'^{*} \simeq p_{*}s^{*} \]
induced by the unit of the adjunction $f^{*} \dashv f_{*}$. 
\end{lemma}
\begin{proof}
  This is a reformulation of Lemma~\ref{lem:leftadjspan}, using the
  natural equivalence \[\Fun^{\mathrm{R}}(\mathcal{S}_{/I}, \mathcal{S}_{/B})
  \simeq \Fun^{\mathrm{L}}(\mathcal{S}_{/B}, \mathcal{S}_{/I})^\op.\qedhere\]
\end{proof}

\begin{remark}
  Restricting to cartesian natural transformations, we see that the
  fibre of
  \[\txt{ev}_{\id_{I}} \colon \txt{PolyFun}(I,J) \to \mathcal{S}_{/J}\]
  at $B \to J$ is equivalent to the core $\infty$-groupoid of
  $\mathcal{S}_{/I \times B}$.
\end{remark}

\subsection{The $\infty$-Category of Polynomial
  Functors}\label{subsec:icatpoly}
We now wish to construct an \icat{} of all polynomial functors, i.e.\
to put the \icats{} $\PolyFun(I,J)$ for varying $I$ and $J$ together
into a single \icat{} $\PolyFun$, fibred over
$\mathcal{S}\times\mathcal{S}$ (by returning $I$ and $J$). We will
define this using a double \icat{} of colax squares of $\infty$-categories,
constructed in \S\ref{subsec:SqL}. We then define a functor from polynomials
to polynomial functors with varying source and target, and prove that
this is an equivalence.

\begin{defn}
  In \S\ref{subsec:SqL} we define a double \icat{}
  $\SqCL(\LCATI)^{\txt{v=radj}}$ where
  \begin{itemize}
  \item the objects are \icats{},
  \item the vertical morphisms are right adjoints,
  \item the horizontal morphisms are arbitrary functors,
  \item the squares (or 2-cells) are colax squares, i.e.~diagrams in
    the 
  $(\infty,2)$-category of \icats{} of shape 
  
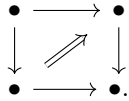
\[
  \begin{tikzcd}
    \bullet \arrow{r} \arrow{d} & \bullet \arrow{d}
    \arrow[Leftarrow,shorten >= 6pt,shorten <= 6pt]{dl} \\
    \bullet \arrow{r} & \bullet.
  \end{tikzcd}
  \]
\end{itemize}
We can pull $\SqCL(\LCATI)^{\txt{v=radj,v-op}}$ back along the
  functor
  $\mathcal{S}_{/\blank}^{*} \colon \mathcal{S}^{\op} \to
  \LCatI^{\radj}$
  taking a space $I$ to $\mathcal{S}_{/I}$ and a map
  $f \colon I \to J$ to
  $f^{*} \colon \mathcal{S}_{/J} \to \mathcal{S}_{/I}$, to obtain a
  double \icat{} where
  \begin{itemize}
  \item the objects are spaces
  \item the vertical morphisms are maps of spaces,
  \item the horizontal morphisms are arbitrary functors between slices,
  \item the squares (or 2-cells) are colax squares using the
    $(\blank)^{*}$-functors for maps of spaces.
  \end{itemize}
  We define the double \icat{} $\txt{POLYFUN}$ to be the 
  sub-double \icat{} of this pullback where the horizontal morphisms
  are polynomial functors. Thus $\txt{POLYFUN}$ has
  \begin{enumerate}[(1)]
  \item spaces as objects,
  \item maps of spaces as vertical morphisms,
  \item polynomial functors as horizontal morphisms,
  \item diagrams of the form
  \[
  \begin{tikzcd}
    \mathcal{S}_{/I} \arrow{r}{P} \arrow[Rightarrow,shorten >= 6pt,shorten <= 6pt]{dr}& \mathcal{S}_{/J} 
     \\
    \mathcal{S}_{/I'} \arrow[swap]{r}{Q} \arrow{u}{f^{*}} & \mathcal{S}_{/J'} 
    \arrow[swap]{u}{g^{*}}
  \end{tikzcd}
  \]
  as squares.
  \end{enumerate}
\end{defn}

\begin{remark}
  Taking mates in the vertical direction should give an equivalence of
  double \icats{}
  \[ \SqCL(\LCATI)^{\txt{v=radj,v-op}} \isoto
  \SqL(\LCATI)^{\txt{v=ladj}}.\]
  Assuming this, our definition of $\txt{POLYFUN}$ is equivalent to
  the alternative, and perhaps more standard,
  definition where the
  vertical morphisms are the left adjoint functors
  $f_{!} \colon \mathcal{S}_{/I} \to \mathcal{S}_{/J}$. We have chosen
  our convention to match with the correct convention for polynomial
  monads, where we really do want lax transformations
  (cf.~\cref{rmk:mndcolaxsql}) with direction reversed
  (which are not the same as colax morphisms of monads) ---
  we thereby avoid unnecessarily using the above-mentioned equivalence of
  lax and colax squares via mates, which we do not prove here.
\end{remark}

\begin{propn}
  The double \icat{} $\txt{POLYFUN}$ is framed, in the sense of
  Definition~\ref{defn:framed}.
\end{propn}
\begin{proof}
  In Proposition~\ref{propn:SqFramed} we prove that the double \icat{}
  $\SqCL(\LCATI)^{\txt{v=radj}}$ is framed. But for a vertical morphism of
  the form $f^{*}$ the four squares of the framing live in the
  sub-double \icat{} $\txt{POLYFUN}$ since the unit and counit
  transformations for $f_{!}\dashv f^{*}$ are cartesian. Thus
  $\txt{POLYFUN}$ is also framed.
\end{proof}

\begin{defn}
  We write $\PolyFun$ for the \icat{} of horizontal morphisms in
  $\txt{POLYFUN}$, i.e. $\txt{POLYFUN}_{[1]}$ if we view
  $\txt{POLYFUN}$ as a cocartesian fibration over $\Dop$.
\end{defn}

Applying Proposition~\ref{propn:equipmt}, we get:
\begin{cor}\label{source-and-target}
  The source-and-target projection
  $\PolyFun \to \mathcal{S} \times \mathcal{S}$ is both cartesian and
  cocartesian. \qed
\end{cor}

\begin{remark}\label{rmk:cartpoly}
  For morphisms $f \colon I \to I'$, $g \colon J \to J'$, the
  cocartesian pushforward of $F \in \PolyFun(I,J)$ along $(f,g)$
  is the composite $g_{!}Ff^{*}$, while the cartesian pullback of $G
  \in \PolyFun(I',J')$ is $g^{*}Ff_{!}$. Note that if $F$ corresponds
  to the diagram
  \[ I \xfrom{s} E \xto{p} B \xto{t} J,\]
  then the pushforward $g_{!}Ff^{*}$ corresponds to 
  \[ I' \xfrom{fs} E \xto{p} B \xto{gt} J',\]
  while the cartesian pullback is more complicated to describe
  diagrammatically.
\end{remark}

\begin{defn}\label{defn:Poly}
  Let $\Pi$ denote the category
  $\bullet \from \bullet \to \bullet \to \bullet$, which decomposes as
  a colimit
  $(\Delta^{1})^{\op} \amalg_{\Delta^{0}} \Delta^{1} \amalg_{\Delta^{0}}
  \Delta^{1}$ in $\CatI$.
  We define $\Poly$ to be the subcategory of
  $\Fun(\Pi, \mathcal{S})$ containing only those morphisms where the
  middle commuting square is cartesian. In other words, we have a
  natural equivalence
 \[\Poly \simeq \Fun((\Delta^{1})^{\op}, \mathcal{S})
  \times_{\mathcal{S}} \Fun(\Delta^{1}, \mathcal{S})^{\txt{cart}}
  \times_{\mathcal{S}} \Fun(\Delta^{1}, \mathcal{S}).\]
\end{defn}

We now define a functor $\Phi \colon \Poly \to \PolyFun$; we do this
by defining three functors to $\PolyFun$ and then combining them using
the horizontal composition in $\txt{POLYFUN}$.
\begin{defn}
  Let $\mathcal{S}_{/\blank}^{*}$ denote the functor
  $\mathcal{S}^{\op} \to \LCatI$ taking $I \in \mathcal{S}$ to
  $\mathcal{S}_{/I}$ and $f \colon I \to J$ to $f^{*}\colon
  \mathcal{S}_{/J} \to \mathcal{S}_{/I}$. This induces a functor
  \[ \Phi_{1} \colon \Fun((\Delta^{1})^{\op}, \mathcal{S}) \simeq
  \Map(\Delta^{\bullet} \times (\Delta^{1})^{\op}, \mathcal{S}) \to
  \Sq_{\bullet,1}(\LCatI)^{\txt{v=radj},\txt{v-op}} \to
  \SqCL(\LCATI)^{\txt{v=radj,v-op}},\]
  which clearly passes through $\PolyFun$.

  Combining $\mathcal{S}_{/\blank}^{*}$ instead with the functor
  \[ \Sq(\CatI)^{\txt{h=ladj}} \to
  \SqL(\CATI)^{\txt{h=radj,h-op}}\]
  of Proposition~\ref{propn:mateftr}, which takes mates in the
  horizontal direction (replacing $f^{*}$ with $f_{*}$), we get a functor
  \[ \Fun(\Delta^{1}, \mathcal{S}) \to \SqL(\CATI)^{\txt{h=radj}}_{1}.\]
  Restricting to $\Fun(\Delta^{1}, \mathcal{S})^{\txt{cart}}$, this
  actually lands in commuting squares, giving
  \[ \Phi_{2} \colon \Fun(\Delta^{1}, \mathcal{S})^{\txt{cart}} \to \Sq(\CATI)^{\txt{h=radj}}_{1},\]
  which factors through $\PolyFun$.

  Finally, combining $\mathcal{S}_{/\blank}^{*}$ with the functor 
  \[ \Sq(\CatI)^{\txt{h=radj}} \to
  \SqCL(\CATI)^{\txt{h=ladj,h-op}}, \]
  of Proposition~\ref{propn:mateftr}, which also takes mates in the
  horizontal direction (replacing $f^{*}$ with $f_{!}$), we get a
  functor
  \[ \Phi_{3} \colon \Fun(\Delta^{1}, \mathcal{S}) \to
  \SqCL(\CATI)^{\txt{h=ladj}}_{1}.\]
  This again factors through $\PolyFun$.
\end{defn}

\begin{remark}
More explicitly, $\Phi_{1}$ is given by 
\[
\begin{tikzcd}
  J  \arrow[swap]{d}{g} & I \arrow[swap]{l}{a} \arrow{d}{f} \\
  L  & \arrow{l}{b} K
\end{tikzcd}
\quad \mapsto \quad
\begin{tikzcd}
  \mathcal{S}_{/J} \arrow{r}{a^{*}}  \arrow[Rightarrow,shorten >= 8pt,shorten <= 8pt]{dr}{\sim}&
  \mathcal{S}_{/I}   \\
  \mathcal{S}_{/L} \arrow[swap]{r}{b^{*}} \arrow{u}{g^{*}} &
  \mathcal{S}_{/K}. \arrow{u}[swap]{f^{*}}
\end{tikzcd}
\]
Similarly, the functor $\Phi_{2}$ is given by
\[
\begin{tikzcd}
  I \drpullback \arrow{r}{a} \arrow[swap]{d}{f} & J \arrow{d}{g} \\
  K \arrow[swap]{r}{b} & L
\end{tikzcd}
\quad \mapsto \quad
\begin{tikzcd}
  \mathcal{S}_{/I} \arrow{r}{a_{*}}  &
  \mathcal{S}_{/J} \\
  \mathcal{S}_{/K} \arrow{u}{f^{*}}\arrow[swap]{r}{b_{*}} &
  \mathcal{S}_{/L}. \arrow{u}[swap]{g^{*}} \arrow[Rightarrow,shorten >= 8pt,shorten <= 8pt]{ul}[swap]{\sim}
\end{tikzcd}
\]
Finally, $\Phi_{3}$ is given by
\[
\begin{tikzcd}
  I \arrow{r}{a} \arrow[swap]{d}{f} & J \arrow{d}{g} \\
  K \arrow[swap]{r}{b} & L
\end{tikzcd}
\quad \mapsto \quad
\begin{tikzcd}
  \mathcal{S}_{/I} \arrow{r}{a_{!}} \arrow[Rightarrow,shorten >= 8pt,shorten <= 8pt]{dr} &
  \mathcal{S}_{/J}  \\
  \mathcal{S}_{/K} \arrow[swap]{r}{b_{!}}  \arrow{u}{f^{*}}&
  \mathcal{S}_{/L}. \arrow{u}[swap]{g^{*}}
\end{tikzcd}
\]
\end{remark}

\begin{defn}\label{defn:PolyPOLYFUN}
  The functors $\Phi_{i}$ agree appropriately under restriction to
  $\mathcal{S}$ to determine a functor
  \[ \Poly \isoto \Fun((\Delta^{1})^{\op},
  \mathcal{S})\times_{\mathcal{S}} \Fun(\Delta^{1}, \mathcal{S})^{\txt{cart}}
  \times_{\mathcal{S}} \Fun(\Delta^{1}, \mathcal{S}) \to \PolyFun
  \times_{\mathcal{S}} \PolyFun \times_{\mathcal{S}}
  \PolyFun.\]
  Combining this with the horizontal composition in $\txt{POLYFUN}$,
  \[ \PolyFun
  \times_{\mathcal{S}} \PolyFun \times_{\mathcal{S}}
  \PolyFun \isofrom
  \txt{POLYFUN}_{\bullet,3} \to \PolyFun,\]
  we get the required functor $\Phi \colon \Poly \to
  \PolyFun$.  
\end{defn}

\begin{thm}\label{thm:polyispolyfun}
  The functor $\Phi \colon \Poly \to \PolyFun$ is an equivalence.
\end{thm}

The following lemma shows that it is enough to prove this fibrewise:
\begin{lemma}\label{lem:ev03}
  The projection $\txt{ev}_{0,3} \colon \Poly \to \mathcal{S}\times
  \mathcal{S}$ is a cocartesian fibration, and $\Phi \colon \Poly \to 
  \PolyFun$ preserves cocartesian morphisms.
\end{lemma}
\begin{proof}
  For a polynomial $P$ given by $I \stackrel{s}\leftarrow E 
  \stackrel{p}\to B \stackrel{t}\to J$ and a 
  morphism $(f \colon I \to I', g\colon J\to J')$ out of $\txt{ev}_{0,3}(P)$, a 
  cocartesian lift is given by
  \[
  \begin{tikzcd}
  I \arrow[swap]{d}{f}& E \arrow{l}[above]{s} \arrow{r}{p} 
  \arrow[equal]{d} 
  \drpullback & B \arrow{r}{t} \arrow[equal]{d}& J\arrow{d}{g} \\
  I' & E \arrow{l}{fs} \arrow[swap]{r}{p} & B \arrow[swap]{r}{gt} & 
  J'  ;
\end{tikzcd}
\]
the cocartesian property is readily checked.
  Such a diagram is sent by $\Phi$ to the 
  colax square 
  \[
  \begin{tikzcd}
    {\mathcal{S}_{/I}} \arrow{r}{P} \arrow[Rightarrow,shorten >= 8pt,shorten <= 8pt]{dr} & 
	{\mathcal{S}_{/J}} 
      \\
    {\mathcal{S}_{/I'}} \arrow{u}{f^{*}} \arrow{r}[below]{g_{!} P f^{*}} & 
	{\mathcal{S}_{/J'}} \arrow{u}[swap]{g^{*}}
  \end{tikzcd}
  \]
  with the natural transformation given by the unit transformation
  $Pf^{*} \to g^{*}g_{!}Pf^{*}$.  But this is precisely the form of 
  the cocartesian edges in $\PolyFun$, as noted in Remark~\ref{rmk:cartpoly}.
\end{proof}

\begin{propn}\label{propn:Phifib}
  For fixed spaces $I$ and $J$, the functor $\Phi$ gives an equivalence
  \[
  \Poly(I,J) \to \PolyFun(I,J)
  \]
  when restricted to the fibre over $(I,J)$.
\end{propn}
\begin{proof}
  Both sides are right fibrations over $\mathcal{S}_{/J}$, so it
  suffices to show that we get an equivalence on fibres over every
  $B \to J$ in $\mathcal{S}_{/J}$. The fibre of $\Poly(I,J)$ is the
  $\infty$-groupoid of spans $I \leftarrow M \to B$, and the fibre of
  $\PolyFun(I,J)$ is the $\infty$-groupoid
  $\Map^{\mathrm{R}}(\mathcal{S}_{/I},\mathcal{S}_{/B})$. The functor
  restricts precisely to the functor in Lemma~\ref{lem:SIB}, shown
  there to be an equivalence.
\end{proof}

\begin{proof}[Proof of Theorem~\ref{thm:polyispolyfun}]
  By Lemma~\ref{lem:ev03} the functor $\Phi$ is a map between
  cocartesian fibrations and preserves cocartesian edges. It therefore
  suffices to show that the induced map on fibres
  $\Poly(I,J) \to \PolyFun(I,J)$ is an equivalence, which is
  Proposition~\ref{propn:Phifib}.
\end{proof}

\subsection{Colimits of Polynomial Functors}\label{sec:colimpoly}
In this subsection we will give two descriptions of colimits of
polynomial functors: First, we will see that colimits in $\Poly$ can
be computed in $\Fun(\Pi, \mathcal{S})$, i.e.\ pointwise in the
diagram. We will also show that colimits in $\PolyFun(I,J)$ can be
computed in $\Fun(\mathcal{S}_{/I}, \mathcal{S}_{/J})$.

\begin{propn}\label{propn:Cartcolim}
  Let $\mathcal{C}$ be a small \icat{} and let $\mathcal{X}$ be an
  $\infty$-topos. The forgetful functor
  $\Fun(\mathcal{C}, \mathcal{X})^{\txt{cart}} \to \Fun(\mathcal{C},
  \mathcal{X})$ preserves and reflects
  all limits and colimits.
\end{propn}
\begin{proof}
  We first consider the case of colimits. Given a diagram
  $\phi \colon \mathcal{I} \to \Fun(\mathcal{C},
  \mathcal{X})^{\txt{cart}}$,
  let
  $\bar{\phi} \colon \mathcal{I}^{\triangleright} \to
  \Fun(\mathcal{C}, \mathcal{X})$
  be a colimit diagram extending the image of $\phi$ in
  $\Fun(\mathcal{C}, \mathcal{X})$. We claim that this colimiting
  cocone in $\Fun(\mathcal{C}, \mathcal{X})$ is also a colimiting
  cocone in the subcategory
  $\Fun(\mathcal{C}, \mathcal{X})^{\txt{cart}}$.

  To show this we must first prove that the commutative squares
  \nolabelcsquare{\phi(i)(c)}{\phi(i)(c')}{\bar{\phi}(\infty)(c)}{\bar{\phi}(\infty)(c')}
  are cartesian, for all maps $c \to c'$ in $\mathcal{C}$. Since
  colimits in functor \icats{} are computed objectwise, this is true
  by descent for the $\infty$-topos $\mathcal{X}$, using
  \cite{HTT}*{Theorem 6.1.3.9(4)}.
  
  Second, we must check that for any cocone
  $\phi' \colon \mathcal{I}^{\triangleright} \to \Fun(\mathcal{C},
  \mathcal{X})^{\txt{cart}}$,
  the canonical map $\bar{\phi} \to \phi'$ in
  $\Fun(\mathcal{C}, \mathcal{X})$ actually belongs to
  $\Fun(\mathcal{C}, \mathcal{X})^{\txt{cart}}$, i.e.\ it is a
  cartesian transformation. Since the transformations $\phi(i)
  \to \phi'(\infty)$ are cartesian, we have pullback squares
  \[
  \begin{tikzcd}[row sep=2.6em,column sep=2.2em]
  \phi(i)(c) \ar[d] \ar[r] \drpullback & \phi(i)(c') \ar[d]  \\
  \phi'(\infty)(c) \ar[r] & \phi'(\infty)(c').
  \end{tikzcd}
  \]
  Colimits in $\mathcal{X}$
  are universal, so this induces a pullback square of colimits
  \[
  \begin{tikzcd}[row sep=2.6em,column sep=2.2em]
  \bar{\phi}(\infty)(c) \ar[d] \ar[r] \drpullback & \bar{\phi}(\infty)(c') \ar[d]  \\
  \phi'(\infty)(c) \ar[r] & \phi'(\infty)(c').
  \end{tikzcd}
  \]
  as required.

  The proof for limits is the same, but simpler, using that limits
  commute and pullbacks preserve limits, which is true in any \icat{}.
\end{proof}

\begin{cor}\label{cor:polycolim}
  Colimits in $\Poly$ are constructed in $\Fun(\Pi, \mathcal{S})$, and
  colimits in $\Poly(I,J)$ are constructed in
  $\Fun(\Delta^{1},\mathcal{S})$ for all spaces $I,J$. In particular,
  the \icats{} $\Poly$ and $\Poly(I,J)$ are cocomplete.
\end{cor}
\begin{proof}
  By definition, the \icat{} $\Poly$ is the fibre product
  $\Fun(\Delta^{1,\op}, \mathcal{S}) \times_{\mathcal{S}}
  \Fun(\Delta^{1}, \mathcal{S})^{\txt{cart}} \times_{\mathcal{S}}
  \Fun(\Delta^{1}, \mathcal{S})$. The projections to $\mathcal{S}$ all
  preserve colimits (using Proposition~\ref{propn:Cartcolim} for the
  middle term), so by \cite{HTT}*{Lemma 5.4.5.5} a diagram in $\Poly$
  is a colimit \IFF{} its composition with the projections to the
  terms in this fibre product are colimits. Now
  Proposition~\ref{propn:Cartcolim} implies that colimits are computed
  in 
  \[ \Fun(\Delta^{1,\op}, \mathcal{S}) \times_{\mathcal{S}}
  \Fun(\Delta^{1}, \mathcal{S}) \times_{\mathcal{S}}
  \Fun(\Delta^{1}, \mathcal{S}) \simeq \Fun(\Pi, \mathcal{S}).\]

  By the same argument, a diagram in \[\Poly(I,J) \simeq
  \mathcal{S}_{/I} \times_{\mathcal{S}}
  \Fun(\Delta^{1},\mathcal{S})^{\txt{cart}} \times_{\mathcal{S}}
  \mathcal{S}_{/J}\] 
  is a colimit \IFF{} its images in $\mathcal{S}_{/I}$,
  $\Fun(\Delta^{1},\mathcal{S})$, and $\mathcal{S}_{/J}$ are
  colimits. But colimits in these over-categories are computed in
  $\mathcal{S}$, so a diagram in $\Poly(I,J)$ is a colimit \IFF{} its
  image in $\Fun(\Delta^{1}, \mathcal{S})$ is a colimit.

  Since $\Fun(\Pi, \mathcal{S})$ and $\Fun(\Delta^{1}, \mathcal{S})$
  are cocomplete, it follows that so are the \icats{} $\Poly$ and $\Poly(I,J)$.
\end{proof}

\begin{remark}
  Note that the corresponding result does not hold in the classical
  $1$-categorical setting of \cite{GambinoKock} (such as in $\Set$),
  since a $1$-topos does not have descent in general.  In the
  $1$-categorical setting, only colimits of diagrams of monomorphisms
  can be computed pointwise, as exemplified by grafting of
  trees~\cite{KockTree}, as will be important below (cf.~Remark~\ref{rmk:1trees}).
\end{remark}

\begin{propn}\label{propn:colimpolyispoly}
  The forgetful functor
  $\PolyFun(I,J) \to \Fun(\mathcal{S}_{/I}, \mathcal{S}_{/J})$
  preserves colimits. In particular, the colimit of a diagram of
  polynomial functors and cartesian transformations is again a
  polynomial functor.
\end{propn}
\begin{proof}
  Consider a diagram $\phi \colon \mathcal{I} \to \PolyFun(I,J)$, where
  the functor $\phi_{x}$ corresponds to the diagram
  \[ I \xfrom{s_{x}} E(x) \xto{p_{x}} B(x) \xto{t_{x}} J.\]
  By Corollary~\ref{cor:polycolim} and 
  Theorem~\ref{thm:polyispolyfun} the colimit of $\phi$ in
  $\PolyFun(I,J)$ corresponds to the diagram
  \[ I \xfrom{s} E \xto{p} B \xto{t} J,\]
  where $E := \colim_{x \in \mathcal{I}} E(x)$ and
  $B := \colim_{x \in \mathcal{I}} B(x)$. On the other hand, since
  colimits in functor \icats{} are computed pointwise, the colimit of
  the diagram in $\Fun(\mathcal{S}_{/I}, \mathcal{S}_{/J})$ is the
  functor
  \[ \phi \colon (K \to I) \mapsto \colim_{x \in \mathcal{I}} \phi_{x}(K \to I).\]
  Let us view $\phi_{x}$ as a functor $\Fun(I, \mathcal{S}) \to
  \Fun(J, \mathcal{S})$; then evaluating at $\alpha \colon I \to
  \mathcal{S}$ and $j \in J$ we have
  \[ \phi_{x}(\alpha)(j) \simeq \colim_{b \in B(x)_{j}} \lim_{e \in
    E(x)_{b}} \alpha(s_{x}e).\]
  Let $\mathcal{B} \to \mathcal{I}$ be the left fibration
  corresponding to the functor $B(\blank)$. This has a map to $J$, and
  the fibre $\mathcal{B}_{j}\to \mathcal{I}$ is the left fibration for
  the functor $B(\blank)_{j}$. Since iterated colimits are colimits
  over cocartesian fibrations, we get
  \[ \phi(\alpha)(j) \simeq \colim_{(b,x) \in \mathcal{B}_{j}} \lim_{e
    \in E(x)_{b}} \alpha(s_{x}e).\]
  Now we observe that the functor $(b,x) \mapsto \lim_{e
    \in E(x)_{b}} \alpha(s_{x}e)$ takes every morphism in
  $\mathcal{B}_{j}$ to an equivalence of spaces: Since
  $\mathcal{B}_{j} \to \mathcal{I}$ is a left fibration it suffices to
  consider morphisms of the form $(b,x) \to (B(f)b, x')$ over $f
  \colon x \to x'$ in $\mathcal{I}$. Then as $\phi(f)$ is a cartesian
  natural transformation the map $E(f)_{b} \colon E(x)_{b} \to
  E(x')_{b}$ is an equivalence and we have
  \[ \lim_{e \in E(x)_{b}} \alpha(s_{x}e) \simeq \lim_{e
    \in E(x)_{b}} \alpha(s_{x'}E(f)e) \simeq \lim_{e' \in E(x')_{b}}
  \alpha(s_{x'}e').\]
  Thus this functor from $\mathcal{B}_{j}$ factors through the space
  obtained by inverting all morphisms in $\mathcal{B}_{j}$. This space
  is precisely $B_{j}\simeq \colim_{x \in \mathcal{I}} B(x)_{j}$ by 
  \cite{HTT}*{Corollary 3.3.4.6}. Since $\mathcal{B}_{j} \to B_{j}$ is
  cofinal by   \cite{HTT}*{Corollary 4.1.2.6}, this means we can replace the
  colimit over $\mathcal{B}_{j}$ by a colimit over the space $B_{j}$.  
  Moreover, since we have pullbacks
  \[
  \begin{tikzcd}
  E(x) \drpullback \ar[d, "p_x"'] \ar[r, "\epsilon_{x}"] & E \ar[d, "p"]  \\
  B(x) \ar[r, "\beta_{x}"'] & B
  \end{tikzcd}
  \]
  by \cite{HTT}*{Theorem 6.1.3.9}, for $b \in B(x)_{j}$ we can
  identify $\lim_{e \in E(x)_{b}} \alpha(s_{x}e)$ with $\lim_{e \in
    E_{b}} \alpha(se)$. Thus we have produced a natural equivalence
  \[ \phi(x) \simeq \colim_{b \in B_{j}} \lim_{e \in E_{b}}
  \alpha(se),\]
  where the right-hand side is the formula for the polynomial functor
  corresponding to the diagram
  \[ I \xfrom{s} E \xto{p} B \xto{t} J,\]
  as required.
\end{proof}

\subsection{Slices over Polynomial Functors}\label{subsec:slicepoly}
In this subsection we consider slices of $\PolyFun$, i.e.\
overcategories $\PolyFun_{/P}$. We will show that these \icats{} are
very well-behaved; specifically, we will prove:
\begin{thm}\label{thm:PolyFun/P=topos}
  For any polynomial functor $P$, the slice $\infty$-category
  $\PolyFun_{/P}$ is an $\infty$-topos; in particular, this \icat{} is
  presentable.  Furthermore, the full inclusion
  \[\PolyFun_{/P} \simeq \Poly_{/P} \to \Fun(\Pi, \mathcal{S})_{/P}\]
  preserves all limits and colimits; it is thus the
  inverse-image part of a geometric morphism.
\end{thm}

\begin{remark}
  This theorem is also true in the $1$-categorical case of $\Set$,
  although we are not aware of a reference.  This is a consequence of
  the observation that the maps in $\Poly$ form a class of standard
  \'etale maps in $\mathcal{P}(\Pi)$, in the axiomatic sense of
  Joyal--Moerdijk~\cite{JoyalMoerdijkOpen}.  The result now follows
  from their Corollary 2.3.
\end{remark}

\begin{lemma}\label{lem:cartslice}
  For any morphism $p \colon E \to B$ in $\mathcal{S}$, the
  functor
  \[\txt{ev}_{1} \colon \Fun(\Delta^{1},
    \mathcal{S})^{\txt{cart}}_{/p} \to \mathcal{S}_{/B}\] is an
  equivalence.
\end{lemma}
\begin{proof}
  The 2-of-3 property for pullback squares implies that
  $\Fun(\Delta^{1}, \mathcal{S})^{\txt{cart}}_{/p}$ can be identified
  with the full subcategory of $\Fun(\Delta^{1}, \mathcal{S})_{/p}$
  spanned by cartesian squares. We can thus identify the map to
  $\mathcal{S}_{/B}$ with a pullback of the forgetful functor from the
  full subcategory of
  $\Fun(\Delta^{1} \times \Delta^{1}, \mathcal{S})$ spanned by
  cartesian squares to the \icat{} of functors from
  $\Delta^{1} \times \Delta^{1} \setminus \{(0,0)\}$ to spaces. The
  latter is an equivalence by \cite{HTT}*{Proposition 4.3.2.15}, since
  it is the forgetful functor from squares that are right Kan extended
  from $\Delta^{1} \times \Delta^{1} \setminus \{(0,0)\}$.
\end{proof}

The main point of the proof of the theorem is the following general
lemma.
\begin{lemma}\label{lem:S/B->Arr/p}
  For any map of spaces $p\colon E\to B$, the full inclusion
  \[ j_p \colon \mathcal{S}_{/B} \simeq \Fun(\Delta^1,\mathcal{S})^{\txt{cart}}_{/p} 
  \longrightarrow \Fun(\Delta^1,\mathcal{S})_{/p}
  \]
  has both a left and a right adjoint.
	The left adjoint of $j_p$ takes a square
        \[
	\begin{tikzcd}
	X \arrow[r] \ar[rd, phantom, "\alpha"]
	\arrow[swap]{d}{x} & Y\arrow{d}{y} \\
	E \arrow[swap]{r}{p} & B\end{tikzcd}\]
	to $y \in \mathcal{S}_{/B}$ and the right adjoint of $j_p$ takes it to $p_*x \times_{p_* 
	p^* y} y$.
\end{lemma}
\begin{proof}
  The functor $\txt{ev}_{1} \colon \Fun(\Delta^{1}, \mathcal{S}) \to
  \mathcal{S}$ has as right adjoint the constant diagram functor
  (which is also given by right Kan extension). By
  \cite{HTT}*{Proposition 5.2.5.1} this induces for $p \colon E \to B$ an adjunction 
  on slice categories
  \[ \txt{ev}_{1} : \Fun(\Delta^{1}, \mathcal{S})_{/p} \rightleftarrows
  \mathcal{S}_{/B} : g_{p}, \]
  where the right adjoint $g_{p}$ takes $y\colon Y \to B$ to the pullback
        \[
	\begin{tikzcd}
	E \times_B Y \arrow[r] \arrow[d] & Y\arrow{d}{y} \\
	E \arrow[swap]{r}{p} & B
      \end{tikzcd}
      \]
  induced by the unit map. We can thus identify the right adjoint
  $g_{p}$ with $j_{p}$, so $j_{p}$ has a left adjoint with the stated
  description.

  It follows from Proposition~\ref{propn:Cartcolim} that $j_{p}$
  preserves limits and colimits, so since $\mathcal{S}_{/B}$ and
  $\Fun(\Delta^{1}, \mathcal{S})_{/p}$ are presentable the adjoint
  functor theorem implies that $j_{p}$ also has a right adjoint.
  To show the right adjoint has the claimed description, for $f$ 
  in $\mathcal{S}_{/B}$ and a
  square $\alpha$ as above, we must establish the equivalence
  \[
  \Map_{/B}( f, p_*x \times_{p_* p^* y} y)
  \simeq
  \Map_{\Fun(\Delta^1,\mathcal{S})_{/p}}(j_p f , \alpha) .
  \]
  But the mapping space on the right is the pullback
  \[
  \Map_{/E}(p^* f, x) \underset{\Map_{/B}(p_! p^* f, y)}{\times} 
  \Map_{/B}(f,y) 
  \]
  which is naturally equivalent to
  \[
  \Map_{/B}(f, p_* x) \underset{\Map_{/B}(f, p_* p^* y)}{\times} 
  \Map_{/B}(f,y) \\
  \simeq 
  \Map_{/B}( f,  p_* x \times_{p_* p^* y} y ) 
  \]
  as required.
\end{proof}

\begin{proof}[Proof of Theorem~\ref{thm:PolyFun/P=topos}]
  Suppose $P$ is represented by $I 
  \stackrel{s}\leftarrow E \stackrel{p}\to B \stackrel{t}\to I$.
By Lemma~\ref{lem:cartslice} we have an equivalence
\[
  \Fun(\Delta^1, \mathcal{S})^{\txt{cart}}_{/p} \simeq 
  \mathcal{S}_{/B}.
\]
  Using this equivalence, we have (via 
  Theorem~\ref{thm:polyispolyfun} and Definition~\ref{defn:Poly}):
  \begin{align*}
  \PolyFun_{/P} \simeq \Poly_{/P} 
  &\simeq
  \left( \Fun((\Delta^{1})^{\op}, \mathcal{S})
  \underset{\mathcal{S}}{\times} \Fun(\Delta^1, \mathcal{S})^{\txt{cart}}
  \underset{\mathcal{S}}{\times} \Fun(\Delta^{1}, 
  \mathcal{S}) \right)_{/(s,p,t)} \\
&\simeq
  \Fun((\Delta^{1})^{\op}, \mathcal{S})_{/s}
  \underset{\mathcal{S}_{/E}}{\times} \Fun(\Delta^1, \mathcal{S})^{\txt{cart}}_{/p}
  \underset{\mathcal{S}_{/B}}{\times} \Fun(\Delta^{1}, 
  \mathcal{S})_{/t} \\
  &\simeq
  \Fun((\Delta^{1})^{\op}, \mathcal{S})_{/s}
  \underset{\mathcal{S}_{/E}}{\times} \ \mathcal{S}_{/B} \
  \underset{\mathcal{S}_{/B}}{\times} \Fun(\Delta^{1}, 
  \mathcal{S})_{/t} .
  \end{align*}
  This is a double pullback of $\infty$-categories which are
  $\infty$-topoi, and the functors involved in the pullbacks are
  left exact left adjoints.  Hence the result is again an
  $\infty$-topos by \cite{HTT}*{Proposition~6.3.2.2}.  In detail, the
  four functors involved are
  \[
  \begin{tikzcd}[column sep={6em,between origins}]
	 \Fun((\Delta^{1})^{\op}, \mathcal{S})_{/s} \arrow[swap]{rd}{f_1} & & 
	 \arrow{ld}{p^*} \mathcal{S}_{/B} \arrow[swap]{rd}{\id}
	 && \arrow{ld}{f_4}\Fun(\Delta^{1}, 
  \mathcal{S})_{/t} \\
  & \mathcal{S}_{/E} && \mathcal{S}_{/B} &
  \end{tikzcd}
  \]
  where $f_1$ and $f_4$ are 
  slices of restriction functors along appropriate $\Delta^0 \to 
  \Delta^1$, hence have both adjoints by 
  \cite{HTT}*{Proposition~5.2.5.1} and its dual.  (The pullbacks are 
  pullbacks in $\LCatI$, or equivalently, pushouts in the 
  $\infty$-category of $\infty$-topoi and geometric morphisms, 
  cf.~\cite{HTT}*{6.3.1.5}.)
  
  The functor $\Poly_{/P} \to \Fun(\Pi, \mathcal{S})_{/P}$ is given by
  $\id \times_{\mathcal{S}_{/E}} j_p \times_{\mathcal{S}_{/B}} \id$,
  where $j_p: \mathcal{S}_{/B} \to \Fun(\Delta^1,\mathcal{S})_{/p}$ is
  from Lemma~\ref{lem:S/B->Arr/p}. We know from
  Proposition~\ref{propn:Cartcolim} that this functor preserves all
  limits and colimits.
\end{proof}

\begin{cor}\label{cor:PolyFun(I,J)=topos}
  For a fixed polynomial functor $P$ represented by $I \leftarrow 
  E \to B \to J$, we have 
  \[
  \PolyFun(I,J)_{/P} \simeq \mathcal{S}_{/B} .
  \]
  In particular, the $\infty$-category $\PolyFun(I,J)_{/P}$ is an 
  $\infty$-topos.
  \qed
\end{cor}

\begin{defn}
  We define the \icat{} $\PolyEnd$ of polynomial endofunctors by the
  pullback square
  \[
  \begin{tikzcd}
	  \PolyEnd \drpullback \arrow{r} \arrow{d} & \PolyFun \arrow{d} \\
	  \mathcal{S} \arrow{r} & \mathcal{S} \times \mathcal{S}.
  \end{tikzcd}
  \]
  Since $\PolyFun$ is cocomplete and the projection to $\mathcal{S}
  \times \mathcal{S}$ preserves colimits by
  Corollary~\ref{cor:polycolim}, it follows from \cite{HTT}*{Lemma
    5.4.5.5} that $\PolyEnd$ is also cocomplete.
\end{defn}

\begin{propn}\label{S->SxS}
  For a fixed polynomial endofunctor $P$, the \icat{} $\PolyEnd_{/P}$
  is an $\infty$-topos.
\end{propn}

\begin{proof}
  In the pullback diagram
\[
  \begin{tikzcd}
	  \PolyEnd_{/P} \drpullback \arrow{r} \arrow{d} & \PolyFun_{/P} \arrow{d}{\txt{ev}_{0,3}} \\
	  \mathcal{S}_{/I} \ar[r, "\Delta"'] & \mathcal{S}_{/I} \times 
	  \mathcal{S}_{/I} ,
  \end{tikzcd}
  \]
  both $\Delta$ and $\txt{ev}_{0,3}$ are left exact left adjoints.
  The former because it is pullback along the codiagonal $I\amalg I \to I$,
  the latter because it is the composite $\Poly_{/P} \to
  \mathcal{P}(\Pi)_{/P} \to \mathcal{S}_{/I} \times \mathcal{S}_{/I}$
  and here the first functor is a left exact left adjoint by
  Theorem~\ref{thm:PolyFun/P=topos} and the second is clear.
  Since the three $\infty$-categories are $\infty$-topoi, the 
  pullback is again an $\infty$-topos by 
  \cite{HTT}*{Proposition~6.3.2.2}.
\end{proof}

\section{Analytic Functors}\label{sec:anal}

\subsection{Analytic Functors and $\kappa$-Accessible Polynomial Functors}\label{subsec:analfun}

\begin{defn}
  A functor $\mathcal{S}_{/I} \to \mathcal{S}_{/J}$ is \emph{analytic}
  if it preserves weakly contractible limits and sifted colimits. 
\end{defn}

\begin{warning}
  This definition of analytic would not be correct if working over the
  category of sets instead of the category of spaces.  See
  Remark~\ref{rmk:analytic/Set} for further discussion of this subtle
  issue.
\end{warning}

From this definition it is immediate (using Theorem~\ref{thm:pnchar})
that an analytic functor is polynomial. We write $\AnFun$ for the full
subcategory of $\PolyFun$ spanned by the analytic functors, and
$\AnFun(I,J)$ for the corresponding subcategory of
$\PolyFun(I,J)$. Similarly, we define the \icat{} $\AnEnd$ of analytic
endofunctors as the pullback
\[
\begin{tikzcd}
  \AnEnd \drpullback \arrow{r} \arrow{d} & \AnFun \arrow{d} \\
  \mathcal{S} \arrow{r} & \mathcal{S} \times \mathcal{S}.
\end{tikzcd}
\]
We now wish to characterize the analytic functors (and also the
$\kappa$-accessible polynomial functors) in terms of their
representing diagrams.

\begin{defn}
  Let $\mathcal{C}$ be a cocomplete \icat, and let $\kappa$ be a
  regular cardinal.  Recall that an object $x$ is called
  \emph{$\kappa$-compact} when $\Map_{\mathcal{C}}(x, \blank) \colon
  \mathcal{C} \to \mathcal{S}$ preserves $\kappa$-filtered
  colimits~\cite{HTT}*{5.3.4.5}, and that it is called \emph{projective}
  when $\Map_{\mathcal{C}}(x, \blank)$ preserves geometric
  realizations~\cite{HTT}*{5.5.8.18}.
\end{defn}

\begin{remark}
  By \cite{HTT}*{Corollary 5.5.8.17} if $\mathcal{C}$ is cocomplete
  then a functor $F\colon \mathcal{C}\to\mathcal{D}$ preserves
  filtered colimits and geometric realizations if and only if it
  preserves sifted colimits.  In particular, $x\in \mathcal{C}$ is
  compact and projective if and only if
  $\Map_{\mathcal{C}}(x, \blank) \colon \mathcal{C} \to \mathcal{S}$
  preserves sifted colimits. If $\mathcal{C}$ is not assumed to be
  cocomplete, we still say that an object $x$ is \emph{compact
    projective} if mapping out of it preserves sifted colimits.
\end{remark}

\begin{lemma}\label{lem:slicecpt}
  Let $\mathcal{C}$ be 
  an \icat{} with finite products.
  Then
  $f \colon x \to y$ is a projective or $\kappa$-compact object in 
  $\mathcal{C}_{/y}$ if $x$
  is a projective or $\kappa$-compact object of $\mathcal{C}$. If
  $\mathcal{C}$ is cartesian closed, then the converse is also true.
\end{lemma}
\begin{proof}
  Consider a diagram $p \colon \mathcal{I} \to \mathcal{C}_{/y}$ that
  has a colimit. Since this colimit is preserved by the forgetful
  functor to $\mathcal{C}$, we have a commutative diagram
  \[
  \begin{tikzcd}
  \colim \Map_{/y}(x, p) \arrow{d} \arrow{r} & \Map_{/y}(x, \colim p)
  \arrow{r} \arrow{d}
  & \{f\} \arrow{d} \\
  \colim \Map(x, p) \arrow{r} & \Map(x, \colim p) \arrow{r} & \Map(x,
  y).
  \end{tikzcd}
\]
 Here the right square is clearly cartesian, and the composite
 square is cartesian since colimits are universal in $\mathcal{S}$. Therefore
 the left square is also cartesian, so if the lower left
 horizontal morphism is an equivalence, so is the top left horizontal
 morphism.

 If $\mathcal{C}$ is cartesian closed, then $\Map_{\mathcal{C}}(x,
 \colim p) \simeq \Map_{/y}(x, y \times \colim p) \simeq \Map_{/y}(x,
 \colim y \times p)$, which gives the converse.
\end{proof}

\begin{propn}\label{prop:adjcolimcpt}
  Consider an adjunction 
  \[
  F : \mathcal{C} \rightleftarrows \mathcal{D} : G .
  \]    
  If $G$ and $\Map_{\mathcal{C}}(x,\blank)$ both preserve
  $\mathcal{I}$-shaped colimits, then also
  $\Map_{\mathcal{D}}(Fx, \blank)$ preserves $\mathcal{I}$-shaped
  colimits.  Conversely, if equivalences in $\mathcal{C}$ are detected
  by mapping out of a collection of objects $x$ such that the functors
  $\Map_{\mathcal{C}}(x, \blank)$ and $\Map_{\mathcal{D}}(Fx, \blank)$
  both preserve $\mathcal{I}$-shaped colimits, then also $G$ preserves
  $\mathcal{I}$-shaped colimits.
\end{propn}
\begin{proof}
  Consider a
  diagram $p \colon \mathcal{I} \to \mathcal{D}$ that has a
  colimit. If $G$ and $\Map_{\mathcal{C}}(x,\blank)$ both preserve
  $\mathcal{I}$-shaped colimits, then we have
  \[ 
  \Map_{\mathcal{D}}(Fx, \colim p) \simeq \Map_{\mathcal{C}}(x,
  \colim Gp) \simeq \colim \Map_{\mathcal{C}}(x, Gp) \simeq \colim
  \Map_{\mathcal{D}}(Fx, p),
  \]
  and so $\Map_{\mathcal{D}}(Fx, \blank)$ also preserves
  $\mathcal{I}$-shaped colimits. Conversely, for an object $x$ in
  the collection we have equivalences 
  \[
  \Map_{\mathcal{C}}(x, G(\colim p)) \simeq
  \Map_{\mathcal{D}}(Fx, \colim p) \simeq 
  \colim \Map_{\mathcal{D}}(Fx, p) \simeq 
  \Map_{\mathcal{C}}(x, \colim Gp),
  \]
  and thus $\colim Gp \to G(\colim p)$ is an equivalence.
\end{proof}

Three special cases of this result are listed in the following 
corollary:
\begin{cor}\label{cor:adjcolimcpt}
  Consider an adjunction \[F : \mathcal{C} \rightleftarrows
  \mathcal{D} : G.\]
  \begin{enumerate}[(i)]
  \item If $G$ preserves $\kappa$-filtered colimits, then $F$
    preserves $\kappa$-compact objects. If equivalences in
    $\mathcal{C}$ are detected by mapping out of $\kappa$-compact
    objects, then the converse is true.
  \item If $G$ preserves geometric realizations, then $F$
    preserves projective objects. If equivalences in
    $\mathcal{C}$ are detected by mapping out of projective
    objects, then the converse is true.
  \item If $G$ preserves sifted colimits, then $F$
    preserves compact projective objects. If equivalences in
    $\mathcal{C}$ are detected by mapping out of compact projective
    objects, then the converse is true.
  \end{enumerate}
\end{cor}

\begin{lemma}\label{lem:spantopt}
  Consider a span of spaces $I \xfrom{f} X \xto{q} *$. The functor
  $q_{*}f^{*}$ preserves $\kappa$-filtered colimits \IFF{} $X$ is
  $\kappa$-compact, and sifted colimits \IFF{} $X$ is compact
  projective, i.e.\ is a finite set.
\end{lemma}
\begin{proof}
  Since equivalences in $\mathcal{S}$ are detected by maps out of $*$,
  and $\Map(*, \blank)$ preserves all colimits,
  Corollary~\ref{cor:adjcolimcpt} implies that $q_{*}f^{*}$ preserves
  $\kappa$-filtered colimits \IFF{} $f_{!}q^{*}$ preserves
  $\kappa$-compact objects, and sifted colimits \IFF{} $f_{!}q^{*}$
  preserves compact projective objects. By Lemma~\ref{lem:slicecpt}
  this is equivalent to $q^{*}$ preserving $\kappa$-compact or compact
  projective objects, respectively. Once again using that these are
  detected in $\mathcal{S}$, this is equivalent to $X \times Y$ being
  $\kappa$-compact or compact projective for all $\kappa$-compact or
  compact projective $Y$. Thus in particular (taking $Y = *$) $X$ is
  $\kappa$-compact or compact projective, but this is enough since if
  $X$ and $Y$ are $\kappa$-compact or compact projective and $p \colon
  \mathcal{I} \to \mathcal{S}$ is a $\kappa$-filtered or sifted
  diagram, then
  \[ \Map(X \times Y, \colim p) \simeq \Map(X, \Map(Y, \colim p))
  \simeq \Map(X, \colim \Map(Y, p)) \simeq \colim \Map(X \times Y,
  p),\]
  so $X \times Y$ is also $\kappa$-compact or compact projective.
\end{proof}

\begin{propn}\label{propn:analfibre}
  Suppose $F \colon \mathcal{S}_{/I} \to \mathcal{S}_{/J}$ is a
  polynomial functor represented by a diagram
  \[ I \xfrom{s} E \xto{p} B \xto{t} J.\]
  \begin{enumerate}[(i)]
	  \item
  $F$ is $\kappa$-accessible \IFF{} the fibres of $p$ are
  $\kappa$-compact spaces.
  \item $F$ is analytic \IFF{} the fibres of
  $p$ are finite sets.
  \end{enumerate}
\end{propn}
\begin{proof}
  We first prove (i):
  Since $t_{!}$ preserves and reflects colimits, $F$ is $\kappa$-accessible 
  \IFF{} the functor $p_{*}s^{*}$ preserves $\kappa$-filtered
  colimits.
  Using the equivalence $\mathcal{S}_{/B} \simeq \Fun(B,
  \mathcal{S})$, we see that $p_{*}s^{*}$ preserves $\kappa$-filtered colimits
  \IFF{} the same holds for $b^{*}p_{*}s^{*}$ for
  every point $b\colon * \to B$. Consider the pullback square
  \[
  \begin{tikzcd}
  E_b \drpullback \ar[d, "i"'] \ar[r, "q"]& \{b\} \ar[d, "b"] \\
  E \ar[r, "p"'] & B.
  \end{tikzcd}
  \]
  We have a Beck-Chevalley equivalence $b^{*}p_{*}s^{*} \simeq
  q_{*}(si)^{*}$. By Lemma~\ref{lem:spantopt} this preserves
  $\kappa$-filtered colimits \IFF{} $E_{b}$ is $\kappa$-compact.
  The proof of (ii) is the same, using sifted colimits instead of 
  $\kappa$-filtered colimits.
\end{proof}

\begin{remark}
  Let $\mathcal{X}$ be an $\infty$-topos and $\mathcal{F}$ a
  collection of morphisms in $\mathcal{X}$ that is stable under
  pullback. Then $\mathcal{F}$ determines a functor $\mathcal{X}^{\op}
  \to \LCatI$ taking $X$ to the full subcategory
  $\mathcal{X}^{\mathcal{F}}_{/X} \subseteq \mathcal{X}_{/X}$ spanned
  by the morphisms in $\mathcal{F}$, and given on morphisms by taking
  pullbacks. The class $\mathcal{F}$ is called \emph{local} if this
  functor preserves limits (\ie{} is a sheaf on
  $\mathcal{X}$); see also \cite{HTT}*{Lemma 6.1.3.7} for alternative characterizations. Following \cite{GepnerKock} we say that
  $\mathcal{F}$ is a \emph{bounded local class} if in addition the
  \icats{} $\mathcal{X}^{\mathcal{F}}_{/X}$ are all essentially
  small. By \cite{HTT}*{Proposition 6.1.6.3} these are exactly the
  classes of morphisms in $\mathcal{X}$ for which there exists a
  \emph{classifier} $U'_{\mathcal{F}} \to U_{\mathcal{F}}$, meaning a
  terminal object in the subcategory
  $\mathcal{O}_{\mathcal{X}}^{(\mathcal{F})}$ of $\Fun(\Delta^{1},
  \mathcal{X})$ with objects the morphisms in $\mathcal{F}$ and
  cartesian squares as morphisms.
\end{remark}

\begin{propn}\label{propn:Poly/F}\label{propn:bddlocalslice}
  Let $\mathcal{F}$ be a bounded local class of morphisms in
  $\mathcal{S}$, with classifier $U'_{\mathcal{F}} \to U_{\mathcal{F}}$, and let
  $F$ be the polynomial functor represented by 
  $* \leftarrow  U'_{\mathcal{F}} \to 
  U_{\mathcal{F}} \to *$.  Then the forgetful functor
  $$
  \PolyFun_{/F} \to \PolyFun
  $$ 
  is fully faithful, and its image is the full subcategory
  $\PolyFun_{\mathcal{F}}$ spanned by the polynomial
  functors with ``middle map'' in $\mathcal{F}$.
\end{propn}

\begin{proof}
    A morphism $P \to F$ in $\PolyFun$ is represented by a diagram
\[\begin{tikzcd}
  I \arrow{d}& E \arrow{l} \arrow{r}{p}\arrow{d} 
  \drpullback & B \arrow{r} \arrow{d}& J\arrow{d} \\
  * & U'_{\mathcal{F}} \arrow{l} \arrow{r} & U_{\mathcal{F}} \ar[r] & * .
\end{tikzcd}\]
Since $U_{\mathcal{F}}$ is the classifier for maps in the class 
$\mathcal{F}$, such a morphism exists \IFF{} $p$ belongs to $\mathcal{F}$,
and the morphism is unique if it exists. Thus the forgetful functor
from  $\PolyFun{}_{/F}$ to $\PolyFun$ is fully faithful, and its image
is precisely the full subcategory $\PolyFun_{\mathcal{F}}$.
\end{proof}

Combining this with Theorem~\ref{thm:PolyFun/P=topos}, we get:
\begin{cor}\label{cor:bbdlocaltopos}
  Let $\mathcal{F}$ be a bounded local class in $\mathcal{S}$. Then
  the \icat{} $\PolyFun_{\mathcal{F}}$ is an $\infty$-topos.
  \qedhere
\end{cor}

Specializing to $\kappa$-accessible and analytic functors, this gives:
\begin{cor}\label{cor:Polykappa=topos}
  Let $\kappa$ be a regular cardinal.
  \begin{enumerate}[(i)]
  \item Let $U'_{\kappa} \to U_{\kappa}$ be the classifying morphism
    for maps whose fibres are $\kappa$-compact spaces, and let
    $P_{\kappa}$ be the polynomial functor represented by
    \[ * \from U'_{\kappa} \to U_{\kappa} \to *.\]
    Then the \icat{} $\PolyFun_\kappa$ of $\kappa$-accessible
    polynomial functors is equivalent to
    $\PolyFun_{/P_{\kappa}}$. Moreover, $\PolyFun_{\kappa}$ is an
    $\infty$-topos.
  \item Let $\mathbf{E}$ be the polynomial functor represented by
      \[* \from \iota\Fin_{*} \to \iota\Fin \to *,\]
      where the middle map is the classifier for morphisms with finite
      discrete fibres. Then $\AnFun$ is equivalent to
      $\PolyFun_{/\mathbf{E}}$. Moreover, the $\infty$-category
      $\AnFun$ is an $\infty$-topos.
  \end{enumerate}
\end{cor}
\begin{proof}
  By Proposition~\ref{propn:analfibre}, the $\kappa$-accessible
  polynomial functors are those whose ``middle map'' belong to the
  bounded local class $\mathcal{F}_\kappa$ of maps with
  $\kappa$-compact fibres. This is equivalent to
  $\PolyFun_{/P_{\kappa}}$ by Proposition~\ref{propn:bddlocalslice},
  and is an $\infty$-topos by Corollary~\ref{cor:bbdlocaltopos}. This
  proves (i), and (ii) follows similarly since analytic functors are
  characterized by having ``middle map'' in the bounded local class of
  maps with finite discrete fibres.
\end{proof}
\begin{remark}
  Note that Corollary~\ref{cor:Polykappa=topos} does not have an analogue 
  in ordinary category theory, because of the lack of classifiers.
\end{remark}
\begin{remark}\label{rmk:PolyFunnotacc}
  The whole $\infty$-category $\PolyFun$ (without cardinal bounds on
  the middle representing maps) is cocomplete by
  Corollary~\ref{cor:polycolim}, but it is not accessible, since
  neither is $\Fun(\Delta^{1}, \mathcal{S})^{\txt{cart}}$. (In
  particular, $\PolyFun$ does not admit a terminal object.) In the
  $\kappa$-bounded case, the minimal generating set for
  $\Fun(\Delta^{1}, \mathcal{S})_{\kappa}^{\txt{cart}} \simeq
  \mathcal{S}_{/U_\kappa}$
  is the set of isomorphism classes of $\kappa$-compact spaces.
  Without the cardinal bound,
  $\Fun(\Delta^{1}, \mathcal{S})^{\txt{cart}}$ is the union of all
  these, and a generating set would have to exhaust
  $\bigcup_\kappa U_\kappa \simeq \iota \mathcal{S}$, which is too big
  to form a set.
\end{remark}

\subsection{Analytic Endofunctors, Symmetric Sequences, and Homotopical Species}\label{subsec:analendo}
In this subsection we will relate analytic endofunctors to (coloured)
symmetric sequences and the homotopical analogue of Joyal's species.

We saw in Corollary~\ref{cor:Polykappa=topos} that the \icat{} $\AnFun$ of analytic
functors is equivalent to the slice
$\PolyFun_{/\mathbf{E}}$. Combining this with
Corollary~\ref{cor:PolyFun(I,J)=topos}, we get:
\begin{cor}\label{cor:An(*)}
  We have
  \[
  \AnEnd(*) 
  \ \simeq \
  \mathcal{S}_{/\iota\Fin}
  \ \simeq \
  \Fun(\iota\Fin,\mathcal{S})
  \ \simeq \
  \prod_{n=0}^{\infty} \Fun(B \Sigma_{n}, \mathcal{S}) . \qedhere \]
\end{cor}
\begin{remark}\label{rmk:symseq}
  In the corollary,
  $\prod_{n=0}^{\infty} \Fun(B \Sigma_{n}, \mathcal{S})$ is the
  \icat{} of \emph{symmetric sequences} in $\mathcal{S}$.  The
  canonical monoidal structure on $\AnEnd(*)$ given by composition
  thus carries over to a monoidal structure on the \icat{} of
  symmetric sequences. Unravelling the formula for composition from
  Theorem~\ref{thm:comp}, we see that this is an $\infty$-categorical
  version of the substitution product on symmetric sequences,
  introduced by Kelly~\cite{KellyOnOperads} to exhibit operads as
  monoids therein.
\end{remark}
 \begin{defn}\label{defn:I-Coll}
  More generally, for a space $I$, we can consider \emph{$I$-coloured symmetric
  sequences} (or \emph{$I$-collections}): these are by definition 
  presheaves on 
  $\mathbf{E}(I) \times I$.  (We shall see a tree interpretation later on 
  in Definition~\ref{defn:Coll}.)
\end{defn}
\begin{propn}\label{propn:An(I)}
  The \icat{} $\AnEnd(I)$ of analytic endofunctors of
  $\mathcal{S}_{/I}$ is equivalent to that of $I$-coloured symmetric
  sequences.
\end{propn}
\begin{proof}
  Let $\mathbf{E}_{I}$ be the cartesian pullback (in the fibration
  $\AnEnd \to \mathcal{S}$) of $\mathbf{E}$ to an
  endofunctor of $I$, i.e.~the pullback along $(i,i)$ for
  $i \colon I \to *$; by Remark~\ref{rmk:cartpoly} this is the
  composite $i^{*}\mathbf{E}i_{!}$. Then $\AnEnd(I)$ is equivalent to
  $\PolyEnd(I)_{/\mathbf{E}_{I}}$. By
  Corollary~\ref{cor:PolyFun(I,J)=topos} this means that $\AnEnd(I)$
  is equivalent to $\mathcal{S}_{/\mathbf{E}_{I}(\id_{I})}$. But here
  \[ \mathbf{E}_{I}(\id_{I}) \simeq i^{*}\mathbf{E}i_{!}(\id_{I})
  \simeq i^{*}\mathbf{E}(I) \simeq I \times \mathbf{E}(I).\qedhere\]

\end{proof}

\begin{lemma}\label{lem:E(X)}
  We have the following explicit formula for evaluation of 
  $\mathbf{E}$ on a space $X$:
  \[
  \mathbf{E}(X) = \colim_{k \in \iota\Fin} \Map(k,X) = 
  \coprod_{k\in\mathbb{N}} X^{\times k}_{\mathrm{h}\Sigma_{k}}  . \qed
  \]
\end{lemma}

The relationship with $\mathbf{E}$ leads to a useful explicit
formula for evaluation of analytic endofunctors:
\begin{propn}\label{propn:pushforwardpbsq}
  Suppose $P$ is an analytic endofunctor, represented by the diagram
  \[
  \begin{tikzcd}
    I \arrow[swap]{d}{\jmath}& E \arrow{l}[above]{s} \arrow{r}{p} \arrow[swap]{d}{\epsilon} 
    \drpullback & B \arrow{r}{t} \arrow{d}{\beta}& I\arrow{d}{\jmath} \\
    * & \iota\Fin_{*} \arrow{l}{u} \arrow[swap]{r}{q} & \iota\Fin \arrow{r} & *.
  \end{tikzcd}
  \]
  Then for every map $f \colon X \to I$ there is a natural pullback
  square
  \[
  \begin{tikzcd}
	Y \drpullback
	\arrow{r}{} \arrow[swap]{d}{p_{*}s^{*}f} & 
	\mathbf{E}(X) \arrow{d}{\mathbf{E}(f)} \\
	B \arrow[swap]{r}{\bar s} & \mathbf{E}(I),
	\end{tikzcd}
  \]
  where
  $\bar s \colon B \to q_* u^* I = \mathbf{E}(I)$ corresponds to
  $u_!  q^* B = E \stackrel{s}\to I$ under the adjunction
  $u_! q^* \dashv q_* u^*$.
\end{propn}

\begin{proof}
By Lemma~\ref{lem:!Cart} we have a cartesian natural transformation
$p_{*}s^{*} \to p_{*}s^{*}\jmath^{*}\jmath_{!} \simeq
p_{*}\epsilon^{*}u^{*}\jmath_{!}$, and using Lemma~\ref{lem:BCpbk} we have a
Beck-Chevalley equivalence that identifies this with a natural
transformation $\eta \colon p_{*}s^{*} \to \beta^{*}q_{*}u^{*}\jmath_{!}$. Consider
the diagram
\[
\begin{tikzcd}
  Y \arrow{r} \arrow{d} & \beta^{*}q_{*}u^{*}X \arrow{d} \arrow{r} & q_{*}u^{*}X \arrow{d}\\
  B \arrow{r}& \beta^{*}q_{*}u^{*}I \arrow{d} \arrow{r}& q_{*}u^{*}I \arrow{d}\\
  & B \arrow[swap]{r}{\beta} & \iota\Fin.
\end{tikzcd}
\]
Here the bottom right square and the composite square in the right
column are cartesian by definition of $\beta^{*}$, hence the top right
square is also cartesian. The top left square is cartesian since $\eta$ is a cartesian
natural transformation, so the composite square in the top row is
cartesian.
\end{proof}

\begin{cor}\label{cor:analytic}
	For $P\colon \mathcal{S} \to \mathcal{S}$ an analytic endofunctor
	as in Proposition~\ref{propn:pushforwardpbsq}, we have
	\[
	P(X) \simeq \coprod_{n\in \mathbb{N}} B_n \times_{\Sigma_{n}} X^{\times n}.
      \]
      where $B_{n}$ is the fibre of $B \to \iota\Fin$ at an
      $n$-element set.
\end{cor}

\begin{proof}
	We calculate, using Proposition~\ref{propn:pushforwardpbsq} and 
	Lemma~\ref{lem:E(X)}:
	\[
	P(X) \simeq B \times_{\iota\Fin} \mathbf{E}(X) \simeq
	B \times_{\iota\Fin} \coprod_{n\in \mathbb{N}} (X^{\times n})_{\mathrm{h}\Sigma_{n}}
	\simeq  \coprod_{n\in \mathbb{N}} (B_n \times X^{\times n})_{\mathrm{h}\Sigma_{n}}.
	\qedhere
	\]
\end{proof}

\begin{remark}
  The formula in Corollary~\ref{cor:analytic} is the origin of the
  terminology ``analytic'': the spaces $B_n$ are the coefficients of
  the ``Taylor expansion'' of $P$.  Joyal~\cite{JoyalAnalytique}
  introduced analytic functors as a categorical analogue of
  exponential generating functions of species, defining them as left
  Kan extensions of species (which are functors
  $\iota \Fin \to \Set$).  He characterized analytic endofunctors of the
  category of sets as those endofunctors that preserve filtered
  colimits and weakly preserve wide pullbacks.  In our approach we
  have \emph{defined} analytic functors in terms of exactness
  properties, but can state an $\infty$-version of Joyal's theorem as
  follows:
\end{remark}
\begin{defn}
We call a functor $F\colon \iota\Fin\to\mathcal{S}$ a \emph{homotopical 
species}.  By left Kan extension along the (non-full) inclusion
$\iota\Fin \to \mathcal{S}$, it defines an endofunctor
$F\colon \mathcal{S}\to\mathcal{S}$,
described explicitly by the formula
\[
F(X) \simeq \colim_{n\in \iota\Fin} \colim_{n\to X} F[n] 
\simeq \colim_{n\in \iota\Fin} F[n] \times X^n 
\simeq \coprod_{n\in\mathbb{N}} (F[n] \times X^n)_{\mathrm{h}\Sigma_{n}}.
\]
\end{defn}
On the other hand, by unstraightening, it corresponds to a map $B \to
\iota\Fin$, and hence to an analytic functor (via Corollary~\ref{cor:An(*)}).
This analytic endofunctor is canonically identified with $F\colon \mathcal{S}
\to \mathcal{S}$, by Corollary~\ref{cor:analytic}, giving:
\begin{propn}[``Joyal's theorem for homotopical species'']
  An endofunctor $P\colon\mathcal{S} \to \mathcal{S}$ is analytic 
  (i.e.~preserves filtered colimits and weakly contractible limits)
  if and only if
  it is the left Kan extension of a ``homotopical species'' (i.e.~a functor
  $F \colon \iota\Fin \to \mathcal{S}$). \qed
\end{propn}
\begin{remark}\label{rmk:analytic/Set}
  From the viewpoint of species, analytic functors over sets are
  actually not the optimal notion, since it is not true in general
  that the exponential generating function of a species agrees with
  the cardinality of its associated analytic functor.  This \emph{is}
  true over spaces (and in fact already over groupoids, as first
  observed by Baez and Dolan~\cite{BaezDolanFSFD} who introduced
  groupoid-valued species under the name {\em stuff types}).  What
  goes wrong in the set case is the behaviour of quotients of group
  actions, which is also responsible for the mere {\em weak}
  preservation of connected limits in Joyal's original theorem.
\end{remark}

\subsection{Trees and Analytic Endofunctors}\label{subsec:treesanalend}
In this subsection we will describe analytic endofunctors in terms of
trees. This uses the interpretation of trees as polynomial
endofunctors from \cite{KockTree}:
\begin{defn}\label{defn:tree}
  A \emph{tree} is by definition a polynomial
  \[
  A \xfrom{s} M \xto{p} N \xto{t} A
  \]
  for which:
  \begin{enumerate}[(1)]
  \item The spaces $A$, $M$, and $N$ are all finite sets.
  \item The function $t$ is injective.
  \item The function $s$ is injective, with a unique element $R$ (the
    \emph{root}) in the complement of its image.
  \item Define a successor function $\sigma \colon A \to A$ as follows:
  First, set
    $\sigma(R) = R$. For $e \in s(M)$ (which is the complement of $R$ 
	in $A$), take $e'$ in $M$ with $s(e') = e$ and set 
	$\sigma(e) = t(p(e'))$. Then for every $e$ there exists some 
	$k\in\mathbb{N}$ such that
    $\sigma^{k}(e) = R$.
  \end{enumerate}
\end{defn}

\begin{remark}
  The intuition behind this notion of ``tree'' is as follows: we think
  of $A$ as the set of edges of the tree, $N$ as the set of
  nodes (our trees do not have nodes at their leaves or root),
  and $M$ as the set of pairs $(v, e)$ where $v$ is a node and
  $e$ is an incoming edge of $v$. The function $s$ is the projection
  $s(v,e) = e$, the function $p$ is the projection $p(v,e) = v$, and
  the function $t$ assigns to each node its unique outgoing edge.
\end{remark}

\begin{defn}
  The \emph{elements} of a tree are its edges and nodes, and a tree can be
  constructed by gluing edges and nodes, as will be formalized below.
  Let $\eta$ denote the tree
  \[ 
  * \from \emptyset \to \emptyset \to * 
  \]
  consisting of an edge without nodes; it is
  called the trivial tree.
  For $n = 0,1,\ldots$ let $C_{n}$ denote
  \[
  n+1 \hookleftarrow n \to * \hookrightarrow n+1,
\]
where the first and last morphisms are disjoint inclusions of $1$ and
$n$ elements in $n+1$; it is the corolla (one-node tree) with $n$
incoming edges.  We refer to the trivial tree and the corollas as
\emph{elementary trees}.
  
  We define $\bbOel$ and $\bbOint$ to be the full subcategories of
  $\AnEnd$ spanned by the elementary trees and all the trees,
  respectively.
\end{defn}

\begin{remark}\label{rmk:1trees}
  Since trees correspond to diagrams of sets, $\bbOel$ and $\bbOint$
  are ordinary categories, and they are equivalent to those considered
  by Kock~\cite{KockTree} (where they are denoted
  $\operatorname{elTr}$ and $\operatorname{tEmb}$, respectively).  It
  is a consequence of the tree axioms (see \cite[Proposition
  1.1.3]{KockTree}) that the morphisms in $\bbOint$ are \emph{tree
    embeddings}, meaning injective on nodes and edges.  The subscript
  ``int'' stands for \emph{inert}; in \S\ref{subsec:comparison} we
  will embed $\bbOint$ into a bigger category of trees $\bbO$, where
  the inert morphisms become the right class of an (active, inert)
  factorization system.
\end{remark}

The category $\bbOint$ admits certain pushouts (and colimits built
from them), namely ones corresponding to grafting of trees: if
$\eta \to S$ picks out the root and $\eta \to R$ picks out a leaf,
then the pushout $S \amalg_\eta R$ calculated in $\AnEnd$ (where it
exists since colimits in $\AnEnd$ can be calculated in
$\Fun(\Pi,\mathcal{S})$) is again a tree $T$, in which $R$ and $S$ are
naturally subtrees --- $T$ is ``$S$ grafted onto $R$''.  Hence the
pushout is also a pushout in $\bbOint$.  Furthermore, since the spaces
involved in the colimit are just sets and since the maps are
injections, the colimit can actually be calculated in $\Set$.  The
details can be found in \cite{KockTree}.
  
For a tree $T \in \bbOint$, we write $\el(T) = \bbOelT$ for the category
$\bbOel \times_{\bbOint} (\bbOint)_{/T}$, and call it the \emph{category of
elements} of $T$.  (Seeing $T$ as a presheaf on $\bbOel$ given by $E
\mapsto \Map_{\bbOint}(E,T)$, this really is its category of elements.)

The grafting construction can readily be iterated to establish the 
following result, which is intuitively clear:
\begin{lemma}\label{lem:treecolim}
  Every tree $T$ is canonically the colimit in $\bbOint$, and in $\AnEnd$,
  of its elementary subtrees:
  $$
  T \simeq \colim_{E \in \el(T)} E .
  $$
\end{lemma}
\begin{proof}
  This is a reformulation of \cite{KockTree}*{Corollary 1.1.24}.
\end{proof}

\begin{lemma}\label{lem:elmaps}
  Given an analytic endofunctor $P$ represented by a diagram
  \[
  \begin{tikzcd}
	I \arrow{d}& E \drpullback \arrow{l}[above]{s} \arrow{r}{p} \arrow{d}{\epsilon} 
	 & B \arrow{r}{t} \arrow{d}{\beta}& I\arrow{d} \\
	* & \iota\Fin_{*} \arrow{l} \arrow{r} & \iota\Fin \arrow{r} & *,
  \end{tikzcd}
  \]
  there are natural equivalences
  \[
  \Map(\eta, P) \simeq I, \qquad \Map(C_{n}, P) \simeq \fib{B}{n},
  \]
  where $\fib{B}{n}$ is the fibre of $\beta$ at an $n$-element set.
\end{lemma}

\begin{proof}
  It is clear that a map $\eta \to P$ is uniquely determined by the
  map $* \to I$, so $\Map(\eta, P) \simeq I$. For $C_{n}$, observe
  that since $n+1$ is the disjoint union of the images of $*$ and $n$,
  the space of maps $C_{n} \to P$ is equivalent to the space of
  cartesian squares 
  \[
  \begin{tikzcd}
	  n \drpullback \arrow{r}{u} \arrow{d} & * \arrow{d} \\
	  E \arrow[swap]{r}{p} & B .
\end{tikzcd}
\]
  More formally,
  this space is described as the pullback
  \[
  \begin{tikzcd}
    \Map(C_n,P) \arrow[phantom]{r}{\simeq} & 
	\Map_{\Fun^{\operatorname{cart}}(\Delta^1,\mathcal{S})}(u,p) 
	\drpullbackS
	\arrow{r}{}
	\arrow{d}{} & B \arrow{d}{\beta} \\
	& \iota\Fin_{n/} \arrow[swap]{r}{\mathrm{codom}} & \iota\Fin .
  \end{tikzcd}
  \]
 But $\iota\Fin_{n/}$ is contractible, so the pullback is $B_n$, as 
  asserted.
\end{proof}

\begin{defn}\label{defn:Coll}
  A \emph{coloured collection} or \emph{coloured symmetric
  sequence} is a presheaf on $\bbOel$. 
\end{defn}

\begin{remark}
The intuition is that the inclusion $\{\eta\} \to \bbOel$ defines a
projection $\mathcal{P}(\bbOel) \to \mathcal{S}$, which can be
interpreted as assigning to a coloured collection its space of
colours, and that the value of a presheaf on the corolla $C_n$ is the
space of $n$-ary operations of the coloured collection.  The $n+1$
different maps of trees $\eta \to C_n$ then extract from an $n$-ary
operation its $n$ input colours and its output colour.
\end{remark}

\begin{defn}
  The inclusion $i \colon \bbOel \to \AnEnd$ extends to a unique
  colimit-preserving functor $i_{!} \colon \mathcal{P}(\bbOel) \to \AnEnd$
  with right adjoint $i^{*} \colon \AnEnd \to \mathcal{P}(\bbOel)$ given by
  the restricted Yoneda functor, i.e.
  \[
  P \mapsto \Map_{\AnEnd}(i(\blank), P).
  \]
\end{defn}

\begin{propn}\label{propn:AnEndpsh}
  The functor $i^{*} \colon \AnEnd \to \mathcal{P}(\bbOel)$ is an equivalence.
\end{propn}

To prove this, we shall use the following general criterion.
\begin{lemma}\label{lem:pshcond}
  Suppose $\mathcal{C}$ is a cocomplete and locally small \icat{} and
  $i \colon \mathcal{C}_{0} \hookrightarrow \mathcal{C}$ is the
  inclusion of an essentially small full subcategory $\mathcal{C}_{0}$
  of $\mathcal{C}$ such that
  \begin{enumerate}[(i)]
  \item the objects of $\mathcal{C}_{0}$ are \emph{completely
      compact}, i.e.~for $C \in \mathcal{C}_{0}$ the functor
    $\Map_{\mathcal{C}}(C, \blank)$ preserves colimits,
  \item the functors $\Map_{\mathcal{C}}(C, \blank)$ for $C \in
    \mathcal{C}_{0}$ are \emph{jointly conservative}, i.e.~if a map $f
    \colon X \to Y$ in $\mathcal{C}$ is such that $f_{*} \colon
    \Map_{\mathcal{C}}(C,X) \to \Map_{\mathcal{C}}(C, Y)$ is an
    equivalence for all $C \in \mathcal{C}_{0}$, then $f$ is an equivalence.
  \end{enumerate}
  Then the adjunction \[i_{!} : \mathcal{P}(\mathcal{C}_{0})
  \rightleftarrows \mathcal{C} : i^{*}\]
  is an adjoint equivalence.
\end{lemma}
\begin{proof}
  The functor $i^{*}$ preserves colimits since the objects of
  $\mathcal{C}_{0}$ are completely compact, and detects equivalences
  since they are jointly conservative.
  The composite
  $i^{*}i_{!} \colon \mathcal{P}(\mathcal{C}_{0}) \to
  \mathcal{P}(\mathcal{C}_{0})$
  is thus a colimit-preserving functor that restricts to the Yoneda
  embedding on $\mathcal{C}_{0}$; it must therefore be the identity,
  and so the unit transformation $\id \to i^{*}i_{!}$ is an equivalence.
  To see that the counit transformation $i_{!}i^{*}\to \id$ is also an
  equivalence, it suffices to show that it is an equivalence after
  applying the conservative functor $i^{*}$, which now follows from
  the invertibility of the unit and one of the adjunction identities.
\end{proof}

\begin{proof}[Proof of Proposition~\ref{propn:AnEndpsh}]
  By Lemma~\ref{lem:pshcond} it suffices to check that the objects in
  $\bbOel$ jointly detect equivalences and are completely compact. A
  morphism
  \[
  \begin{tikzcd}
    I \arrow{r} \arrow{d}& E \drpullback
	\arrow{r}\arrow{d}& B\arrow{r}\arrow{d} &
    I \arrow{d}\\ 
    I' \arrow{r}& E'\arrow{r} & B'\arrow{r} & I'
  \end{tikzcd}
  \]
  in $\AnEnd$ is an equivalence \IFF{} the maps $I \to I'$ and $B \to
  B'$ are equivalences. The latter map is an equivalence \IFF{} for
  every $n$ the map on fibres $B_{n} \to B'_{n}$ is an equivalence. It
  thus follows from Lemma~\ref{lem:elmaps} that the objects in
  $\bbOel$ detect equivalences. Similarly, mapping out of them
  preserves colimits since these are computed levelwise by
  Corollary~\ref{cor:polycolim} and pullbacks preserve colimits.
\end{proof}

Having described analytic functors in terms of elementary trees, we now
describe them in terms of general trees.

\begin{defn}\label{def:Seg}
  The inclusion $u\colon \bbOel \to \bbOint$ induces a geometric 
  morphism $u_{*} \colon \mathcal{P}(\bbOel) \to \mathcal{P}(\bbOint)$, fully 
  faithful since $u$ is,
  hence identifying $\mathcal{P}(\bbOel)$ as a left exact localization
  of $\mathcal{P}(\bbOint)$.  We denote the image by $\PSeg(\bbOint)$
  and call its objects \emph{Segal presheaves}:
  \[
  \mathcal{P}(\bbOel) \stackrel\sim\to \PSeg(\bbOint) \subset \mathcal{P}(\bbOint)
  \]
  A presheaf $\Phi \in \mathcal{P}(\bbOint)$ is thus a Segal presheaf if it
  is a right Kan extension of its restriction to $\bbOel$.  This right Kan
  extension is calculated in the standard way using limits: 
  A presheaf $\Phi$ is Segal when the natural
  map $\Phi(T) \to \lim_{E \in \el(T)^{\op}} \Phi(E)$ is an equivalence.
  (Recall that $\el(T)=\bbOelT$ is the category of elements of $T$.)
\end{defn}

\begin{defn}
  Let $i_{0}$ denote the inclusion $\bbOint \hookrightarrow
  \AnEnd$. This extends to a unique colimit-preserving functor $i_{0,!} \colon
  \mathcal{P}(\bbOint) \to \AnEnd$ with right adjoint $i_{0}^{*}$
  given by the restricted Yoneda embedding.
\end{defn}

The commutative triangle of inclusion functors
  \[ 
  \begin{tikzpicture}
  \matrix (m) [matrix of math nodes,row sep=2.4em,column sep=1.1em,text height=1.5ex,text depth=0.25ex]
  {  \pgfmatrixnextcell \AnEnd \pgfmatrixnextcell \\
  \bbOel \pgfmatrixnextcell \pgfmatrixnextcell \bbOint \\ };
  \path[->,font=\footnotesize] %
  (m-2-1) edge node[left] {$i$} (m-1-2)%
  (m-2-3) edge node[right] {$i_0$} (m-1-2)%
  (m-2-1) edge node[below] {$u$} (m-2-3);%
  \end{tikzpicture}
  \]
induces a commutative diagram of right adjoint functors
  \[ 
  \begin{tikzcd}[shorten >= 8pt,shorten <= 8pt] 
	  {} & \AnEnd 
	  \arrow{ddl}[above]{i^{*}} \arrow{ddr}{i_{0}^{*}} & \\
	  &&\\
	\mathcal{P}(\bbOel) 
	&&   \mathcal{P}(\bbOint).
		\arrow{ll}{u^{*}}
  \end{tikzcd}
  \]
The functor $u^{*}$ given by composition with $u$ also has a right
adjoint $u_{*}$, given by right Kan extension along $u^{\op}$.
\begin{lemma}\label{lem:i0=u*i*}
  The natural transformation
  \[ i_{0}^{*} \to u_{*} u^{*} i_{0}^{*} \simeq u_{*} i^{*},\]
  induced by the unit for the adjunction $u^{*} \dashv u_{*}$,
  is an equivalence.
\end{lemma}

\begin{proof}
  For $P\in \AnEnd$ and $T\in \bbOint$, we have
  \[
  (i_0^{*}P)(T) \simeq \Map(i_0 T,P) 
  \simeq \Map\left(i_0\left(\colim_{E\in \el(T)} u E\right),P\right) 
  \simeq \Map\left(\colim_{E\in \el(T)} i_0 u E,P\right) 
  \simeq \lim_{E\in \el(T)^{\op}} \Map(i E,P),
  \]
  since $T$ is the colimit of its elementary subtrees in $\bbOint$ and
  this colimit is preserved by the inclusion $i_{0}$ by
  Lemma~\ref{lem:treecolim}. But this gives precisely the limit
  formula for the right Kan extension $u_*i^* P$.
\end{proof}

\begin{propn}\label{prop:ff_int}
  The functor $i_{0}^{*}\colon \AnEnd \to \mathcal{P}(\bbOint)$ is fully
  faithful with image the Segal presheaves. In other words, it induces
  an equivalence
  \[ \AnEnd \isoto \PSeg(\bbOint).\]
\end{propn}

\begin{proof}
	By Lemma~\ref{lem:i0=u*i*}, $i_{0}^{*}$ factors as $i^*$ followed 
	by $u_*$.  But $i^*$ is an equivalence by 
	Proposition~\ref{propn:AnEndpsh}, so $i_{0}^{*}$ is fully 
	faithful because $u_*$ is, and has the same 
	image as $u_*$, which is $\PSeg(\bbOint)$ by definition.
\end{proof}

\section{Initial Algebras and Free Monads}
\label{sec:IAFM}

\subsection{Initial Lambek Algebras}\label{subsec:lambek}
In this subsection we prove an \icatl{} version of the existence
theorem for initial Lambek algebras. In ordinary category theory, the
study of initial algebras for endofunctors goes back to
Lambek~\cite{LambekFixpoint}, while the existence result is due to
Ad\'amek~\cite{AdamekFreeAlgebras}. 
In the present account, we establish the initial-algebra
theorem via a lightweight version of bar-cobar duality.

\begin{defn}\label{defn:algP}
  Let $P:\mathcal{C}\to\mathcal{C}$ be any endofunctor.  Recall that a
  {\em Lambek $P$-algebra} is a pair $(A,a)$ where $A$ is an object of
  $\mathcal{C}$ and $a\colon PA\to A$ is a morphism of $\mathcal{C}$.
  Dually, a {\em Lambek $P$-coalgebra} is a pair $(C,c)$ where $C$ is
  an object and $c\colon C \to PC$ is a morphism.  (We shall omit the
  attribute `Lambek' for the rest of this subsection.)
  Formally, the
  $\infty$-categories of $P$-algebras and $P$-coalgebras are defined as
  pullbacks
  $$
  \begin{tikzcd}
	\alg_{P}(\mathcal{C}) \drpullback \ar[r] \ar[d] & \mathcal{C}^{\Delta^1} \ar[d, "{(\txt{ev}_0,\txt{ev}_1)}"] \\
	\mathcal{C} \ar[r, "{(P,\id)}"'] & \mathcal{C}\times\mathcal{C} ,
	\end{tikzcd}
  \qquad
  \begin{tikzcd}
	\coalg_{P}(\mathcal{C}) \drpullback \ar[r] \ar[d] & \mathcal{C}^{\Delta^1} \ar[d, "{(\txt{ev}_0,\txt{ev}_1)}"] \\
	\mathcal{C} \ar[r, "{(\id,P)}"'] & \mathcal{C}\times\mathcal{C} .
	\end{tikzcd}
  $$
\end{defn}

\begin{defn}
  If $(A,a)$ is a $P$-algebra and $(C,c)$ is a $P$-coalgebra, then a
  \emph{$P$-twisting morphism} is a morphism $f \colon C \to A$ in $\mathcal{C}$
  together with a commutative square
  $$
  \begin{tikzcd}
	  & C \ar[ddd, "f"] \ar[ld, "c"'] \\[-4ex]
PC \ar[d, "Pf"'] & \\
PA \ar[rd, "a"'] & \\[-4ex]
& A  .
  \end{tikzcd}
  $$
  We define the space $\Tw_P(C,A)$ 
  of $P$-twisting morphisms from $C$ to $A$ as the equalizer
  $$
  \Tw_P(C,A) \to \Map(C,A) \rightrightarrows \Map(C,A)  ,
  $$
  where the two maps $\Map(C,A) \to \Map(C,A)$ are the identity
  and $f\mapsto a \circ Pf \circ c$.
  (The equation $f \simeq a \circ Pf \circ f$ may be viewed as the
  analogue of the {\em Maurer--Cartan equation} in this context.)
\end{defn}

It will be useful to note that a $P$-twisting morphism may also be seen as a
$\Tw(P)$-algebra in the twisted arrow \icat{} $\Tw(\mathcal{C})$, as we proceed
to establish.
  
\begin{defn}
  Recall that, for $\mathcal{C}$ an $\infty$-category, the {\em twisted arrow
  \icat{}} $\Tw(\mathcal{C})$ has
  as
  objects the morphisms in $\mathcal{C}$, and a morphism in
  $\Tw(\mathcal{C})$ from $f'\colon X'\to Y'$ to $f\colon X\to Y$ is a
  commutative diagram
  $$
  \begin{tikzcd}
	  & X \ar[ddd, "f"] \ar[ld] \\[-4ex]
X' \ar[d, "f'"'] & \\
Y' \ar[rd] & \\[-4ex]
& Y .
  \end{tikzcd}
  $$
  See \cite{BarwickMackey} or \cite{HA}*{\S 5.2.1} for a more formal definition.
  Note that in our convention it is the {\em codomain} component
  that determines the direction of the morphism.  (Lurie~\cite{HA}*{\S 5.2.1}
  uses the opposite convention.)
  
  There is a canonical left fibration 
  $$(\operatorname{dom}, \operatorname{codom})  \colon 
  \Tw(\mathcal{C}) \to \mathcal{C}^{\op}\times \mathcal{C} ,
  $$
  corresponding to the functor $\Map(-,-)\colon \mathcal{C}^{\op}\times \mathcal{C} \to 
  \mathcal{S}$.
\end{defn}

\begin{propn}\label{propn:algTwP}
  The endofunctor $P\colon \mathcal{C}\to\mathcal{C}$ induces an
  endofunctor $\Tw(P) \colon \Tw(\mathcal{C}) \to \Tw(\mathcal{C})$
  and a morphism of endofunctors
  \csquare{\Tw(\mathcal{C})}{\Tw(\mathcal{C})}{\mathcal{C}^{\op}
    \times \mathcal{C}}{\mathcal{C}^{\op} \times
    \mathcal{C}.}{\Tw(P)}{}{}{P^{\op} \times P} 
    This induces a
  functor
  \[
  \alg_{\Tw(P)}(\Tw(\mathcal{C})) \longrightarrow 
  \alg_{P^{\op} \times P}(\mathcal{C}^{\op} \times \mathcal{C}) \simeq
  \alg_{P^{\op}}(\mathcal{C}^{\op}) \times \alg_{P}(\mathcal{C}) \simeq 
  \coalg_{P}(\mathcal{C})^{\op} \times \alg_{P}(\mathcal{C})
  \]
  which is a left fibration such that the fibre over $(C,A)\in\coalg_{P}(\mathcal{C})^{\op} \times \alg_{P}(\mathcal{C})$ is the space $\Tw_P(C,A)$ of
  $P$-twisting morphism from $C$ to $A$.
\end{propn}

Before we prove this, we make some simple observations:
\begin{lemma}\label{lem:pbcocart}
  Consider a commutative diagram of \icats{}
  \[
  \begin{tikzcd}
    \mathcal{E}' \arrow{r} \arrow{d} & \mathcal{E} \arrow{d} &
    \mathcal{E}'' \arrow{l} \arrow{d} \\
    \mathcal{B}' \arrow{r}  & \mathcal{B} &
    \mathcal{B}'' \arrow{l}
  \end{tikzcd}
  \]
  where the vertical maps are cocartesian fibrations and the upper
  horizontal maps preserve cocartesian morphisms. Then the induced
  functor
  \[ \mathcal{E}' \times_{\mathcal{E}} \mathcal{E}'' \to \mathcal{B}'
  \times_{\mathcal{B}} \mathcal{B}'' \]
  is again a cocartesian fibration, and the canonical functors to
  $\mathcal{E}'$, $\mathcal{E}$, and $\mathcal{E}''$ all preserve
  cocartesian morphisms. Moreover, if the vertical maps are actually left
  fibrations, then so is this new map.
\end{lemma}
\begin{proof}
  Given a morphism $f$ in $\mathcal{B}' \times_{\mathcal{B}}
  \mathcal{B}''$ it is easy to see that the morphism in $\mathcal{E}'
  \times_{\mathcal{E}} \mathcal{E}''$ corresponding to a compatible
  choice of cocartesian morphisms over the images of $f$ in
  $\mathcal{B}'$, $\mathcal{B}$, and $\mathcal{B}''$ is cocartesian.
\end{proof}

\begin{lemma}\label{lem:algcoCart}
  Suppose $\mathcal{C}$ and $\mathcal{D}$ are $\infty$-categories
  equipped with endofunctors $P \colon \mathcal{C}\to\mathcal{C}$ and
  $Q \colon \mathcal{D}\to\mathcal{D}$, and $F \colon \mathcal{D}\to\mathcal{C}$ is a
  cocartesian fibration which is compatible with $P$ and $Q$ in the sense that
  there is a commutative square
  \[\begin{tikzcd}
    \mathcal{D} \ar[r, "Q"] \ar[d, "F"'] & \mathcal{D} \ar[d, "F"] \\
    \mathcal{C} \ar[r, "P"'] & \mathcal{C},
  \end{tikzcd}\]
  and $Q$ preserves $F$-cocartesian morphisms. Then the resulting functor
  $\alg_Q(\mathcal{D})\to\alg_P(\mathcal{C})$ is a cocartesian fibration.
  Furthermore, if $F$ is actually a left fibration, then
  $\alg_Q(\mathcal{D})\to\alg_P(\mathcal{C})$ is itself a left fibration.
\end{lemma}
\begin{proof}
  The functor $F$ induces morphisms of cocartesian (respectively,
  left) fibrations
  \[
  \begin{tikzcd}
  \mathcal{D}^{\Delta^1} \ar[r] \ar[d] & \mathcal{D}^{\partial\Delta^1} \ar[d] \\
  \mathcal{C}^{\Delta^1} \ar[r] & \mathcal{C}^{\partial\Delta^1}
  \end{tikzcd}
  \quad \text{ and } \quad
  \begin{tikzcd}
  \mathcal{D} \ar[d] \ar[r, "{(Q,\id)}"] & \mathcal{D}\times\mathcal{D} \ar[d] \\
  \mathcal{C} \ar[r, "{(P,\id)}"'] & \mathcal{C}\times\mathcal{C}
  \end{tikzcd}
  \]
  (note that this second square commutes by virtue of our assumption
  on $F$).  Taking pullbacks, we obtain the natural map
  $\alg_Q(\mathcal{D})\to\alg_P(\mathcal{C})$ which is therefore a
  cocartesian (respectively, left) fibration by
  Lemma~\ref{lem:pbcocart}.
\end{proof}

\begin{proof}[Proof of Proposition~\ref{propn:algTwP}]
  It follows from Lemma~\ref{lem:algcoCart} applied to the left
  fibration $F \colon \Tw(\mathcal{C}) \to \mathcal{C}^{\op} \times
  \mathcal{C}$ that the induced map
  \[\alg_{\Tw(P)}(\Tw(\mathcal{C})) \to \coalg_{P}(\mathcal{C})^{\op}
  \times \alg_{P}(\mathcal{C})\]
  is a left fibration. It remains to see that the fibre over
  $(C,A)\in\coalg_{P}(\mathcal{C})^{\op} \times \alg_{P}(\mathcal{C})$
  is the space of $P$-twisting morphisms.  By construction, this fibre
  is obtained as the pullback of fibres of the induced left
  fibrations. The fibre of $F$ over the object $(C,A)$ is the space
  $\Map(C,A)$, and the fibre of
  $F^{\Delta^1} \colon
  \Tw(\mathcal{C})^{\Delta^1}\to (\mathcal{C}^{\op})^{\Delta^1}\times\mathcal{C}^{\Delta^1}$
  is computed as $\Map(C,A)^{\Delta^1}\simeq\Map(C,A)$.  Hence the
  pullback $M$ of the fibres fits into the commutative square
  \[\begin{tikzcd}
  M \ar[rr] \ar[d] && \Map(C,A) \ar[d, "{\txt{diag.}}"] \\
  \Map(C,A) \ar[rr, "{f\mapsto(a\circ Pf\circ c,f)}"'] && 
  \Map(C,A)\times\Map(C,A).
  \end{tikzcd}\]
  But this $M$ is just a pullback reformulation of the equalizer definition of 
  $\Tw_P(C,A)$.
\end{proof}

\begin{lemma}
  A $P$-coalgebra morphism $(C,c) \to (C',c')$ 
  induces a map
  $\Tw_P(C',A) \to \Tw_P(C,A)$ by pre-composition.  Similarly, a $P$-algebra morphism $(A',a') \to (A,a)$ induces 
  a map $\Tw_P(C,A') \to  \Tw_P(C,A)$
  by post-composition.
\end{lemma}

\begin{proof}
  This is immediate from the description of these spaces as fibres of a left fibration.
\end{proof}

If $(C,c)$ is a $P$-coalgebra, then $(PC,Pc)$ is a $P$-coalgebra,
and $c\colon (C,c) \to (PC, Pc)$ is a $P$-coalgebra morphism.
The following is the key property of twisting morphisms:
\begin{lemma}\label{lem:PCA}
  For a $P$-coalgebra $(C,c)$ and a $P$-algebra $(A,a)$,
  the map $\Tw_P(PC,A) \to \Tw_P(C,A)$ which sends $g$ to $g\circ c$, is an equivalence, 
  with inverse the map $\Tw_P(C,A)\to\Tw_P(PC,A)$ which sends $f$ to $a \circ Pf$.
\end{lemma}

\begin{proof}
  We first detail the inverse.
  If $f\colon C \to A$ is a twisting morphism with square 
  $$\begin{tikzcd}
	PC \ar[d, "Pf"'] &   C \ar[d, "f"] \ar[l, "c"'] \\
	PA \ar[r, "a"'] & A ,
  \end{tikzcd}$$
  apply $P$ and paste with a trivial square like this:
  $$\begin{tikzcd}
	PPC \ar[d, "PPf"'] &  PC \ar[d, "Pf"] \ar[l, "Pc"'] \\
	PPA \ar[r, "Pa"'] \ar[d, "Pa"']  & PA \ar[d, "a"]\\
	PA  \ar[r, "a"'] & A . 
  \end{tikzcd}$$
  The left vertical composition is $Pa\circ PPf\simeq P(a \circ Pf)$, so
  the composite square exhibits $a \circ Pf$ as a twisting morphism,
  as required.  

  To see that the two constructions are inverse, we check that the respective composites are naturally equivalent to the respective identity functors. Starting with the 
  square for $f\colon C \to A$, going left and then back right gives
  $$\begin{tikzcd}
     PC \ar[d, "Pc"'] & C \ar[d, "c"] \ar[l, "c"']\\	
	 PPC \ar[d, "PPf"'] & PC \ar[d, "Pf"] \ar[l, "Pc"'] \\
    PPA \ar[r, "Pa"'] \ar[d, "Pa"'] & PA \ar[d, "a"] \\
    PA \ar[r, "a"'] & A ,
  \end{tikzcd}$$
  but this is homotopic to the original square for $f$ since $f$ is 
  twisting.
  On the other hand, starting with the square for $g\colon PC \to A$,
  going right and then back left gives
  $$\begin{tikzcd}
	 PPC \ar[d, "PPc"'] & PC \ar[d, "Pc"] \ar[l, "Pc"'] \\
	PPPC \ar[d, "PPg"'] & PPC \ar[d, "Pg"] \ar[l, "PPc"'] \\
	PPA \ar[r, "Pa"'] \ar[d, "Pa"'] & PA \ar[d, "a"] \\
	PA \ar[r, "a"'] & A
  \end{tikzcd}$$
  which is also homotopic to the original square for $g$ since $g$ is 
  twisting.
\end{proof}

\begin{defn}\label{defn:cobar}
  Assume $\mathcal{C}$ has filtered colimits and that $P\colon 
  \mathcal{C}\to\mathcal{C}$ preserves
  them. Then for a $P$-coalgebra $C$ we have a diagram
  \[ C \xto{c} PC \xto{Pc} P^{2}C \to \cdots.\]
  Let $I_{C} := \colim_{n \to \infty} P^{n}C$ be the colimit of this
  sequence. Then there is a canonical map
  \[ 
I_{C}\simeq \colim_{n}
  P^{n+1}C \to P(\colim_{n} P^{n}\mathcal{C}) \simeq PI_{C}.\]
  Since $P$ preserves filtered colimits, this map is an
  equivalence. If $u$ denotes its inverse, the pair $(I_{C}, u)$ is a
  $P$-algebra. We denote this $P$-algebra $\Omega C$ and refer to it
  as the \emph{cobar construction} of $C$. 
\end{defn}

We will now establish a universal property of the cobar construction,
which in particular implies that it determines a functor
$\Omega \colon \coalg_{P}(\mathcal{C}) \to \alg_{P}(\mathcal{C})$.
Under further assumptions, we will see that it is left adjoint to a
dual \emph{bar} construction, which gives a $P$-coalgebra from a
$P$-algebra.

\begin{lemma}\label{lem:UA}
  If $(U,u)$ is a $P$-coalgebra for which $u$ is an equivalence,
  with inverse $v$ giving a $P$-algebra $(U,v)$,
  then for any $P$-algebra $(A,a)$ we have
  $$
  \Map_{\alg_P}( U, A ) \simeq \Tw_P (U, A)  .
  $$
\end{lemma}

\begin{proof}
  Consider the diagram
  \[
  \begin{tikzcd}
    \Map_{\alg_{P}(\mathcal{C})}((U,v),(A,a)) 
	\drpullbackS
	\arrow{r} \arrow{d} & \Map_{\mathcal{C}^{\Delta^{1}}}(v, a)
  \arrow{d} \\
  \Map_{\mathcal{C}}(U, A) \ar[r, "{f{\mapsto}(Pf,f)}"]
  \ar[d, shift left, "f{\mapsto}a\circ Pf \circ u"' inner sep=1.5ex] \ar[d, shift
  right, "f{\mapsto}f" inner sep=1.5ex] 
  & \Map_{\mathcal{C}}(PU,PA)
  \times \Map_{\mathcal{C}}(U,A) \ar[d, shift left, "g{\mapsto}a\circ
  g"' inner sep=1.5ex] \ar[d, shift right, 
  "f{\mapsto}f\circ v" inner sep=1.5ex] \\
  \Map_{\mathcal{C}}(U, A) \ar[r, "f{\mapsto}f\circ v"'] & 
  \Map_{\mathcal{C}}(PU, A)
  \end{tikzcd}
  \]
  The top square is a pullback by definition of $\alg_{P}(\mathcal{C})$ 
  as a pullback.  The bottom map is an equivalence since $v$ is an equivalence.
  Since the right fork is an equalizer, it follows (by a standard argument, for 
  example by expressing equalizers as pullbacks) that also
  the left fork is an equalizer, hence $\Map_{\alg_P}( U, A ) \simeq \Tw_P (U, 
  A)$ as required.
\end{proof}

\begin{defn}
  Assume that $\mathcal{C}$ has filtered colimits and that $P\colon 
  \mathcal{C}\to\mathcal{C}$ preserves
  them.  Given a $P$-coalgebra $(C,c)$, the {\em universal $P$-twisting morphism}
  $(C,c) \to (\Omega C, u)$ is the canonical map
  $$
  C \to \colim_{n \to \infty} P^n C ,
  $$
  which is $P$-twisting by virtue of the diagram
  $$
  \begin{tikzcd}
	  & C \ar[ddd] \ar[ld, "c"'] \\[-4ex]
PC \ar[d] & \\
P(\Omega C) \ar[rd, "u"'] & \\[-4ex]
& \Omega C 
  \end{tikzcd}
  $$
  (where all the morphisms are the canonical ones from the colimit diagram
  defining $\Omega C$). 
\end{defn}

\begin{propn}\label{propn:OmegaCA}
  For a $P$-coalgebra $(C,c)$ and a $P$-algebra $(A,a)$ there is a
  canonical equivalence
  $$
  \Map_{\alg_P}(\Omega C, A) \simeq \Tw_P(C,A)
  $$
  given by precomposing with the universal $P$-twisting morphism.
\end{propn}

\begin{proof}
  By Lemma~\ref{lem:UA}, we have
  $$
  \Map_{\alg_P}(\Omega C, A) \simeq \Tw_P (\colim_n P^n C, 
  A) ,
  $$
  and the latter space is described as an equalizer
\[
 	\Tw_P(\colim_n P^n C,A) \to \Map_\mathcal{C}(\colim_n P^n C,A) 
	\rightrightarrows \Map_\mathcal{C}(\colim_n P^n C,A) .
\]
  The mapping spaces are in turn limits.  Altogether we can write down a
  big commutative diagram
  $$\begin{tikzcd}
    \vdots \ar[d] && \vdots \ar[d] \\
	\Map(P^2 C, A) \ar[rr, shift left, "{\id}"] 
	\ar[rr, shift right, "{a \circ P( - ) \circ P^2(c)}"'] \ar[d]
	&& \Map(P^2 C, A) \ar[d]
	\\
	\Map(PC, A) \ar[rr, shift left, "{\id}"] 
	\ar[rr, shift right, "{a \circ P( - ) \circ P(c)}"']  
	\ar[d]
	&& \Map(PC, A) \ar[d]
	\\
    \Map(C, A) \ar[rr, shift left, "{\id}"] 
    \ar[rr, shift right, "{a \circ P( - ) \circ c}"']  
	&& \Map(C, A)
  \end{tikzcd}$$
  (It is clear that it commutes, both for the identity maps
  and for the other horizontal maps.)

  We calculate the limit of this diagram in two ways.  First we
  calculate the limit of each column, yielding the parallel pair of maps
  \[ \Map_\mathcal{C}(\colim_n P^n C,A) \rightrightarrows
  \Map_\mathcal{C}(\colim_n P^n C,A),\]
  and then we take the equalizer of this to obtain
  $\Tw_P(\colim_n P^n C,A)$.  On the other hand, we can calculate the
  limit by first taking the equalizer of each row.  That gives in each
  row the space $\Tw_P(P^n C, A)$, and then we can calculate the
  sequential limit of this new column.  Now note that all the maps in
  the new column are equivalences: this follows from 
  Lemma~\ref{lem:PCA}.  So the limit is equivalent to just the zeroth
  space $\Tw_P(C, A)$ as claimed.
\end{proof}

\begin{propn}\label{propn:initalg}
  If $\mathcal{C}$ is an $\infty$-category with filtered colimits and an initial 
  object $\emptyset$, and $P\colon \mathcal{C}\to\mathcal{C}$ is a 
  filtered-colimit-preserving endofunctor, then the $\infty$-category
  $\alg_P(\mathcal{C})$ has an initial object, given by $\Omega\emptyset$.
\end{propn}

\begin{proof}
  $\emptyset$ has a unique coalgebra structure, and 
  Proposition~\ref{propn:OmegaCA} gives, for any $P$-algebra $A$,
  $$
  \Map_{\alg_P}(\Omega \emptyset, A) \simeq \Tw_P(\emptyset,A).
  $$
  Since the latter space is clearly contractible, it follows that
  $\Omega\emptyset$ is an initial $P$-algebra.
\end{proof}

\begin{remark}
  The constructions, results and proofs go through more generally when
  $\mathcal{C}$ has $\kappa$-filtered colimits and $P\colon 
  \mathcal{C}\to\mathcal{C}$ preserves them.  The
  important point is that even if $P$ does not preserve $\omega$-filtered
  colimits, there is still a transition map $I \to PI$ at each colimit
  step, so that $I$ is again a $P$-coalgebra.  $P$ can now be applied
  iteratively again, the next colimit can be taken, and so on, until the
  resulting chain is longer than $\kappa$, and $P$ will preserve the
  colimit to yield an invertible structure map for the resulting
  $P$-coalgebra, and hence a $P$-algebra.  Accepting notation such as
  $\colim_{n \leq \kappa} P^n C$ for transfinite application of $P$
  alternated with taking colimits, all the subsequent constructions go
  through.
\end{remark}

\begin{remark}
  All the constructions, results and proofs can be dualized: assume that
  $\mathcal{C}$ has cofiltered limits, and that $P$ preserves them.  Then there is a
  functor
  $$
  B\colon \alg_P \to \coalg_P
  $$
  taking a $P$-algebra $(A,a)$ to the limit of the chain $A
  \stackrel{a}\leftarrow PA \stackrel{Pa}\leftarrow P^2 A \leftarrow \cdots$.
  (This is called the {\em bar construction}.)

  The notion of $P$-twisting morphism is still the same, but now the
  results are about applying $P$ to $A$ instead of $C$.  
  Lemma~\ref{lem:PCA}
  becomes the statement that the following maps are inverse homotopy 
  equivalences:
  \begin{eqnarray*}
	\Tw_P(C,PA) & \longrightarrow & \Tw_P(C,A) \\
    g & \longmapsto & a \circ g \\
	Pf \circ c  & \longmapsfrom & f 
  \end{eqnarray*}
  Assuming that $\mathcal{C}$ has cofiltered limits and $P$ preserves them,
  we get
  $$
  \Map_{\alg_P}(C, BA) \simeq \Tw_P(C,A).
  $$
  Putting together the two sides of duality, we get:
\end{remark}

\begin{thm}
  If $\mathcal{C}$ has filtered colimits and cofiltered limits, and if
  $P\colon \mathcal{C}\to\mathcal{C}$ preserves them, then $\Omega$ is left adjoint
  to $B$.  Altogether
  $$
  \Map_{\alg_P}( \Omega C, A) \simeq \Tw_P(C,A) \simeq
  \Map_{\coalg_P}(C, BA) .
  $$
\end{thm}

In the case of interest here, $\mathcal{C}$ will be presentable, and $P$
will be analytic.  In particular $P$ then preserves filtered colimits, and
also preserves cofiltered limits (since these are weakly contractible),
so the theorem applies.

\subsection{Free Monads}\label{subsec:freemndexist}

\begin{defn}\label{defn:Ptil}
  Let $\mathcal{C}$ be an $\infty$-category with binary coproducts,
  let $P\colon \mathcal{C} \to \mathcal{C}$ be an endofunctor, and let $X$ 
  be an object of $\mathcal{C}$.  Define a new endofunctor
  $P_X \colon \mathcal{C}_{X/} \to \mathcal{C}_{X/}$ as the 
  composite
  \[
  \mathcal{C}_{X/} \stackrel{u_X}\longrightarrow
  \mathcal{C} \stackrel{P}\longrightarrow
  \mathcal{C} \stackrel{a_X}\longrightarrow
  \mathcal{C}_{X/} ,
  \]
  where $u_X$ is the forgetful functor, with left adjoint $a_X = X\amalg( \ )$.
\end{defn}
\begin{lemma}\label{lem:Ptil}
  In the situation of the previous definition, there is a canonical
  equivalence
  $$
  \alg_{P_X}(\mathcal{C}_{X/}) \simeq \alg_P(\mathcal{C})_{X/},
  $$
  where $\alg_{P}(\mathcal{C})_{X/}  := \alg_{P}(\mathcal{C})
    \times_{\mathcal{C}} \mathcal{C}_{/X}.$
\end{lemma}
\begin{proof}
  Both $\infty$-categories are defined as pullbacks:
  \[
  \begin{tikzcd}
	\alg_{P_X}(\mathcal{C}_{X/}) \drpullback \ar[r] \ar[d] & 
	(\mathcal{C}_{X/})^{\Delta^1} \ar[d, "{(\txt{ev}_0,\txt{ev}_1)}"] \\
	\mathcal{C}_{X/} \ar[r, "{(P_X, \id)}"'] & \mathcal{C}_{X/}\times \mathcal{C}_{X/}
  \end{tikzcd}
\qquad\qquad
  \begin{tikzcd}
	\alg_P(\mathcal{C})_{X/} \drpullback \ar[r] \ar[d] & 
	\alg_P(\mathcal{C}) \drpullback \ar[r] \ar[d] &
	\mathcal{C}^{\Delta^1} \ar[d, "{(\txt{ev}_0,\txt{ev}_1)}"] \\
	\mathcal{C}_{X/} \ar[r, "u_X"'] & \mathcal{C} \ar[r, "{(P, \id)}"'] & 
	\mathcal{C}\times \mathcal{C}  .
  \end{tikzcd}
  \]
  The bottom functors can be factored into three steps, respectively:
  \[
  \begin{tikzcd}
	\mathcal{C}_{X/} \ar[r, "{(u_X, \id)}"'] 
	& \mathcal{C}\times \mathcal{C}_{X/} \ar[r, "{P\times \id}"']
	& \mathcal{C}\times \mathcal{C}_{X/} \ar[r, "{a_X \times \id}"']
	& \mathcal{C}_{X/}\times \mathcal{C}_{X/},
  \end{tikzcd}\]
\[  \begin{tikzcd}
	\mathcal{C}_{X/} \ar[r, "{(u_X, \id)}"'] 
	& \mathcal{C}\times \mathcal{C}_{X/} \ar[r, "{P \times\id}"']
	& \mathcal{C}\times \mathcal{C}_{X/} \ar[r, "{\id \times u_X}"']
	& \mathcal{C}\times \mathcal{C} .
  \end{tikzcd}
  \]
  The first two steps are the same, and for the last step we have
  an equivalence of pullbacks
  \[
  \begin{tikzcd}
  (\mathcal{C}_{X/})^{\Delta^1} \ar[d] 
  & \ar[l] \dlpullback \cdot \drpullback \ar[r] \ar[d] & 
  \mathcal{C}^{\Delta^1} \ar[d] \\
  \mathcal{C}_{X/}\times \mathcal{C}_{X/} & \ar[l, "{a_X \times \id}"] 
  \mathcal{C}\times \mathcal{C}_{X/} \ar[r, "{\id\times u_X}"'] & \mathcal{C}\times \mathcal{C}
  \end{tikzcd}
  \]
  since $a_X$ is left adjoint to $u_X$.
\end{proof}

\begin{lemma}\label{lem:Upres}
  Suppose $\mathcal{C}$ has colimits of shape $K$ and $P\colon \mathcal{C} 
  \to \mathcal{C}$ preserves them.  Then $\alg_P(\mathcal{C})$ has colimits
  of shape $K$ and the forgetful functor $U\colon \alg_P(\mathcal{C}) \to 
  \mathcal{C}$ preserves and reflects them.
\end{lemma}
\begin{proof}
  The \icat{} $\alg_P(\mathcal{C})$ is defined as a pullback
  \[
  \begin{tikzcd}[row sep=3em,column sep=2.5em]
  \alg_{P}(\mathcal{C}) \drpullback \ar[r] \ar[d, "U"'] & 
  \mathcal{C}^{\Delta^{1}} \ar[d, "{(\txt{ev}_0,\txt{ev}_1)}"] \\
  \mathcal{C} \ar[r, "{(P,\id)}"'] & \mathcal{C} \times \mathcal{C}.
  \end{tikzcd}
  \]
  Here the \icats{} $\mathcal{C}^{\Delta^{1}}$, $\mathcal{C}$, and
  $\mathcal{C} \times \mathcal{C}$ have colimits of shape $K$, and the
  functors $(\txt{ev}_{0},\txt{ev}_{1})$ and $(P, \id)$ preserve
  them. It therefore follows from \cite{HTT}*{Lemma 5.4.5.5} that
  $\alg_{P}(\mathcal{C})$ has colimits of shape $K$, and that a
  diagram $K^{\triangleright} \to \alg_{P}(\mathcal{C})$ is a colimit
  \IFF{} its images in $\mathcal{C}$ and $\mathcal{C}^{\Delta^{1}}$
  are colimits. Since the functor $(\txt{ev}_0,\txt{ev}_1)$ preserves
  and reflects
  colimits, this is equivalent to the image under $U$ being a colimit.
\end{proof}

\begin{propn}\label{propn:algPmonadic}
  Suppose $\mathcal{C}$ is an \icat{} with filtered colimits and binary
  coproducts, and $P
  \colon \mathcal{C} \to \mathcal{C}$ is an endofunctor that preserves
  filtered colimits. Then the forgetful functor $U \colon \alg_{P}(\mathcal{C})
  \to \mathcal{C}$ has a left adjoint, and the resulting adjunction is
  monadic.
\end{propn}
\begin{proof}
  To see that $U$ has a left adjoint, it suffices by \cite{GepnerKock}*{Corollary 2.3}
  to show that for every $X \in \mathcal{C}$ the \icat{}
  \[ \alg_{P}(\mathcal{C})_{X/}  := \alg_{P}(\mathcal{C})
    \times_{\mathcal{C}} \mathcal{C}_{X/}\]
  has an initial object. But $\alg_{P}(\mathcal{C})_{X/}$ can be
  identified with $\alg_{P_X}(\mathcal{C}_{X/})$ by
  Lemma~\ref{lem:Ptil}, where $P_X = a_X \circ P \circ u_X$ as in
  Definition~\ref{defn:Ptil}. Moreover, the functor $P_X$ preserves
  filtered colimits (indeed $u_X$ preserves filtered colimits by the
  dual of Lemma~\ref{lem:slicecreates}, $P$ preserves filtered
  colimits by assumption, and $a_X$ is a left adjoint).  Therefore
  $\alg_{P_{X}}(\mathcal{C}_{X/})$ has an initial object by
  Proposition~\ref{propn:initalg} since $\mathcal{C}_{X/}$ obviously
  has an initial object, and has filtered colimits by the dual of
  Lemma~\ref{lem:slicecreates}.

  To show that the resulting adjunction is monadic, we apply the
  Lurie--Barr--Beck monadicity theorem \cite{HA}*{Theorem 4.7.3.5}.
  For this we must show that $U$ detects equivalences, which is clear,
  and that $\alg_{P}(\mathcal{C})$ has colimits of $U$-split
  simplicial diagrams, and $U$ preserves these.  Consider a $U$-split
  simplicial diagram
  $A_{\bullet} \colon \Dop \to \alg_{P}(\mathcal{C})$.
  By (the proof of) \cite{HTT}*{Lemma 5.4.5.5} it is enough to show
  that the images of $A_{\bullet}$ in $\mathcal{C}$ and
  $\mathcal{C}^{\Delta^{1}}$ have colimits, and these are preserved by
  the functors
  $(P, \id) \colon \mathcal{C} \to \mathcal{C} \times \mathcal{C}$ and
  $(\txt{ev}_{0}, \txt{ev}_{1}) \colon \mathcal{C}^{\Delta^{1}}\to
  \mathcal{C} \times \mathcal{C}$.
  Since $A_{\bullet}$ is $U$-split, it follows from \cite{HA}*{Remark
    4.7.2.3} that $U(A_{\bullet})$ has a colimit $C \in \mathcal{C}$,
  and this is preserved by any functor, in particular by $P$. Thus the
  map $PC \simeq \colim PU(A_{\bullet}) \to \colim
  U(A_{\bullet}) \simeq C$ induced by the algebra structure maps in
  $A_{\bullet}$ is a colimit in $\mathcal{C}^{\Delta^{1}}$. Thus 
  $A_\bullet$ has a colimit in $\alg_P(\mathcal{C})$ and 
  $U$ preserves it.
\end{proof}

\begin{notation}
  For any \icat{} $\mathcal{C}$ we write $\Mnd(\mathcal{C})$ for the
  \icat{} of monads on $\mathcal{C}$, defined as the \icat{}
  $\Alg(\End(\mathcal{C}))$ of associative algebras in
  $\End(\mathcal{C})$ with respect to the monoidal structure given by
  composition, as in \cite{HA}*{\S 4.7.1}. If $T \in
  \Mnd(\mathcal{C})$ is a monad on $\mathcal{C}$, we write
  $\Alg_{T}(\mathcal{C})$ for the \icat{} of $T$-algebras in
  $\mathcal{C}$ (which can be defined as the \icat{}
  $\LMod_{T}(\mathcal{C})$ of left $T$-modules via the action of
  $\End(\mathcal{C})$ on $\mathcal{C}$). Note that we write lowercase
  $\alg$ for Lambek algebras for an endofunctor 
  and uppercase $\Alg$ for algebras for a monad.
\end{notation}
\begin{defn}
  If $\mathcal{C}$ is an \icat{} with filtered colimits and
  binary coproducts, and $P \in \End(\mathcal{C})$ is a
  filtered-colimit-preserving endofunctor on $\mathcal{C}$, we write
  $\freemonad{P} \in \Mnd(\mathcal{C})$ for the monad
  associated to the monadic adjunction 
  \[ \mathcal{C} \rightleftarrows \alg_{P}(\mathcal{C})\]
  of Proposition~\ref{propn:algPmonadic}; this exists by
  \cite{HA}*{Proposition 4.7.3.3}.
\end{defn}

With this notation, we have:
\begin{cor}\label{cor:alg=Alg}
  There is a canonical equivalence
  \[
  \alg_P(\mathcal{C}) \simeq \Alg_{\freemonad{P}}(\mathcal{C})
  \]
  over $\mathcal{C}$. \qed
\end{cor}

The following result shows that $\freemonad{P}$ is the \emph{free} monad 
on $P$:
\begin{propn}\label{propn:freemnd}
  Suppose $\mathcal{C}$ is an \icat{} with filtered colimits and
  binary coproducts, and $P \colon \mathcal{C} \to \mathcal{C}$ is an
  endofunctor that preserves filtered colimits. 
  Then for every monad $T$ on $\mathcal{C}$ the morphism
  \[ \Map_{\Mnd(\mathcal{C})}(\freemonad{P}, T) \to
  \Map_{\txt{End}(\mathcal{C})}(P, T) \] induced by the natural
  transformation $P \to \freemonad{P}$ is an equivalence.
\end{propn}

The final ingredient needed for the proof of Proposition~\ref{propn:freemnd}
is the following observation:
\begin{propn}\label{propn:algPmap}
  For any endofunctor $P \colon \mathcal{C} \to
  \mathcal{C}$ and any adjunction \[L : \mathcal{C} \rightleftarrows
  \mathcal{D} : R\] there is a natural equivalence
  \[ \Map_{/\mathcal{C}}(\mathcal{D}, \alg_{P}(\mathcal{C})) \simeq
  \Map_{\End(\mathcal{C})}(P, RL).\]
\end{propn}
\begin{proof}
  It is enough to establish
  \[
  \Map_{/\mathcal{C}}(\mathcal{D}, \alg_{P}(\mathcal{C})) \simeq
  \Map_{\Fun(\mathcal{D}, \mathcal{C})}(PR, R) ,
  \]
  because the latter space is equivalent to $\Map_{\End(\mathcal{C})}(P, RL)$
  by adjunction: precomposing with the adjunction $L \dashv R$ we get for
  any \icat{} $\mathcal{X}$ an adjunction
  \[ R^{*} : \Fun(\mathcal{D}, \mathcal{X}) \rightleftarrows
    \Fun(\mathcal{C}, \mathcal{X}) : L^{*}, \]
  with $R^{*}$ left adjoint to $L^{*}$, and hence a natural equivalence of
  mapping spaces
  \[ \Map_{\Fun(\mathcal{D}, \mathcal{X})}(FR, G) \simeq
    \Map_{\Fun(\mathcal{C}, \mathcal{X})}(F, GL).\]
  Consider the diagram
  \[
    \begin{tikzcd}
      \mathcal{D} \ar[rdd, bend right=16, "R"'] \ar[rd, dotted, "\Phi"] 
      \ar[rrd, dotted, bend left=16, "\Psi"] & \\
	& \alg_{P}(\mathcal{C}) \drpullback \ar[r] \ar[d] & \mathcal{C}^{\Delta^1} \ar[d, "{(\txt{ev}_0,\txt{ev}_1)}"] \\
	& \mathcal{C} \ar[r, "{(P,\id)}"'] & \mathcal{C}\times\mathcal{C}.
    \end{tikzcd}
  \]
  Since $\alg_P(\mathcal{C})$ is defined as a pullback, we see that
  giving $\Phi \in \Map_{/\mathcal{C}}(\mathcal{D}, \alg_{P}(\mathcal{C}))$ 
  is equivalent to giving $\Psi$, which amounts precisely 
  to giving a natural transformation from $PR$ to $R$,
  as required.
\end{proof}

\begin{proof}[Proof of Proposition~\ref{propn:freemnd}.]
  Combine the equivalences of 
  Proposition~\ref{propn:algPmap}, Corollary~\ref{cor:alg=Alg}, and
  Corollary~\ref{cor:mndeq}.
\end{proof}

\begin{defn}
  Let $\mathcal{C}$ be an \icat{} with filtered colimits. We write
  $\End^{\omega}(\mathcal{C})$ for the full subcategory of
  $\End(\mathcal{C})$ spanned by the endofunctors that preserve
  filtered colimits. These are closed under composition, and so we get
  an \icat{} $\Mnd^{\omega}(\mathcal{C}) := \Alg(\End^{\omega}(\mathcal{C}))$, 
  the full subcategory of
  $\Mnd(\mathcal{C})$ spanned by the monads that preserve
  filtered colimits.
\end{defn}

\begin{cor}\label{cor:freemonadC}
  Suppose $\mathcal{C}$ is an \icat{} with filtered colimits and
  binary coproducts. Then the forgetful functor
  $\Mnd^{\omega}(\mathcal{C}) \to \End^{\omega}(\mathcal{C})$ has a
  left adjoint.
\end{cor}
\begin{proof}
  By Proposition~\ref{propn:freemnd}, for each
  $P\in \End^{\omega}(\mathcal{C})$  the \icat{}
  $\Mnd^{\omega}(\mathcal{C})_{P/}$ has an initial object, namely the
  free monad $\freemonad{P}$ on $P$. This implies that the forgetful
  functor has a left adjoint, which assigns to every endofunctor $P$ its
  free monad $\freemonad{P}$.
\end{proof}

Our next goal is to prove that this free monad adjunction is itself
monadic, at least if we impose further restrictions on the monads:
\begin{defn}
  Suppose $\mathcal{C}$ is an \icat{} with sifted colimits. We write
  $\End^{\sigma}(\mathcal{C})$ for the full subcategory of
  $\End(\mathcal{C})$ spanned by the endofunctors that preserve sifted
  colimits, and let $\Mnd^{\sigma}(\mathcal{C})$ denote the full
  subcategory of $\Mnd(\mathcal{C})$ of monads whose underlying
  endofunctors preserve sifted colimits.
\end{defn}

\begin{lemma}\label{lem:Pbarsift}
  Suppose $\mathcal{C}$ is an \icat{} with sifted colimits and binary
  coproducts. If $P \colon \mathcal{C} \to \mathcal{C}$ preserves
  sifted colimits, then the underlying endofunctor of the free monad
  $\freemonad{P}$ on $P$ also preserves sifted colimits.
\end{lemma}
\begin{proof}
  It suffices to show that the forgetful functor
  $U \colon \alg_{P}(\mathcal{C}) \to \mathcal{C}$ preserves sifted
  colimits, but this is a special case of Lemma~\ref{lem:Upres}.
\end{proof}

Thus if $\mathcal{C}$ is an \icat{} with sifted colimits and binary
coproducts, then the free monad functor restricts to give an
adjunction
\[ F : \End^{\sigma}(\mathcal{C}) \rightleftarrows
\Mnd^{\sigma}(\mathcal{C}) : U.\]
\begin{propn}\label{propn:Mndhasifted}
  Suppose $\mathcal{C}$ is an \icat{} with sifted colimits. Then
  $\Mnd^{\sigma}(\mathcal{C})$ has sifted colimits, and these are
  preserved by $U$.
\end{propn}
\begin{proof}
  $\Mnd^{\sigma}(\mathcal{C})$ is the \icat{} of associative algebras
  in the monoidal \icat{} $\End^{\sigma}(\mathcal{C})$, where the
  tensor product, i.e.\ composition, commutes with sifted colimits in
  each variable (since we are considering endofunctors that preserve
  these colimits). The result is therefore a special case of
  \cite{HA}*{Proposition 3.2.3.1}.
\end{proof}

\begin{cor}\label{cor:freesiftedmonad}
  Suppose $\mathcal{C}$ is an \icat{} with sifted colimits and binary
  coproducts. Then the adjunction
  \[ F : \End^{\sigma}(\mathcal{C}) \rightleftarrows
  \Mnd^{\sigma}(\mathcal{C}) : U.\]
  is monadic.
\end{cor}
\begin{proof}
  We already know from Proposition~\ref{propn:Mndhasifted} that
  $\Mnd^{\sigma}(\mathcal{C})$ has all sifted colimits and that $U$
  preserves these. It therefore suffices by \cite{HA}*{Theorem 4.7.3.5} to
  show that $U$ detects equivalences, which follows from
  \cite{HA}*{Lemma 3.2.2.6}.
\end{proof}

We end this subsection by noting that, under rather restrictive
hypotheses on $\mathcal{C}$, this implies that $\Mnd^{\sigma}(\mathcal{C})$ is
presentable:
\begin{defn}\label{defn:siftpres}
  Let us say that an \icat{} $\mathcal{C}$ is
  \emph{compact projectively generated} if it is of the form
  $\mathcal{P}_{\Sigma}(\mathcal{C}_{0})$ for some small \icat{}
  $\mathcal{C}_{0}$ with finite coproducts, using the notation of
  \cite{HTT}*{\S 5.5.8}.
\end{defn}
\begin{remark}\label{rmk:siftpres}
  The only reason for introducing this notion is that it implies that
  $\End^{\sigma}(\mathcal{C})$ is presentable. We believe this should
  be true for any presentable \icat{} $\mathcal{C}$, but we will not
  attempt to prove this as it is not needed for our purposes.
\end{remark}

\begin{cor}\label{cor:mndsigmapres}
  Suppose $\mathcal{C}$ is a compact projectively generated \icat{}. Then
  $\Mnd^{\sigma}(\mathcal{C})$ is a presentable \icat{}.
\end{cor}
\begin{proof}
  Since $\mathcal{C}$ is compact projectively generated, the \icat{}
  $\End^{\sigma}(\mathcal{C})$ is equivalent to $\Fun(\mathcal{C}_{0},
  \mathcal{C})$ where $\mathcal{C}_{0}$ is a small \icat{}, and so
  this \icat{} is presentable. Moreover, the \icat{}
  $\Mnd^{\sigma}(\mathcal{C})$ has sifted colimits by
  Proposition~\ref{propn:Mndhasifted} and these are preserved by the
  forgetful functor to $\End^{\sigma}(\mathcal{C})$. Applying
  \cite{enr}*{Lemma A.5.8, Proposition A.5.9} to the 
  adjunction 
  \[ F : \End^{\sigma}(\mathcal{C}) \rightleftarrows
  \Mnd^{\sigma}(\mathcal{C}) : U,\]
  which is monadic by Corollary~\ref{cor:freesiftedmonad}, it follows
  that $\Mnd^{\sigma}(\mathcal{C})$ is presentable.
\end{proof}

\subsection{An Explicit Description of the Free Monad}\label{subsec:freecolim}

We will now give a more explicit description of the free monad
$\freemonad{P}$ as the colimit of a sequence of functors.
\begin{defn}\label{defn:Pn}
  For $\mathcal{C}$ an $\infty$-category with filtered colimits and binary 
  coproducts, and $P\colon \mathcal{C} \to \mathcal{C}$ an endofunctor that
  preserves filtered colimits,
  we will recursively define endofunctors $P_n$ and natural transformations
  \[
  \begin{tikzcd}
	P_0 \ar[r, "f_1"] 
	& P_1 \ar[r, "f_2"] 
	& P_2 
	\ar[r, "f_3"] 
	& \cdots.
  \end{tikzcd}
  \]
  Here $P_0 := \id_{\mathcal{C}}$, and recursively $P_{n+1} :=
  \id_{\mathcal{C}} \amalg (P \circ P_n) $.  For the natural
  transformations, $f_1$ is the coproduct inclusion, and $f_{n+1} :=
  \id_{\mathcal{C}} \amalg P (f_n) $.
\end{defn}

\begin{propn}\label{propn:freemonadind}
  With notation as above, we have a natural equivalence
  $$
  \colim_{n\to\infty} P_n \isoto \freemonad{P}.
  $$
\end{propn}

\begin{lemma}
  Let $F \colon \mathcal{C} \to \alg_{P}(\mathcal{C})$ be the left
  adjoint to the forgetful functor $U \colon \alg_{P}(\mathcal{C}) \to
  \mathcal{C}$. The composite of $F$ with the forgetful functor
  $\alg_{P}(\mathcal{C}) \to \mathcal{C}^{\Delta^{1}}$ corresponds to
  a natural transformation $\phi \colon P\freemonad{P} \to
  \freemonad{P}$. The induced transformation
  \[  \id_{\mathcal{C}} \amalg P\freemonad{P} \xto{\eta | \phi}
    \freemonad{P},\]
  where $\eta$ is the unit for the adjunction $F \dashv U$ and
  $\eta|\phi$ means $\eta$ on the first summand and $\phi$ 
  on the second summand,
  is an equivalence.
\end{lemma}
\begin{proof}
  Evaluating at $X \in \mathcal{C}$ the map $X \amalg
  P\freemonad{P}(X) \to \freemonad{P}(X)$ is the structure map
  $P_{X}\freemonad{P}X \to \freemonad{P}X$ exhibiting $\freemonad{P}X$
  as the initial $P_{X}$-algebra, which we know is an equivalence.
\end{proof}

\begin{proof}[Proof of Proposition~\ref{propn:freemonadind}]
  To define a natural transformation $\colim_{n\to\infty} P_n \to
  \freemonad{P}$ we recursively define natural transformations
  $\pi_{n} \colon P_{n}
  \to \freemonad{P}$ and equivalences $\pi_{n} \circ f_{n} \simeq
  \pi_{n-1}$. 

  We start by setting $\pi_{0} := \eta \colon \id_{\mathcal{C}} \to
  \freemonad{P}$, and then given $\pi_{n}$ we define $\pi_{n+1}$ as
  the composite
  \[ P_{n+1} = \id \amalg PP_{n} \xto{\id \amalg P\pi_{n}} \id \amalg
  P\freemonad{P} \xto{\eta | \phi} \freemonad{P}.\]

  We then have a commutative diagram
\[
\begin{tikzcd}
  \id \arrow[bend left=20]{drr}{\eta}  \arrow[swap, hookrightarrow]{d}{f_{1}}\\
  \id \amalg P \arrow[swap]{r}{\id \amalg P\eta}& \id \amalg P \freemonad{P}
  \arrow[swap]{r}{\eta | \phi} & \freemonad{P},
\end{tikzcd}
\]
which gives an equivalence $\pi_{1} \circ f_{1} \simeq \pi_{0}$. For
$n > 0$, the composite 
$\pi_{n+1} \circ f_{n+1}$ is equivalent to the composite
\[ \id \amalg PP_{n-1} \xto{\id \amalg Pf_{n-1}} \id \amalg PP_{n}
\xto{\id \amalg P\pi_{n}} \id \amalg P\freemonad{P} \xto{\eta \amalg
  \phi} \freemonad{P}.\] We can rewrite this as
\[ \id \amalg PP_{n-1} \xto{\eta | \phi \circ P(\pi_{n}f_{n})}
\freemonad{P},\]
which is equivalent to $\eta |  \phi\pi_{n-1}$ --- this is the
same as $\pi_{n}$ by definition, giving the required equivalence
$\pi_{n+1} \circ f_{n+1} \simeq \pi_{n}$.

It remains to show that the induced map
$\colim_{n} P_{n}X \to \freemonad{P}X$ is an equivalence for all
$X \in \mathcal{C}$.  By definition, $\freemonad{P}(X)$ is the
underlying object in $\mathcal{C}$ of the initial $P_X$-algebra
(notation as in Definition~\ref{defn:Ptil}), in turn described in
Proposition~\ref{propn:initalg} as the colimit of $P_X^n (\id_X)$ as
$n\to \infty$.  But on underlying objects we clearly have
$P_n X \simeq P_X^n (\id)$, and under this identification, the
transition maps $f_n \colon P_{n-1}X \to P_n X$ are the iterated
$P_X$-coalgebra structure maps
$c_n \colon P_X^{n-1}(\id) \to P_X^n (\id)$, as in
Definition~\ref{defn:cobar}.  Hence
$\freemonad P \simeq \colim_{n\to \infty} P_n$ as required.
\end{proof}

\begin{lemma}\label{lem:en}
  For $(A,a)$ a $P$-algebra, the underlying map of the counit, 
  $U\epsilon_A\colon UFUA 
  \to UA$ is naturally identified with the colimit of the sequence of
  maps $e_n \colon P_n A \to A$, defined recursively with
  $e_0 \colon A \to A$
  the identity, and $e_{n+1}$ defined as the composite
  $$
  P_{n+1} A = A \amalg  P P_n A \xrightarrow{\id\amalg P(e_n)} A\amalg PA \xrightarrow{\id|a} A.  
  $$
\end{lemma}

\begin{proof}
  We know that $FUA$ is the underlying object in $\alg_P(\mathcal{C})$
  of $\colim_{n\to \infty} P_A^n (\id_A)$, the
  initial $P_A$-algebra.  Since $(A,a)$ is a $P$-algebra, the morphism
  $\id_A$ becomes naturally a $P_A$-algebra, hence
  there is a unique homomorphism of $P_A$-algebras $\colim_{n\to
  \infty} P_A^n (\id_A) \to \id_A$.  The image under the forgetful functor
  $\alg_{P_A}(\mathcal{C}_{A/}) \to \alg_P(\mathcal{C})$ is the counit
  $\epsilon_A$.  By Proposition~\ref{propn:OmegaCA}, this corresponds to
  the unique $P_A$-twisting morphism $\id_A \to \id_A$.  By (the proof
  of) Lemma~\ref{lem:PCA}, the counit $\epsilon_A
  \colon \colim_{n\to \infty} P_A^n (\id_A) \to \id_A$ is induced by the
  sequence of twisting morphisms
  $$
  e_n \colon P_A^n (\id_A)  \to \id_A,
  $$
  where $e_0 \colon \id_A \to \id_A$
  is the identity map, and recursively
  $e_{n+1}$ is defined as the composite
  $$
  (A\amalg P)^{n+1}(\id_A) \simeq (A\amalg P)(A\amalg P)^n(\id_A)
  \xrightarrow{(A\amalg P)(e_n)} (A\amalg P)(A) \xrightarrow{\id|a} A .
  $$
  The forgetful functor $U\colon \alg_{P_A}(\mathcal{C}_{A/}) \to
  \mathcal{C}$ preserves filtered colimits by Lemma~\ref{lem:Upres}, so
  under the identifications $P_n A \simeq P_A^n(\id_A)$, valid in
  $\mathcal{C}$, this is precisely the sequence of maps of the statement.
\end{proof}

\begin{construction}\label{constr:PnPm}
  We define natural transformations, for  $m,n\geq 0$
  $$
  \mu_{m,n}\colon P_n \circ P_m \to P_{m+n}
  $$
  recursively as follows.  For the base case $n=0$ (all $m\geq 0$) we take
  $\mu_{m,0} \colon \id_{\mathcal{C}}\circ P_m
  \to P_m$ to be the identity natural transformation.
  Assuming $\mu_{m,n}\colon P_n \circ P_m \to P_{m+n}$ 
  defined, define $\mu_{m,n+1}$ to be the composite
  $$
  P_{n+1} \circ P_m
  \simeq P_m \amalg (P \circ P_n \circ P_m)
    \xrightarrow{\id \amalg P(\mu_{m,n})}
  P_m \amalg P\circ P_{m+n}
  \to
  P_m \amalg P_{m+n+1} \to P_{m+n+1} .
  $$
  Here the second arrow is the sum inclusion $P \circ P_{m+n} \to
  \id_{\mathcal{C}} \amalg (P \circ P_{m+n}) \simeq P_{m+n+1}$ and the third
  adds the natural transformation $P_m \to P_{m+n+1}$ which is a composite
  of $f_k$ in the defining chain.
\end{construction}

\begin{lemma}\label{lem:mu++}
  The natural transformations $\mu_{m,n}$ are compatible with the 
  transition maps $f_k$ in both variables. More precisely,
  we have commutative diagrams for all $m,n\geq 0$
  \[
  \begin{tikzcd}
	P_n \circ P_m \ar[r, "\mu_{m,n}"] \ar[d, "f_{n+1}P_{m}"'] 
	\arrow[dr, phantom, "{(1)}"]
	& P_{m+n} \ar[d, "f_{m+n+1}"] \\
	P_{n+1}\circ P_m \ar[r, "\mu_{m,n+1}"'] & P_{m+n+1},
  \end{tikzcd}
  \qquad
  \begin{tikzcd}
	P_n \circ P_m \ar[r, "\mu_{m,n}"] \ar[d, "P_n(f_{m+1})"']
	\arrow[dr, phantom, "{(2)}"]
	& P_{m+n} \ar[d, "f_{m+1+n}"] \\
	P_n\circ P_{m+1} \ar[r, "\mu_{m+1,n}"'] & P_{m+1+n}   .
  \end{tikzcd}
  \]
\end{lemma}
\begin{proof}
  For both statements, the proof is by induction on $n$, the $n=0$ 
  cases being trivial.  Assuming we have the square (1), we also have
  \[
  \begin{tikzcd}
	P_{n+1}\circ P_m \ar[r, equal] \ar[d, "f_{n+2}P_{m}"'] &[-3ex]
	P_m \amalg P (P_n P_m) \ar[d, "\id \amalg P(f_{n+1}P_{m})"'] 
	\ar[rr, "\id\amalg P(\mu_{m,n})"] 	\arrow[drr, phantom, "{(3)}"]
&& 
	P_m \amalg P P_{m+n} \ar[d, "\id\amalg P(f_{m+n+1})"] \ar[r] &
	P_m \amalg P_{m+n+1} \ar[d] \ar[r] &
	P_{m+n+1} \ar[d, "f_{m+n+2}"] \\
	P_{n+2}\circ P_m \ar[r, equal] &
	P_m \amalg P(P_{n+1} P_m) 
	\ar[rr, "\id\amalg P(\mu_{m,n+1})"'] 
&&
	P_m \amalg P P_{m+n+1} \ar[r] &
	P_m \amalg P_{m+n+2} \ar[r] &
	P_{m+n+2}
  \end{tikzcd}
  \]
  Indeed, square (3) commutes by induction (the right summand is
  $P$ applied to the square (1) of the induction hypothesis) and
  the two following squares obviously commute.  The horizontal composites
  are precisely $\mu_{m,n+1}$ and $\mu_{m,n+2}$.

  Assuming that square (2) commutes, we also have
  \[
  \begin{tikzcd}
	P_{n+1}\circ P_m \ar[r, equal] \ar[d, "P_n(f_{m+1})"'] &[-3ex]
	P_m \amalg P (P_n P_m) \ar[d, "f_{m+1} \amalg PP_n(f_{m+1})"'] 
	\ar[rr, "\id\amalg P(\mu_{m,n})"] \arrow[drr, phantom, "{(4)}"]
	&& 		
	P_m \amalg P P_{m+n} \ar[d, "f_{m+1}\amalg P(f_{m+1+n})"] \ar[r] 
	&[-3ex]
	P_m \amalg P_{m+n+1} \ar[d] \ar[r] &[-3ex]
	P_{m+n+1} \ar[d, "f_{m+n+2}"] \\
	P_{n+1}\circ P_{m+1} \ar[r, equal] &
	P_{m+1} \amalg P(P_n P_{m+1}) 
	\ar[rr, "\id\amalg P(\mu_{m+1,n})"'] &&
	P_{m+1} \amalg P P_{m+1+n} \ar[r] &
	P_{m+1} \amalg P_{m+n+2} \ar[r] &
	P_{m+n+2}
  \end{tikzcd}
  \]
  Indeed, the square (4) commutes by induction (the right summand is
  $P$ applied to the square (2) of the induction hypothesis) and
  the two following squares obviously commute.  The horizontal composites
  are precisely $\mu_{m,n+1}$ and $\mu_{m+1,n+1}$.  
\end{proof}

\begin{lemma}\label{lem:mun}
	The colimit of the sequence of natural transformations
	$$
	\mu_{m,n} \colon P_n \circ P_m \to P_{m+n}
	$$
	for $m\to \infty$ is naturally identified with the maps
	$$
	e_n \colon P_n \circ \freemonad{P} \to \freemonad{P} 
	$$
	of Lemma~\ref{lem:en}.
\end{lemma}
\begin{proof}
  Induction on $n$.  The case $n=0$ is clear, as both maps are the identity.
  Suppose $\displaystyle \colim_{m\to\infty} \mu_{m,n} \simeq e_n$. 
  Write down $\mu_{m,n+1}$
  according to the recursive definition:
  $$
  P_{n+1} P_m  
  \simeq P_m \amalg (P  P_n  P_m)
    \xrightarrow{\id \amalg P(\mu_{m,n})}
  P_m \amalg P P_{m+n}
  \to
  P_m \amalg P_{m+n+1} \to P_{m+n+1} .
  $$
  Take the colimit as $m\to\infty$ to find
  $$
  P_{n+1} \freemonad{P}
  \simeq \freemonad{P} \amalg (P P_n \freemonad{P})
    \xrightarrow{\id \amalg P(e_n)}
  \freemonad{P} \amalg P\circ \freemonad{P}
  \to
  \freemonad{P} \amalg \freemonad{P} \to \freemonad{P} ,
  $$
  by induction, using that all the functors commute with filtered 
  colimits.  But this is precisely the recursive description of
  $e_{n+1}$, given in \ref{lem:en}.
\end{proof}

\begin{propn}\label{propn:mudescript}
  The multiplication $\mu\colon \freemonad{P}\circ
  \freemonad{P}\to\freemonad{P}$ is the colimit, for $n\to \infty$, of the
  natural transformations $\mu_n \colon P_n \circ P_n \to P_{2n}$ of
  Construction~\ref{constr:PnPm}.  The unit $\eta\colon \id_{\mathcal{C}}
  \to \freemonad{P}$ is the colimit of the sequence of natural
  transformations $\eta_n \colon \id_{\mathcal{C}} \to P_n$.
\end{propn}
\begin{proof}
  Thanks to the compatibilities with the transition maps of
  Lemma~\ref{lem:mu++}, we can compute the colimit first by holding $n$ 
  fixed. Lemma~\ref{lem:mun} tells us that for each $n$ fixed, the 
  $m\to\infty$ colimit
  is the map $e_n\colon P_n \freemonad{P} \to \freemonad{P}$, and 
  Lemma~\ref{lem:en} then tells us that the $n\to\infty$ colimit of those
  is the monad multiplication.
\end{proof}

\subsection{Free Monads in Families}\label{subsec:freemndfamily}
In this section we will extend our results on free monads to the
setting of monads and endofunctors on varying base \icats{}. In
\S\ref{subsec:monad},
we review results of \cite{adjmnd} that
lead to a commutative triangle
\opctriangle{\Mnd(\CATI)_{\lax}}{\End(\CATI)_{\lax}}{\CatI,}{}{}{}
where $\Mnd(\CATI)_{\lax}$ is an \icat{} of monads and lax morphisms,
$\End(\CATI)_{\lax}$ is an \icat{} of endofunctors and lax morphisms, and
the functors to $\CatI$ send monads and endofunctors to the \icat{}
they are defined on. This is given fibrewise by the forgetful functor $\Alg(\End(\mathcal{C}))^{\op}
 \to \End(\mathcal{C})^{\op}$.
We then show
in Corollary~\ref{cor:Mndradjcocart} that if we restrict to the
subcategory $\CatI^{\radj}$ where the morphisms are right adjoint
functors, we get a commutative diagram
\opctriangle{\Mnd(\CATI)_{\lax}^{\radj}}{\End(\CATI)_{\lax}^{\radj}}{\CatI^{\radj},}{}{}{}
where the two downward functors are cocartesian fibrations, the
horizontal functor preserves cocartesian morphisms, and the right-hand
functor is also a cartesian fibration. We need to introduce notation
for a restricted version of these \icats{}:
\begin{defn}
  Let $\LCatISR$ be the \icat{} of compact projectively generated
  \icats{} (in the sense of Definition~\ref{defn:siftpres} --- but see
  Remark~\ref{rmk:siftpres}), with morphisms the functors that are
  right adjoints and preserve sifted colimits. Then we define
  $\LEndLSR$ to be the full subcategory of the pullback
  of $\LEnd(\LCATI)_{\lax} \to \LCatI$ to $\LCatISR$ spanned by
  the endofunctors that preserve sifted colimits; we also define
  $\LMndLSR$ similarly.
\end{defn}

\begin{propn}\label{propn:Mndsigmacocart}
  There is a commuting diagram
  \opctriangle{\LMndLSR}{\LEndLSR}{\LCatISR,}{}{}{}
  where the two downward functors are cocartesian fibrations and the
  horizontal functor preserves cocartesian morphisms. Moreover, both
  the downward functors are also cartesian fibrations.
\end{propn}
\begin{proof}
  It is immediate from Corollary~\ref{cor:Mndradjcocart} that the
  downward functors are cocartesian fibrations and the horizontal
  functor preserves cocartesian morphisms: from the
  description of the cocartesian morphisms there it follows
  that these full subcategories contain the cocartesian morphisms
  whose sources lie in the subcategories. Similarly, the right-hand
  functor is a cartesian fibration. 

  It remains to prove that the functor $\LMndLSR \to \LCatISR$ is a
  cartesian fibration. Since we know it is a cocartesian fibration,
  this is equivalent to showing that the functor
  $\phi^{\op}_{!} \colon \Mnd^{\sigma}(\mathcal{C}) \to
  \Mnd^{\sigma}(\mathcal{D})$
  corresponding to the cocartesian pushforward along a map
  $\phi \colon \mathcal{C} \to \mathcal{D}$ has a left adjoint. Since
  the forgetful functor $\LMndLSR \to \LEndLSR$ preserves cocartesian
  morphisms we have a commutative square
  \csquare{\Mnd^{\sigma}(\mathcal{C})}{\Mnd^{\sigma}(\mathcal{D})}{\End^{\sigma}(\mathcal{C})}{\End^{\sigma}(\mathcal{D}).}{\phi^{\op}_!}{}{}{\phi^{\op}_!}
  Here we know that $\Mnd^{\sigma}(\mathcal{C})$ and
  $\Mnd^{\sigma}(\mathcal{D})$ are presentable \icats{} by
  Corollary~\ref{cor:mndsigmapres}. By the adjoint functor theorem it
  therefore suffices to show that the functor $\phi^{\op}_{!}$ is
  accessible and preserves limits. But the forgetful functor
  $\Mnd^{\sigma}(\mathcal{D}) \to \End^{\sigma}(\mathcal{D})$ is a
  monadic right adjoint by Corollary~\ref{cor:freesiftedmonad} and
  preserves sifted colimits by
  Proposition~\ref{propn:Mndhasifted}. Thus limits and sifted colimits
  are computed in $\End^{\sigma}(\mathcal{D})$, so it suffices
  to show that the composite
  $\Mnd^{\sigma}(\mathcal{C}) \to \End^{\sigma}(\mathcal{D})$
  preserves limits and is accessible. The same observations apply to
  the forgetful functor
  $\Mnd^{\sigma}(\mathcal{C}) \to \End^{\sigma}(\mathcal{C})$, so in
  the end it is enough to prove that
  $\phi_{!}^{\op} \colon \End^{\sigma}(\mathcal{C})
  \to \End^{\sigma}(\mathcal{D})$
  preserves limits and is accessible, or equivalently that this is a
  right adjoint. But this follows from the projection
  $\LEndLSR \to \LCatISR$
  being a cartesian and cocartesian fibration (or from the explicit
  description of the cocartesian morphisms).
\end{proof}

\begin{propn}\label{propn:freemonadomega}
  The forgetful functor $\LMndLSR
  \to \LEndLSR$ has a right adjoint that
  commutes with the projections to $\LCatISR$, which takes an
  endofunctor to its free monad.
\end{propn}
\begin{proof}
  By (the dual of) \cite{HA}*{Proposition 7.3.2.6} it suffices to show
  that the functor on fibres over each
  $\mathcal{C} \in \LCatISR$ has a right
  adjoint. But this functor can be identified with the forgetful
  functor
  $\Mnd^\sigma(\mathcal{C})^{\op} \to \End^\sigma(\mathcal{C})^{\op}$,
  so this follows from Corollary~\ref{cor:freemonadC}.
\end{proof}

We now wish to prove that this free monad adjunction is in fact monadic
(which we saw fibrewise in Corollary~\ref{cor:freesiftedmonad}):
\begin{thm}\label{thm:siftedfreemonadmonadic}
  The forgetful functor $\LMndLSR \to \LEndLSR$ has a right adjoint that
  commutes with the projections to $\LCatISR$, and
  the resulting adjunction is comonadic.
\end{thm}

To prove this  we  will use the following general observation:
\begin{propn}\label{propn:fibadjmnd}
  Suppose we have a diagram
  \opctriangle{\mathcal{C}}{\mathcal{D}}{\mathcal{B}}{U}{p}{q}
  where
  \begin{enumerate}[(1)]
  \item $p$ and $q$ are cocartesian fibrations,
  \item for $b \in \mathcal{B}$ the \icat{} $\mathcal{C}_{b}$ has
    geometric realizations,
  \item for $f \colon b \to b'$ in $\mathcal{B}$, the cocartesian
    pushforward functor $f_{!} \colon \mathcal{C}_{b} \to
    \mathcal{C}_{b'}$ preserves geometric realizations,
  \item $U$ has a left adjoint $F \colon \mathcal{D} \to \mathcal{C}$
    such that $pF \simeq q$,
  \item the adjunction $F \dashv U$ restricts in each fibre to an
    adjunction $F_{b} \dashv U_{b}$,
  \item $U_{b} \colon \mathcal{C}_{b} \to \mathcal{D}_{b}$ preserves
    geometric realizations,
  \item $U_{b}$ detects equivalences for every $b \in \mathcal{B}$.
  \end{enumerate}
  Then the adjunction $F \dashv U$ is monadic.
\end{propn}

\begin{remark}
  It follows from these assumptions that each adjunction $F_{b} \dashv
  U_{b}$ is monadic.
\end{remark}

\begin{remark}
  Our proof of this result follows the argument used to prove
  monadicity for enriched categories in \cite{WolffVCat}.
\end{remark}

\begin{lemma}\label{lem:bisimpcond}
  Let $\mathcal{C}$ be an \icat{} with all small colimits, and suppose $F \colon (\simp^{\op})^{\triangleright} \times
  (\simp^{\op})^{\triangleright} \to \mathcal{C}$ is a diagram such
  that for every $[n] \in \simp$ the diagrams $F|_{\{[n]\} \times
    (\simp^{\op})^{\triangleright}}$ and
  $F|_{(\simp^{\op})^{\triangleright} \times \{[n]\} }$ are colimit
  diagrams. Then the following are equivalent:
  \begin{enumerate}[(i)]
  \item the restriction $F_{\{\infty\} \times
    (\simp^{\op})^{\triangleright}}$ is a colimit diagram,
  \item  the restriction $F_{(\simp^{\op})^{\triangleright}
      \times \{\infty\}}$ is a colimit diagram,
  \item the commutative square
    \nolabelcsquare{F([0],[0])}{F(\infty,[0])}{F([0],\infty)}{F(\infty,\infty)}
    is a pushout,
  \item $F$ is the left Kan extension of its restriction to
    $\simp^{\op}\times \simp^{\op}$.
  \end{enumerate}
\end{lemma}
\begin{proof}
  Functoriality of left Kan extensions and some easy cofinality arguments.
\end{proof}

\begin{proof}[Proof of Proposition~\ref{propn:fibadjmnd}]
  By \cite{HA}*{Theorem 4.7.3.5} it remains to show that $\mathcal{C}$
  has colimits of $U$-split simplicial diagrams, these colimits
  are preserved by $U$, and $U$ detects equivalences.

  Let us first check that $U$ detects equivalences. Suppose therefore
  that $f \colon c \to c'$ is a morphism in $\mathcal{C}$ such that $Uf$ is an
  equivalence in $\mathcal{D}$. Then $g := qUf \simeq pf$ is an equivalence
  in $\mathcal{B}$. We can factor $f$ as $c \xto{\phi} g_{!}c \xto{f'}
  c'$ where $\phi$ is a $p$-cocartesian morphism over $g$ and $f'$ is
  a morphism in the fibre $\mathcal{C}_{p(c')}$. But then $\phi$ is an
  equivalence since it is cocartesian over the equivalence $g$, and
  $f'$ is an equivalence since by assumption (7) $U$ detects equivalences
  fibrewise over $\mathcal{B}$.

  Using assumptions (1), (2), and (3) we see by \cite{HTT}*{Corollary
    4.3.1.11} that $p$-colimits of simplicial 
  diagrams exist in $\mathcal{C}$. Moreover, by \cite{HTT}*{Proposition
    4.3.1.5} a $p$-colimit diagram whose underlying diagram in
  $\mathcal{B}$ is a colimit is a colimit diagram in
  $\mathcal{C}$. Thus $\mathcal{C}$ has colimits for simplicial
  diagrams whose underlying diagrams in $\mathcal{B}$ have colimits
  --- in particular, this holds for $U$-split simplicial diagrams,
  since by definition their underlying diagrams in $\mathcal{B}$ can
  be extended to  split simplicial diagrams, which are colimit
  diagrams by \cite{HTT}*{Lemma 6.1.3.16}. 

  Moreover, since $\simp^{\op}$ is weakly contractible, it
  follows from \cite{HTT}*{Proposition 4.3.1.10} that if $\phi \colon
  \simp^{\op} \to \mathcal{C}$ is a diagram in $\mathcal{C}_{b}$ for
  some $b$, then its colimit in $\mathcal{C}_{b}$ is also a colimit in
  $\mathcal{C}$.

  Suppose then that $\phi \colon  \simp^{\op} \to \mathcal{C}$ is a
  $U$-split diagram. Using the monad $T := UF$ we can extend this to a
  diagram $\Phi \colon \simp^{\op} \times
  (\simp^{\op})^{\triangleright} \to \mathcal{C}$, where
  $\Phi|_{\simp^{\op} \times \{\infty\}} \simeq \phi$ and
  $\Phi|_{\simp^{\op} \times \{[n]\}} \simeq FT^{n}U\phi$. The
  underlying diagram in $\mathcal{B}$ of each row $\Phi|_{\simp^{\op}
    \times \{[n]\}}$ is split, hence the rows all have colimits in
  $\mathcal{C}$. Let $\overline{\Phi} \colon (\simp^{\op})^{\triangleright}
  \times (\simp^{\op})^{\triangleright} \to \mathcal{C}$ denote the
  left Kan extension of $\Phi$. Observe that the column
  $\Phi|_{\{[n]\} \times (\simp^{\op})^{\triangleright}}$ is a free
  resolution of $\phi([n])$ in the fibre $\mathcal{C}_{b}$. It is
  therefore a colimit diagram in $\mathcal{C}_{b}$, and hence in
  $\mathcal{C}$. Thus by Lemma~\ref{lem:bisimpcond} the last column
  $\Phi|_{\{\infty\} \times (\simp^{\op})^{\triangleright}}$ is also a
  colimit diagram.

  Now consider $U\overline{\Phi}$. The rows
  $(U\overline{\Phi})|_{(\simp^{\op})^{\triangleright} \times \{[n]\}}$ can
  all be extended to split simplicial diagrams, and are therefore
  colimits in $\mathcal{D}$. The columns $(U\overline{\Phi})|_{\{[n]\}
    \times (\simp^{\op})^{\triangleright}}$ can similarly be extended
  to split simplicial diagrams (in a single fibre) so they are also
  colimit diagrams. Finally, the underlying diagram in $\mathcal{B}$
  is constant, so the last column $(U\overline{\Phi})|_{\{\infty\} \times
    (\simp^{\op})^{\triangleright}}$ lies in a single fibre; it is
  therefore a colimit in $\mathcal{D}$ since $U$ preserves geometric
  realizations in each fibre. Applying Lemma~\ref{lem:bisimpcond}
  again we conclude that the last row
  $(U\overline{\Phi})|_{(\simp^{\op})^{\triangleright} \times \{\infty\}}$
  is also a colimit, i.e.~the colimit of $\phi$ is preserved by $U$.
\end{proof}

\begin{proof}[Proof of Theorem~\ref{thm:siftedfreemonadmonadic}]
  We will apply Proposition~\ref{propn:fibadjmnd} to the commutative
  triangle
  \opctriangle{\LMndLSRop}{\LEndLSRop}{\LCatISRop.}{U}{p}{q}
  From our previous results the required conditions are satisfied
  here:
  \begin{enumerate}[(1)]
  \item $p$ and $q$ are cocartesian fibrations by
    Proposition~\ref{propn:Mndsigmacocart}.
  \item The cocartesian pushforward functors are left adjoints, since
    $p$ and $q$ are also cartesian fibrations, and so
    preserve all colimits.
  \item $U$ has a left adjoint $F$ such that $pF \simeq q$ by
    Proposition~\ref{propn:freemonadomega}.
  \item This adjunction restricts to an adjunction in each fibre by
    construction.
  \item The fibrewise right adjoints preserve sifted colimits by
    Proposition~\ref{propn:Mndhasifted}. 
  \item The fibrewise right adjoints detect equivalences since they
    are monadic by Corollary~\ref{cor:freesiftedmonad}.
  \end{enumerate}
\end{proof}

\section{Analytic Monads and $\infty$-Operads}
\label{sec:analytic-trees}

\subsection{Analytic Monads}\label{subsec:anmnd}
An \emph{analytic monad} is a monad on $\mathcal{S}_{/I}$ whose
underlying endofunctor is analytic, and whose unit and multiplication
transformations are cartesian. In other words, it is an associative
algebra in $\AnEnd(I)$ under composition. We write $\AnMnd(I)$ for the
\icat{} of analytic monads on $\mathcal{S}_{/I}$, defined as the
subcategory of $\Mnd(\mathcal{S}_{/I})$ with analytic monads as
objects and the morphisms of monads whose underlying maps in
$\End(\mathcal{S}_{/I})$ are cartesian transformations as
morphisms. Similarly, we define an \icat{} $\AnMnd$ over $\mathcal{S}$
of analytic monads over varying base spaces as a subcategory of the
pullback of $\LMndLop \to \LCatI^{\op}$ along
$\mathcal{S}_{/\blank}^{*} \colon \mathcal{S} \to \LCatI^{\op}$.  We
then get a commutative diagram
\opctriangle{\AnMnd}{\AnEnd}{\mathcal{S}.}{}{}{}
We will now use our results on free monads to show that the forgetful
functor $\AnMnd \to \AnEnd$ has a left adjoint, and the resulting
adjunction is monadic.

To prove this, we will first show that the free monad on an analytic
endofunctor is analytic:
\begin{propn}\label{propn:Pbar-anal}
  The free monad $\freemonad{P}$ on an analytic endofunctor $P$ is
  again an analytic endofunctor, and its structure maps $\mu \colon
  \freemonad{P} \circ \freemonad{P} \to \freemonad{P}$ and $\eta
  \colon \id \to \freemonad{P}$ are cartesian natural transformations.
\end{propn}

We will prove this using the colimit description of the free monad 
$$
\freemonad{P} \simeq \colim_n P_n
$$
from Proposition~\ref{propn:freemonadind}. The key observation is the following:
\begin{lemma}\label{lem:Ph-anal}
  Each endofunctor $P_n$ is analytic, and the transition maps
  $f_{n+1}\colon P_n \to P_{n+1}$ are cartesian.
\end{lemma}

\begin{proof}
  The case $n=0$ is clear since $P_{0} \simeq \id$ is certainly analytic, and
  $f_{1} \colon \id \to \id \amalg P$ is cartesian since for any
  map $s \colon X \to Y$ over $I$, the square \csquare{X}{X \amalg
    PX}{Y}{Y \amalg PY}{}{s}{P_1 s}{} is cartesian, as coproducts of cartesian
  squares are cartesian in an $\infty$-topos. If $P_{n}$ is analytic,
  then $P_{n+1} \simeq \id \amalg (P \circ P_{n})$ is analytic, as
  composites and colimits of analytic functors are analytic, and if
  $f_{n}$ is cartesian, then
  $f_{n+1} \simeq \id \amalg P(f_n)$ is cartesian: the squares
  \nolabelcsquare{X \amalg P(P_{n-1}X)}{X \amalg P(P_{n} X)}{Y \amalg
    P(P_{n-1}Y)}{Y \amalg P(P_{n}Y)}
  are cartesian since $P$ preserves pullbacks, and coproducts of
  cartesian squares are cartesian in an $\infty$-topos. This implies
  the required result by induction.
\end{proof}

\begin{proof}[Proof of Proposition~\ref{propn:Pbar-anal}]
  We have $\freemonad{P} \simeq \colim_n P_n$.  By
  Proposition~\ref{propn:colimpolyispoly}, the colimit of any diagram
  of polynomial functors and cartesian transformations is again a
  polynomial functor (corresponding to the colimit of the associated
  polynomials). Since analytic endofunctors are a slice of polynomial
  endofunctors by Corollary~\ref{cor:Polykappa=topos} (or since we
  know from Lemma~\ref{lem:Pbarsift} that $\freemonad{P}$ preserves
  sifted colimits), the endofunctor $\freemonad{P}$ is therefore analytic.

  According to Proposition~\ref{propn:mudescript}, the multiplication
  $\mu\colon \freemonad{P}\circ \freemonad{P} \to \freemonad{P}$ is
  the colimit of the natural transformations
  $\mu_h \colon P_h \circ P_h \to P_{2h}$ of
  Construction~\ref{constr:PnPm}.  Tracing through the definitions,
  these are constructed from sum inclusions (which are cartesian since
  in slices $\mathcal{S}_{/I}$ sums are disjoint), applying $P$ (which
  preserves cartesianness since it is itself cartesian), and sums of
  cartesian natural transformations, which are again cartesian (since
  $\mathcal{S}_{/I}$ is locally cartesian closed).  So all the natural
  transformations $\mu_h\colon P_h \circ P_h \to P_{2h}$ are
  cartesian.  Finally, a filtered colimit of cartesian natural
  transformations is again cartesian by
  Proposition~\ref{propn:Cartcolim}, 
  so $\mu \simeq \colim_h \mu_h$ is also cartesian.  As to the unit
  $\eta\colon \id \to \freemonad{P}$, by
  Proposition~\ref{propn:mudescript} it is the filtered colimit of the
  natural transformations $\eta_h \colon \id \to P_h$, each being just
  a sum inclusion and hence cartesian.  Thus $\eta$ is again
  cartesian.
\end{proof}

\begin{lemma}\label{lem:freeoncart}
  If $u \colon R' \to R$ is a cartesian natural transformation between
  polynomial endofunctors on $\mathcal{S}_{/I}$, then the induced
  natural transformation
  $\freemonad{u} \colon \freemonad{R}{}' \to \freemonad{R}$ of free
  monads is again cartesian.
\end{lemma}
\begin{proof}
  The natural transformation $\freemonad{u}$ is the colimit of the
  sequence of natural transformations $u_h: R'_h \rightarrow R_h$ with
  $u_0 := \id_{\id}$ and $u_{h+1}: R'_{h+1} \to R_{h+1}$ defined as the
  composite
  $$
  \id \amalg (R'\circ R'_h) \xrightarrow{\id \amalg R'(u_h)}
  \id \amalg (R' \circ R_h) \xrightarrow{\id \amalg u}
  \id \amalg (R\circ R_h) .
  $$
  Each $u_h$ is a cartesian natural transformation.  Indeed, $u_0\simeq\id$
  clearly is, and if $u_h$ is then so is $u_{h+1}$ since $R'$ preserves
  pullbacks and $u$ is cartesian.  Finally, since $\freemonad{u}$ is
  the filtered colimit of cartesian natural transformations, it is 
  again cartesian (since pullbacks distribute over filtered colimits).
\end{proof}

\begin{lemma}\label{lem:anadj}\ 
  \begin{enumerate}[(i)]
  \item   If $P$ is an analytic endofunctor, then the unit map $P \to
    \freemonad{P}$ is cartesian.
  \item If $(T,\mu,\eta)$ is an analytic monad, then the counit map
    $\freemonad{T} \to T$ is cartesian.
  \end{enumerate}
\end{lemma}

\begin{proof}
  The unit map $P \to \freemonad{P}$ is the sum inclusion
  $P \to \id \amalg P = P_1$ followed by the colimit map $P_1 \to
  \freemonad{P}$.
  Sum inclusions are cartesian by disjointness of sums, and the colimit map is
  cartesian by 
  Proposition~\ref{propn:Cartcolim}, since all the transition maps are cartesian
  by Lemma~\ref{lem:Ph-anal}.
  
  The counit map is the map of monads corresponding to the forgetful
  functor $\phi \colon \Alg_{T}(\mathcal{S}_{/I}) \to
  \alg_{T}(\mathcal{S}_{/I})$. If we denote these two monadic
  adjunctions by
  \[ F : \mathcal{C} \rightleftarrows \Alg_{T}(\mathcal{S}_{/I}) :
  U\qquad f : \mathcal{C} \rightleftarrows \alg_{T}(\mathcal{S}_{/I}) :
  u,\]
  then this natural transformation $\freemonad{T} \simeq uf \to UF
  \simeq T$ is given as the composite
  \[ uf \to ufUF \simeq ufu\phi F \to u \phi F \simeq UF,\]
  where the first map comes from the unit for $F \dashv U$ and the
  second map from the counit for $f \dashv u$. Unwinding our
  description of the counit $ufu \to u$ from Lemma~\ref{lem:en}, we
  see that this is the colimit of a sequence of natural
  transformations $\varepsilon_h \colon T_h \to T$ defined 
  recursively as follows: $\varepsilon_0 \colon T_0 = \id \to T$ is the unit of 
  the monad, cartesian by assumption.
  Assuming we have a cartesian natural transformation 
  $\varepsilon_h \colon T_h \to T$, the next map 
  $\varepsilon_{h+1}\colon \id \amalg (T\circ T_h) \to T$
  is defined as the sum of the unit $\id \to T$ and the composite
  $$
  T\circ T_h \stackrel{T(\varepsilon_h)}\longrightarrow T \circ T 
  \stackrel\mu\longrightarrow T .
  $$
  But $\varepsilon_h$ is cartesian by induction, and $T$ preserves cartesian maps
  since it is analytic, and $\mu$ is cartesian by assumption.  The colimit of
  all the $\varepsilon_h$ is the natural transformation $\freemonad{T}
  \to T$, which is then cartesian by descent (Proposition~\ref{propn:Cartcolim} again).
\end{proof}

\begin{cor}\label{cor:Pbar=an}
  The forgetful functor $\AnMnd(I) \to \AnEnd(I)$ has a left adjoint,
  taking an analytic endofunctor to its free monad, and the resulting
  adjunction is monadic.
\end{cor}
\begin{proof}
  By Corollary~\ref{cor:freesiftedmonad} we have a monadic adjunction
  \[ F : \End^{\sigma}(\mathcal{S}_{/I}) \rightleftarrows
  \Mnd^{\sigma}(\mathcal{S}_{/I}) : U.\]
  From Proposition~\ref{propn:Pbar-anal} and Lemma~\ref{lem:freeoncart} we
  know that the composite
  \[ \AnEnd(I) \to \End^{\sigma}(\mathcal{S}_{/I}) \xto{F}
  \Mnd^{\sigma}(\mathcal{S}_{/I})\]
  lands in the subcategory $\AnMnd(I)$. Moreover, by
  Lemma~\ref{lem:anadj} the unit and counit transformations for
  $F \dashv U$ restrict to natural transformations valued in
  $\AnEnd(I)$ and $\AnMnd(I)$. Since these restrictions still satisfy
  the adjunction identities, the adjunction restricts to an adjunction
  between $\AnEnd(I)$ and $\AnMnd(I)$. Identifying $\AnMnd(I)$ with
  $\Alg(\AnEnd(I))$, we see by the same proofs as for
  Proposition~\ref{propn:Mndhasifted} and
  Corollary~\ref{cor:freesiftedmonad} that $\AnMnd(I) \to \AnEnd(I)$
  is a monadic right adjoint.  
\end{proof}

Now, by the exact same argument as in the proof of
Theorem~\ref{thm:siftedfreemonadmonadic}, we get:
\begin{cor}\label{cor:freeanalyticmonadic}
  The forgetful functor $\AnMnd \to \AnEnd$ has a left adjoint,
  compatible with the projections to $\mathcal{S}$, and
  the resulting adjunction is monadic.
\end{cor}

\subsection{Free Analytic Monads in Terms of Trees}
\label{subsec:freeanalytictrees}

In this section we will obtain an explicit description of the free
monad on an analytic endofunctor in terms of trees, thus extending the
description of free analytic monads from \cite{KockTree} to the
\icatl{} setting. 

\begin{defn}
  We shall need various groupoids derived from $\bbOint$.
  First of all let
  \[
  \tr := \iota\bbOint
  \]
  denote the groupoid of all trees, and 
  let $\corol \subseteq \tr$ denote the subgroupoid of corollas (which
  is equivalent to $\coprod_{n \geq 0} B\Sigma_{n}$).
\end{defn}

\begin{defn}
  If $T$ is a tree
  \[ A \xfrom{s} M \xto{p} N \xto{t} A,\]
  then the \emph{leaves} of $T$ are the elements of $A$ that are not
  in the image of $t$. 
  Morphisms of trees do not necessarily preserve leaves, but isomorphisms
  do, yielding a functor $\leaves \colon \tr \to \iota\Fin$.
\end{defn}

\begin{defn}\label{def:trP}
  Suppose $P$ is an analytic endofunctor represented by the diagram
  \[
  I \from E \xto{p} B \to I.
  \]
  Define spaces $\tr'$, $\tr(P)$ and $\tr'(P)$ by pullback as follows:
  \[
  \begin{tikzcd}
  \tr'(P) \drpullback \arrow{r}{} \arrow{d}{} & \tr(P) \drpullback \arrow{r}{} \arrow{d}{} & \AnEnd{}_{/P} \arrow{d}{} \\
  \tr' \drpullback \arrow{r}{} \arrow{d}{} & \tr \arrow{r}{} 
  \arrow{d}{\leaves}& \AnEnd \\
  \iota\Fin_* \arrow{r}{} & \iota\Fin  .&
  \end{tikzcd}
  \]
 The objects of $\tr(P)$ are \emph{$P$-trees}, defined as 
  diagrams
    \[
  \begin{tikzcd}
    A \arrow{r} \arrow{d}& M \drpullback
	\arrow{r}\arrow{d}& N\arrow{r}\arrow{d} &
    A \arrow{d}\\ 
    I \arrow{r}& E\arrow{r} & B\arrow{r} & I  ,
  \end{tikzcd}
  \]
 where the first row is a tree.
  The objects of $\tr'(P)$ are $P$-trees with a marked leaf,
  which amount to diagrams
    \[
  \begin{tikzcd}
    * \arrow{r} \arrow{d}& \emptyset \drpullback
	\arrow{r}\arrow{d}& \emptyset \drpullback \arrow{r}\arrow{d} &
    * \arrow{d}\\     A \arrow{r} \arrow{d}& M \drpullback
	\arrow{r}\arrow{d}& N\arrow{r}\arrow{d} &
    A \arrow{d}\\ 
    I \arrow{r}& E\arrow{r} & B\arrow{r} & I  .
  \end{tikzcd}
  \]
  Here the upper right square being a pullback expresses that the 
  edge is a leaf (cf.~\cite{KockTree}).
  (Note that while $\tr$ and $\tr'$ are $1$-truncated (i.e.~ordinary 
  groupoids), $\tr(P)$ and $\tr'(P)$ are not in general so.  For 
  example, $\tr(P)$ contains the space $B$ (cf.~\ref{lem:elmaps}).) 
  The map $\tr(P) \to
  \tr$ corresponds to the functor $\Map(\blank,
  P) \colon \tr \to \mathcal{S}$, and we have the explicit formula
  \[ \tr(P)  \simeq \colim_{T \in \tr}\Map_{\AnEnd}(T, P) \simeq
  \coprod_{T \in \tr} \Map_{\AnEnd}(T,
  P)_{\mathrm{h}\Aut(T)},\]
  where we are implicitly identifying the groupoid $\tr$
  with its opposite.

  The vertical composite $\tr(P) \longrightarrow \tr 
  \stackrel{\leaves}\longrightarrow \iota\Fin$ 
  factors
  also through ``$P$-coloured finite sets'', which could be denoted 
  $\iota\Fin(P)$ or $\mathbf{E}(I)$:
  $$
  \tr(P) \stackrel{\leaves}\longrightarrow  \mathbf{E}(I) \to \iota\Fin .
  $$

  There is a canonical map $\rho\colon \tr(P) \to I$ which to a $P$-tree assigns the
  colour of its root edge.  Formally, for each tree $T$ consider the 
  inclusion of the root edge $\eta\to T$.  The associated maps $\Map(T,P) \to 
  \Map(\eta,P) \simeq I$ assemble into $\tr(P) \simeq \colim_{T\in \tr} 
  \Map(T,P) \to \Map(\eta,P) \simeq I$.  Similarly, there is a canonical
  map $\lambda\colon \tr'(P) \to I$ which returns the colour of the 
  marked leaf.  Formally, this is the composite 
  $\tr'(P) \simeq \colim_{T\in \tr'} 
  \Map(T,P) \to \Map(\eta,P) \simeq I$, where this time $\eta\to T$ picks 
  out the marked leaf.
\end{defn}

\begin{thm}\label{thm:freeanalmonad}
  If $P$ is an analytic endofunctor, then $\freemonad{P}$, the
  (underlying endofunctor of the) free monad on $P$, is represented by
  the polynomial
  \[
  I \stackrel{\lambda}\longleftarrow \tr'(P) \longrightarrow \tr(P)
  \stackrel{\rho}\longrightarrow I.
  \]
\end{thm}

\begin{remark}
  See \cite{KockTree} for the analogous result in the case of sets,
  and \cite{KockDSE} for the groupoid case.
\end{remark}

To prove this, we use the description of $\freemonad{P}$ given in
Definition~\ref{defn:Pn} and Proposition~\ref{propn:freemonadind}, as the
colimit of the sequence of functors defined by $P_0 \simeq \id$ and $P_{h+1} \simeq
\id \amalg (P \circ P_{h})$, which we shall describe in terms of trees of
bounded height.  For this we need some notation.

\begin{defn}
  For $A \leftarrow M \to N \to A$ a tree, the \emph{height} of $e\in
  A$ is the minimal $k\in \mathbb{N}$ such that $\sigma^{k}(e)$ is the
  root edge.  Here $\sigma$ is the `successor' function (or
  walk-to-the-root function) from the definition of tree
  (\ref{defn:tree}).  The \emph{height} of the tree $T$ is the maximal
  height of its edges.  Hence the trivial tree has height $0$ and any
  corolla has height $1$.
  
  Let $\tr_{\leq h}$ denote the subgroupoid of
  $\tr$ containing only the trees of height $\leq h$.
  For $P$ an analytic endofunctor, and for each $h\in \mathbb{N}$, 
  define groupoids $\tr'_{\leq h}$, $\tr_{\leq h}(P)$ and $\tr'_{\leq h}(P)$
  by pullbacks
  \[
  \begin{tikzcd}
  \tr'_{\leq h}(P) \drpullback \arrow{r}{p_{h}} \arrow{d}{} & \tr_{\leq h}(P) \drpullback \arrow{r}{} \arrow{d}{} & \AnEnd{}_{/P} \arrow{d}{} \\
  \tr'_{\leq h} \drpullback \arrow{r}{} \arrow{d}{} & \tr_{\leq h} \arrow{r}{} 
  \arrow{d}{\leaves}& \AnEnd \\
  \iota\Fin_* \arrow{r}{} & \iota\Fin &
  \end{tikzcd}
  \]
  We have
  \[
  \tr_{\leq h}(P) := \colim_{T
    \in \tr_{\leq h}}\Map_{\AnEnd}(T, P) .
	\]
  Let
  $\tr_{\bullet,\leq h}$ denote the subgroupoid
  containing the trees of height $\leq h$ that have a root node
  (i.e.~we exclude $\eta$). We have a forgetful functor $\bottomnode
  \colon \tr_{\bullet,\leq h} \to \tr_{\bullet,\leq 1} \simeq \corol \simeq \iota\Fin$
  that takes a tree to its root corolla.
\end{defn}

\begin{lemma}\label{lem:fibreofbottomnode}
  For each fixed $k\in \iota\Fin$, we have a pullback square
\[
\begin{tikzcd}
	(\tr_{\leq h})^{\times k} \drpullback \arrow{r}\arrow{d} & 
	\tr_{\bullet,\leq h+1} \arrow{d}{\bottomnode} \\
	* \arrow[swap]{r}{k} & \iota\Fin 
\end{tikzcd}
\]
\end{lemma}
\begin{proof}
  The map $(\tr_{\leq h})^{\times k} \to \tr_{\bullet,\leq h+1}$
  takes a $k$-tuple of height-$h$ trees and
  grafts them onto the corolla $k$.  It is clear that the fibre of
  this map is the set of automorphisms of $k$, just as the fibre of
  $k\colon * \to \iota\Fin$.
\end{proof}

Let $P$ be an analytic endofunctor.  Recall from Definition~\ref{defn:Pn}
the sequence of endofunctors $P_h$ defined by $P_0 := \id$, and $P_{h+1} :=
\id \amalg (P \circ P_{h})$.  By Lemma~\ref{lem:Ph-anal}, each $P_h$ is
analytic.

\begin{propn}\label{propn:Ph}
  If $P$ is represented by the polynomial
  \[ I \from E \to B \to I,\]
  then $P_h$ is represented by the polynomial
  \[
  I \stackrel{s_{h}}\longleftarrow \tr'_{\leq h}(P) 
  \stackrel{p_{h}}\longrightarrow
  \tr_{\leq h}(P) \stackrel{t_{h}}\longrightarrow  I .
  \]
  Here $s_{h}$ assigns to a leaf-marked tree the colour of the marked leaf,
  and $t_{h}$ assigns the colour of the root edge.
\end{propn}
The proof requires a couple of auxiliary results,  exploiting that the 
analytic functor $P$ lives over $\mathbf{E}$.  With notation as
in Proposition~\ref{propn:pushforwardpbsq}, we have the diagram
  \[
  \begin{tikzcd}
    I \arrow[swap]{d}{\jmath}& E \arrow{l}[above]{s} \arrow{r}{p} \arrow[swap]{d}{\epsilon} 
    \drpullback & B \arrow{r}{t} \arrow{d}{\beta}& I\arrow{d}{\jmath} \\
    * & \iota\Fin_{*} \arrow{l}{u} \arrow[swap]{r}{q} & \iota\Fin \arrow{r} & *.
  \end{tikzcd}
  \]

\begin{lemma}\label{lem:treerootfib}
  With notation as above, there is a natural pullback square
\[
\begin{tikzcd}
	\tr_{\bullet, \leq h+1}(P) \drpullback
	\arrow{r}{} \arrow[swap]{d}{\bottomnode} & 
	\mathbf{E}(\tr_{\leq h}(P)) \arrow{d}{\mathbf{E}(t_h)} \\
	\corol(P)\simeq B \arrow[swap]{r}{\leaves = \bar s} & \mathbf{E}(I),
	\end{tikzcd}
\]
  where $\bar s \colon B \to q_* u^* I = \mathbf{E}(I)$ corresponds to
  $u_!  q^* B = E \stackrel{s}\to I$ under the adjunction
  $u_! q^* \dashv q_* u^*$.
\end{lemma}
Unravelling the definitions, this says that giving a $P$-tree with a bottom node and of height 
$\leq h+1$ is the same thing as giving the bottom node (a $P$-corolla)
and a $P$-forest of trees of height $\leq h$ whose root edges match the
leaves of the bottom node.

\begin{proof}
  Expanded in colimits, the asserted square reads as follows.
    \nolabelcsquare{\displaystyle\colim_{T \in \tr_{\bullet,\leq h+1}} \Map(T,
  P)}{\displaystyle\colim_{k \in \iota\Fin} \tr_{\leq h}(P)^{\times k}}{\displaystyle\colim_{k \in \iota\Fin}
  \Map(C_{k}, P)}{\displaystyle\colim_{k \in \iota\Fin} \Map(\eta, P)^{\times k}.}
  By Lemma~\ref{lem:fibreofbottomnode}, we can rewrite
  $\colim_{T \in \tr_{\bullet,\leq h+1}} \Map(T, P)$ as an
  iterated colimit
  \[\colim_{k \in \iota\Fin} \colim_{(T_{i}) \in (\tr_{\leq h})^{\times k}}
  \Map(T, P).\]
  Since we have natural pushouts
  $T \simeq C_{k} \amalg_{\coprod_{i=1}^{k} \eta} \coprod_{i=1}^{k}
  T_{i}$
  this is equivalent to
  \[\colim_{k \in \iota\Fin} \colim_{(T_{i}) \in (\tr_{\leq h})^{\times k}} \Map(C_{k}, P)
  \underset{\prod_{i=1}^{k}\Map(\eta, P)}{\times} \
  \prod_{i=1}^{k}\Map(T_{i}, P).\]
  Since colimits in $\mathcal{S}$ are universal and products commute
  with colimits, we can rewrite this as
  \[
  \colim_{k \in \iota\Fin} \left(\Map(C_{k}, P)
  \underset{\prod_{i=1}^{k}\Map(\eta, P)}{\times} \
  \prod_{i=1}^{k}
  \colim_{T_{i} \in \tr_{\leq h}} \Map(T_{i}, P)\right) \simeq
  \colim_{k \in \iota\Fin} \left(\Map(C_{k}, P) 
  \underset{\Map(\eta, P)^{\times k}}{\times} \
    \tr_{\leq h}(P)^{\times k}\right).
\]
But colimits over $\infty$-groupoids commute with weakly contractible
limits by Lemma~\ref{lem:igpdcolimwclim}, so this is equivalent
to
\[(\colim_{k \in \iota\Fin} \Map(C_{k}, P)) \times_{\colim_{k \in
      \iota\Fin} \Map(\eta, P)^{\times k}} (\colim_{k \in \iota\Fin}
  \tr_{\leq h}(P)^{\times k}).\qedhere\]
\end{proof}

\begin{cor}\label{cor:P(th)}
  $P$ evaluated on the map $t_h \colon \tr_{\leq h}(P)
  \to I$ yields $t_{h+1} \colon \tr_{\bullet, \leq h+1}(P) \to I$.
\end{cor}
\begin{proof}
  Proposition~\ref{propn:pushforwardpbsq} describes $p_{*}s^{*}(t_h)$
  as the vertical left map in the pullback
\[
\begin{tikzcd}
	Y \drpullback
	\arrow{r}{} \arrow{d} & 
	\mathbf{E}(\tr_{\leq h}(P)) \arrow{d}{\mathbf{E}(t_h)} \\
	\corol(P)\simeq B \arrow[swap]{r}{\bar s} & \mathbf{E}(I) ,
	\end{tikzcd}
\]
  and Lemma~\ref{lem:treerootfib} identifies this map as
  $\bottomnode \colon \tr_{\bullet, \leq h+1}(P) \to B$.  Applying finally 
  $t_!$ is to compose with $t$, yielding $t_{h+1}$ as required.
\end{proof}

\begin{proof}[Proof of Proposition~\ref{propn:Ph}]
  This goes by induction on $h$.  First of all, $P_0:=\id$ is
  represented by $I \leftarrow I \to I \to I$, but $I$ is also the
  space of $P$-trees of height $\leq 0$.  This establishes the base of
  the induction.  Assuming we have already established that $P_h$ is
  represented by trees of height $\leq h$, we need to identify the
  composite $P \circ P_{h}$, using the explicit description given in
  Theorem~\ref{thm:comp}.  We have already computed the space in the
  upper right corner (denoted $D$ in the big diagram in
  \ref{thm:comp}): by Corollary~\ref{cor:P(th)}, it is the space
  $\tr_{\bullet, \leq h+1}(P)$ of trees of height $\leq h+1$ and with
  a bottom node.  To compute the space in the upper left corner
  (denoted $G$ in the big diagram in \ref{thm:comp}), we need first to
  pull back along $p: E \to B$: this gives the same space of trees but
  with a marked incoming edge of the bottom node.  This space comes
  with a canonical projection to $\tr_{\leq h}(P)$ given by returning
  the tree sitting over that marked edge.  Finally we need to pull
  back along $p_h$, which amounts to marking a leaf of that marked
  subtree.  Together the two pullbacks amount to marking any leaf,
  giving thus the space $\tr'_{\bullet, \leq h+1}(P)$.  Finally, the
  formula for $P_{h+1}$ adds in the trivial tree by means of the
  summand $\id$.  This compensates precisely for the requirement of
  having a bottom node.
\end{proof}

\begin{proof}[Proof of Theorem~\ref{thm:freeanalmonad}]
  We know from Proposition~\ref{propn:freemonadind} that the free monad 
  $\freemonad{P}$ is the colimit of the sequence $P_h$, where $P_{h}$ is
  represented by
  \[
  I \longleftarrow \tr'_{\leq h}(P) \longrightarrow \tr_{\leq h}(P) \longrightarrow  I
  \]
  by Proposition~\ref{propn:Ph}. It then follows from
  Proposition~\ref{propn:colimpolyispoly},
  Theorem~\ref{thm:polyispolyfun}, and
  Corollary~\ref{cor:polycolim} that the colimit $\freemonad{P}$
  is the polynomial functor corresponding to the pointwise colimit of
  these diagrams, which is clearly
  \[
  I \longleftarrow \tr'(P) \longrightarrow \tr(P) \longrightarrow  I 
  \]
  as asserted.
\end{proof}

\begin{remark}
  The monad structure on $\freemonad{P}$ is also pleasantly described 
  in terms of trees.  The space of operations of $\freemonad{P}\circ 
  \freemonad{P}$ is $\freemonad{P}(\tr(P))$, the space of $P$-trees
  whose leaves are decorated by $P$-trees in a compatible way.  More 
  precisely, the objects of $\freemonad{P}(\tr(P))$ are tuples
  $$
  \big(R, 
  \begin{tikzcd}[cramped, sep=tiny]
	  \leaves(R) \ar[rr, "f"] \ar[rd] &[-1ex]&[-1ex] \tr(P) \ar[ld, "\rho"] \\ & I & 
  \end{tikzcd} 
  \big)
  $$
  where $R$ is a $P$-tree and $f$ assigns to each leaf of $R$ a 
  $P$-tree whose root edge has the same colour.  The monad 
  multiplication $\freemonad{P}\circ \freemonad{P} \to \freemonad{P}$
  now simply takes this configuration and glues those trees onto the 
  leaves of $R$.  Clearly this construction is just the colimit of
  the same construction with trees of height $m$ and $n$, which is the
  tree interpretation of the natural transformations 
  $P_n \circ P_m \to P_{m+n}$ from~\ref{constr:PnPm}.
  
  Cartesianness of $\mu$ can also be established along these lines:
  the arity of an operation $(R,f)$ as above is the disjoint union
  of all the leaves of all the upper trees.  Clearly this is the same
  as the set of all leaves of the resulting total tree.
\end{remark}

\begin{propn}\label{prop:colimPbar}
  Let $P$ be an analytic endofunctor on $\mathcal{S}_{/I}$, and $\freemonad{P}$ the free
  monad on $P$.  Under the equivalence $i^{*} \colon \AnEnd
  \stackrel\sim\to \mathcal{P}(\bbOel)$ of
  Proposition~\ref{propn:AnEndpsh}, the underlying endofunctor of
  $\freemonad{P}$ is identified with the presheaf
  \[ C_n \longmapsto \colim_{T \in \nTr} \Map(T,P).\]
  (and $\eta \mapsto I$).
  Here $\nTr$ is the homotopy fibre over $n$ of the map $\tr \to
  \iota\Fin$ that sends a tree to its set of leaves.
\end{propn}
  
\begin{proof}
  The equivalence $i^*$ sends $\freemonad{P}$ to
  the presheaf $C_n \mapsto \Map(C_n, \freemonad{P})$.  But we know from
  Theorem~\ref{thm:freeanalmonad} that $\freemonad{P}$ is represented by
  $I \leftarrow \tr'(P) \to \tr(P) \to I$.  Now by \ref{lem:elmaps}, we have
  $\Map(C_n , \freemonad{P}) \simeq \nTr(P)$, the latter defined as the left 
  composite pullback
  \[
  \begin{tikzcd}
  \nTr(P) \drpullback \arrow{r}{} \arrow{d}{} & \tr(P) \drpullback \arrow{r}{} \arrow{d}{} & \AnEnd{}_{/P} \arrow{d}{} \\
  \nTr \drpullback \arrow{r}{} \arrow{d}{} & \tr \arrow{r}{} 
  \arrow{d}{\leaves}& \AnEnd \\
  * \arrow[swap]{r}{n} & \iota\Fin  .&
  \end{tikzcd}
  \]
  From the top composite pullback, we get (in analogy with \ref{def:trP})
  $\nTr(P) \simeq \colim_{T\in \nTr} \Map(T,P)$, as claimed.  In summary:
  \[
  i^* \freemonad{P} (C_n) \simeq \Map(C_n , \freemonad{P})
  \simeq \nTr(P) \simeq \colim_{T\in \nTr} \Map(T,P) .\qedhere
  \]
\end{proof}

\subsection{Analytic Monads versus Dendroidal Segal Spaces}\label{subsec:comparison}
In this subsection we will prove the main result of the paper, that
analytic monads are equivalent to \iopds{}. First, we need to recall
the model of \iopds{} we will use for the comparison, namely the
dendroidal Segal spaces of Cisinski and
Moerdijk~\cite{CisinskiMoerdijkDendSeg}.

\begin{defn}
  The \emph{dendroidal category} $\bbO$ is the full subcategory of
  $\AnMnd$ spanned by the image of $\bbOint$, i.e.~the free monads
  on the trees.
\end{defn}
\begin{remark}
  Since trees themselves are polynomials in $\Set$, and since the free
  monad on a set polynomial is again a set polynomial, the definition
  given here agrees with that of \cite{KockTree}, which in turn is
  just a polynomial reformulation of the original definition of
  \cite{MoerdijkWeiss}.  Recall that $\bbO$ has as morphisms the monad
  maps between free monads on the trees, and that $\bbO$ has an
  active--inert factorization system (also called the generic--free
  factorization system~\cite{BergerMelliesWeber}): The inert maps are
  the tree inclusions, defined formally as the morphisms of polynomial
  functors between trees, forming the category $\bbOint$ studied so
  far.    
  The active maps are given by node refinements, characterized also
  as the monad maps that preserve leaves and root. This includes the
  codegeneracy case where a unary node is ``refined'' into a nodeless
  tree. To specify an active map out of a corolla with set of leaves
  $L$ amounts to giving a tree with $L$ as set of leaves.
  The only active map out of the trivial tree $\eta$ is the
    identity.  A general active map $T\to T'$ is specified by giving
    an active map out of each node corolla, and then gluing together
    the resulting trees along roots and leaves, according to the same
    recipe that gave $T$ as the colimit of its elementary trees.
\end{remark}

\begin{defn}\label{def:PSeg}
  A presheaf $F$ on $\bbO$ is called a \emph{Segal presheaf} if its
  restriction $j^{*}F$ along the inclusion $j \colon \bbOint \to \bbO$
  is a Segal presheaf on $\bbOint$, as in \ref{def:Seg}.
  We define the \icat{} $\PSeg(\bbO)$  of Segal presheaves to be the pullback
  \[
  \begin{tikzcd}
	  \PSeg(\bbO) \drpullback \arrow{r} \arrow{d} & \mathcal{P}(\bbO) 
	  \arrow{d}{j^*} \\
	  \PSeg(\bbOint) \arrow{r} & \mathcal{P}(\bbOint) .
  \end{tikzcd}
  \]
\end{defn}

\begin{thm}\label{thm:AnMnd=PSeg}
  The 
  restricted Yoneda embedding $N \colon \AnMnd \to \mathcal{P}(\bbO)$
  is fully faithful, and its essential image is $\PSeg(\bbO)$.  We
  thus have an equivalence of $\infty$-categories
  \[
  \AnMnd \simeq \PSeg(\bbO) .
  \]
\end{thm}

The proof will be based on the following general observation, which is an
\icatl{} version of a result of Berger, Melli\`{e}s and
Weber~\cite{BergerMelliesWeber}; they use it to give a proof of the
``nerve theorem'' for monads (originally due to
Weber~\cite{WeberFamilial}).
\begin{propn}\label{propn:monadff}
  Suppose given a commutative square of \icats{}
  \csquare{\mathcal{E}_{1}}{\mathcal{E}_{2}}{\mathcal{B}_{1}}{\mathcal{B}_{2}}{\bar{\phi}}{U_{1}}{U_{2}}{\phi}
  such that
  \begin{enumerate}[(1)]
  \item the functor $U_{i}$ has a left adjoint $F_{i}$ for $i = 1,2$,
  \item the adjunction $F_{i} \dashv U_{i}$ is monadic for $i = 1,2$,
  \item the functor $\phi$ is fully faithful,
  \item the mate transformation $F_{2}\phi \to
    \bar{\phi}F_{1}$ is a natural equivalence,
  \end{enumerate}
  Then the functor $\bar{\phi}$ is also fully faithful, and its
  essential image consists of those $A \in \mathcal{E}_{2}$ such that
  $U_{2}A$ is in the image of $\phi$. (In other words, the commutative
  square above is cartesian.)
\end{propn}
\begin{proof}
  We first prove that $\bar{\phi}$ is fully faithful, i.e.~that for
  all $A, B \in \mathcal{E}_{1}$ the map
  \[ \Map_{\mathcal{E}_{1}}(A, B) \to \Map_{\mathcal{E}_{2}}(\bar{\phi}A,
  \bar{\phi}B)\]
  is an equivalence.

  First suppose $A$ is free, i.e.~of the form $F_{1}X$ for some $X \in
  \mathcal{B}_{1}$. Then we have natural equivalences
  \begin{multline*}
  \Map_{\mathcal{E}_{1}}(F_{1}X, B) \simeq \Map_{\mathcal{B}_{1}}(X, U_{1}B)
  \simeq \Map_{\mathcal{B}_{2}}(\phi X, \phi U_{1} B) \simeq
  \Map_{\mathcal{B}_{2}}(\phi X, U_{2}\bar{\phi} B) \\ 
  \hfill \simeq
  \Map_{\mathcal{E}_{2}}(F_{2}\phi X, \bar{\phi} B) \simeq
  \Map_{\mathcal{E}_{2}}(\bar{\phi}F_{1}X, \bar{\phi} B).
  \end{multline*}
  For a general $A$, we can choose a $U_{1}$-split simplicial free
  resolution $A_{\bullet}$. Since each $A_{n}$ is free, we have
  natural equivalences
  \[
  \Map_{\mathcal{E}_{1}}(A, B) \simeq \lim \Map_{\mathcal{E}_{1}}(A_{\bullet}, B)
  \simeq \lim \Map_{\mathcal{E}_{2}}(\bar{\phi}A_{\bullet}, \bar{\phi}B).\]
  Since $A_{\bullet}$ is $U_{1}$-split and $U_{2}\bar{\phi} \simeq
  \phi U_{1}$, the simplicial diagram $\bar{\phi}A_{\bullet}$ is
  $U_{2}$-split. Since $F_{2}\dashv U_{2}$ is monadic, this implies
  that the colimit $|\bar{\phi}A_{\bullet}|$ exists and is preserved
  by $U_{2}$. There is a canonical map $|\bar{\phi}A_{\bullet}| \to
  \bar{\phi}(A)$ and to see that it is an equivalence it suffices to
  show that it is one after applying $U_{2}$. But since the diagram
  $U_{1}A_{\bullet}$ is split, so is $\phi U_{1} A_{\bullet} \simeq
  U_{2}\bar{\phi} A_{\bullet}$ and therefore its colimit is $\phi
  U_{1}A \simeq U_{2} \bar{\phi} A$, as required. We thus have a
  natural equivalence
  \[
  \Map_{\mathcal{E}_{1}}(A, B) \simeq \lim
  \Map_{\mathcal{E}_{2}}(\bar{\phi}A_{\bullet}, \bar{\phi}B) \simeq
  \Map_{\mathcal{E}_{2}}(|\bar{\phi}A_{\bullet}|, \bar{\phi}B) \simeq
  \Map_{\mathcal{E}_{2}}(\bar{\phi}A, \bar{\phi}B).\]
  This shows that $\bar{\phi}$ is fully faithful; it remains to prove
  that if $A \in \mathcal{E}_{2}$ satisfies $U_{2}A \simeq \phi X$ for
  some $X \in \mathcal{B}_{1}$, then $A$ is in the image of
  $\bar{\phi}$. We can view $A$ as the geometric realization of its
  canonical free resolution $A_{\bullet} := F_{2}(U_{2}F_{2})^{\bullet}U_{2}A$. We
  have $U_{2}F_{2}\phi \simeq U_{2}\bar{\phi}F_{1} \simeq \phi
  U_{1}F_{1}$, so $F_{2}(U_{2}F_{2})^{n}U_{2}A \simeq
  \bar{\phi}F_{1}(U_{1}F_{1})^{n}X$. Since $\bar{\phi}$ is fully
  faithful, the diagram $A_{\bullet}$ factors through
  $\mathcal{E}_{1}$, i.e.~we have a simplicial diagram $A'_{\bullet}$
  in $\mathcal{E}_{1}$ such that $\bar{\phi}A'_{\bullet} \simeq
  A_{\bullet}$. The diagram $A_{\bullet}$ is also $U_{2}$-split,
  and since $U_{2}A \simeq \phi X$ the extension of $U_{2}A_{\bullet}$
  to a split simplicial diagram factors through $\mathcal{B}_{1}$ as
  $\phi$ is fully faithful. Thus $A'_{\bullet}$ is
  $U_{1}$-split. Since the adjunction $F_{1} \dashv U_{1}$ is monadic,
  this implies that $A'_{\bullet}$ has a colimit $A'$ in
  $\mathcal{E}_{1}$, and this colimit is preserved by $U_{1}$. There
  is then a canonical map $A \to \bar{\phi}A'$, and this is an
  equivalence since $U_{2}$ detects equivalences. This proves that $A$
  is in the essential image of $\bar{\phi}$, as required.
\end{proof}

We are going to apply Proposition~\ref{propn:monadff} to the
commutative diagram
  \[
  \begin{tikzcd}
	  \AnMnd 
	  \arrow{r}{N} \arrow{d}{U} & \mathcal{P}(\bbO) \arrow{d}{j^*} \\
	  \AnEnd \arrow[swap]{r}{N_{\xint}} & \mathcal{P}(\bbOint) .
  \end{tikzcd}
  \]
  The vertical functors have left adjoints $F$ and $j_!$, 
  respectively.  A key step (which will be Proposition~\ref{propn:leftadjcommute} 
  below) is to show that the mate commutes
  \[
  \begin{tikzcd}
	  \AnMnd 
	  \arrow{r}{N}  & \mathcal{P}(\bbO)  \\
	  \AnEnd \arrow{u}{F}\arrow[swap]{r}{N_{\xint}} & \mathcal{P}(\bbOint) .
	  \arrow{u}{j_!}  
  \end{tikzcd}
  \]
  To establish this, we will need:
  \begin{itemize}
  \item a simplified formula for $j^* j_!$ in terms 
  of active maps (Lemma~\ref{lem:colimlem} below),
\item some results regarding compatibility of Segal presheaves with active 
  maps and colimits of subtrees (Corollary~\ref{cor:generalcolimoftrees} and 
  Lemma~\ref{lem:colimPhiT'}),
\item the formula for the free monad on an analytic endofunctor in
  terms of trees, already established in Proposition~\ref{prop:colimPbar},
  which allows for reduction to the case of elementary trees 
  (Lemma~\ref{lem:colimcolim}).
  \end{itemize}

\begin{lemma}\label{lem:fbeta}
  Let $\beta \colon T \to T'$ be an active map.
  \begin{enumerate}[(i)]
  \item There is an induced functor
	\[
	\beta_! \colon \bbOel{}_{/T} \to \bbOint{}_{/T'}
	\]
	which takes an elementary tree $E \to T$ to the inert part of the
	active-inert factorization of the composite $E \to T \to T'$, as in
	\[
	\begin{tikzcd}
		T \arrow{r}{\beta} & T' \\
		E \arrow{u}{f} \arrow[swap]{r}{\alpha} & \beta_! E  .
		\arrow[swap]{u}{f'}
	\end{tikzcd}
	\]
      \item There is an induced colimit decomposition of $T'$ into subtrees 
	$\beta_!(E) \subset T'$:
	\[
	T' \simeq \colim_{E\in \el(T)} \beta_!(E) .
	\]
        \end{enumerate}
\end{lemma}

\begin{proof}
  (i) is clear.  For (ii), note that active maps preserve leaves and root.  The
  colimit $T \simeq \colim_{E\in \el(T)} E$ is an iterated grafting, i.e.~an
  iterated pushout over trivial trees, each included into one tree as the root
  and into another tree as a leaf.  Since the only active map out of $\eta$ is
  the identity, the colimit asserted in (ii) is again an iterated pushout over
  trivial trees, and since for each $E\in \el(T)$ the active map $E \to \beta_!
  E$ preserves leaves and root, the colimit asserted in (ii) is again an iterated
  grafting.  Finally, since all the trees $\beta_!  E$ are subtrees of $T'$, and
  are disjoint on nodes, the colimit defines a subtree of $T'$.  Since each node
  in $T'$ appears in precisely one of these subtrees, the colimit must actually
  be all of $T'$.  (In fact, the whole map $\beta$ is the colimit of the maps $E
  \to \beta_!(E)$, cf.~\cite{KockTree}*{1.3.4 and 1.3.16}.)
\end{proof}

\begin{defn}
  Let $\Fun^{\act}(\Delta^1,\bbO)\subset\Fun(\Delta^1,\bbO)$ denote the full
  subcategory of the arrow category of $\bbO$ spanned by the active maps.
  Thanks to the active--inert factorization system in $\bbO$, the domain
  projection $\Fun^{\act}(\Delta^1,\bbO) \to \bbO$ is a cartesian fibration: the
  cartesian arrows are the squares with codomain arrow inert (see
  \cite{GalvezKockTonksIII}*{Lemma 1.3}).  The associated right fibration we
  pull back to $\bbOint$ and straighten to get a presheaf $\Act \colon
  (\bbOint)^\op \to \mathcal{S}$.  Thus $\Act(T)$ is the $\infty$-groupoid of
  active maps $T \to T'$ in $\bbO$ (actually just a $1$-groupoid); for example,
  $\Act(C_{n})$ is the groupoid $\nTr$ of trees with $n$ leaves.  Note also that
  $\Act(\eta) \simeq *$, the only active map out of $\eta$ being the identity.
\end{defn}

\begin{remark}
  As $\Act \colon (\bbOint)^\op \to \mathcal{S}$ factors through the
  full subcategory $\mathrm{Gpd}\subset\mathcal{S}$ of groupoids
  (equivalently, $1$-truncated spaces), we are really just applying a
  1-categorical straightening result here. One could also directly
  define $\Act$ as an explicit (pseudo)functor, but this would involve
  making arbitrary choices, since inert-active factorizations are only
  defined up to unique isomorphism.
\end{remark}

\begin{lemma}\label{lem:actsegal}
  The presheaf $\Act\colon (\bbOint)^{\op}\to\mathcal{S}$ satisfies the
  Segal condition.  More precisely, for any tree $T$, we have
  $$
  \Act(T)\simeq\lim_{E\in\el(T)}\Act(E) 
  \simeq \prod_{C\in\corol(T)}\Act(C).
  $$
\end{lemma}

\begin{proof}
  The second equivalence follows from $\Act(\eta) \simeq *$.
  The map $\Act(T)\to\prod_{C\in\corol(T)}\Act(C)$ sends an active map
  $\beta\colon T \to T'$ to the collection of active maps $\alpha\colon C \to S'$ as in
  Lemma~\ref{lem:fbeta}(i).  A map in the other direction is given by gluing
  together all the subtrees $S'$ according to the same recipe as the corollas
  $C$ glue together to give $T$, as in Lemma~\ref{lem:fbeta}(ii).  This
  constitutes a bijection at the level of isomorphism classes
  by~\cite{KockTree}*{1.3.16}.  Since the spaces involved are just
  $1$-groupoids it thus remains to check that the automorphism groups match up.
  But an automorphism of an active map $\beta\colon T\to T'$ is the same as an
  automorphism of $T'$ that fixes the edges from $T$, and this amounts to giving
  for each $C \in \corol T$ an automorphism of the corresponding tree $S'$
  that fixes all leaves, which in turn is precisely to give an automorphism of
  the active map $\alpha\colon C \to S'$.  So $\Aut(\beta) = \prod
  \Aut(\alpha)$ as required.
\end{proof}

\begin{lemma}\label{lem:colimlem}
  The active-inert factorization system on $\bbO$ induces an equivalence
  \[
  (j^*j_{!}\Phi)(T) \simeq \colim_{T \to T' \in \Act(T)} \Phi(T')  
  \]
  for each Segal presheaf $\Phi$ and each tree $T$.
\end{lemma}

\begin{proof}
  Since $j\colon\bbOint\to\bbO$ is the identity on objects,
  $(j^*j_!\Phi)(T) \simeq (j_!\Phi)(jT) \simeq (j_!\Phi)(T)$.  By
  the usual formula for the left Kan extension, we have that
  \[
  (j_!\Phi)(T)\simeq\colim_{T\to T'\in ((\bbOint)_{T/})^{\op}}\Phi(T'),
  \]
  so it suffices to show that the functor
  $\Act(T)^{\op}\to ((\bbOint)_{T/})^{\op}$ is cofinal. Invoking
  \cite{HTT}*{4.1.3.1}, it suffices to show that for all
  $f\colon T\to T'\in ((\bbOint)_{T/})$, the pullback
  \[
  \Act(T)\underset{(\bbOint)_{T/}}{\times}(\bbOint)_{T//f}
  \]
  is a weakly contractible \icat{}.
  But this is precisely the $\infty$-category of 
  active-inert factorizations of $f\colon T\to T'$, which is
  contractible by \cite[Proposition 5.2.8.17]{HTT}.
\end{proof}

\begin{propn}\label{propn:TSR}
  Let $\Phi\colon\bbOint^{\op}\to\mathcal{S}$ be a Segal presheaf.  Let $T \simeq S \amalg_\eta R$
  be the tree obtained by grafting a tree $S$ onto a leaf of another tree
  $R$.  Then the canonical
  map
  $$
  \Phi(T) \to \Phi(S) \times_{\Phi(\eta)} \Phi(R)
  $$
  is an equivalence.
\end{propn}
\begin{proof}
  Since $\Phi$ is Segal, we have $\Phi(T) \stackrel{\sim}\to \lim_{E\in \el(T)} 
  \Phi(E)$.  On the other hand, we have $\el(T) \simeq \el(S)  
  \amalg_{\el(\eta)} \el(R)$.  It follows that the limit can be computed in
  steps:
  $$
  \lim_{E\in \el(T)} \Phi(E) \simeq \lim_{E\in \el(S)}\Phi(E) 
  \times_{\Phi(\eta)} \lim_{E\in \el(R)}\Phi(E) \simeq 
  \Phi(S) \times_{\Phi(\eta)} \Phi(R) .
  $$
\end{proof}
\begin{cor}\label{cor:generalcolimoftrees}
  If $\Phi\colon\bbOint^{\op}\to\mathcal{S}$ is a Segal presheaf, and if $T\simeq \colim_{\mathcal{I}} R$ is
  a colimit of certain subtrees grafted to each other, then
  the canonical map
  $$
  \Phi(T) \to \lim_{\mathcal{I}} \Phi(R)
  $$
  is an equivalence.
\end{cor}
\begin{proof}
  This follows by iterated application of Proposition~\ref{propn:TSR}.
\end{proof}

\begin{cor}\label{cor:Phiact}
  For $\Phi\colon\bbOint^{\op}\to\mathcal{S}$ a Segal presheaf and $\beta\colon T \to T'$ an active map,
  the canonical map
  \[
  \Phi(T')\longrightarrow\lim_{E\in\el(T)}\Phi(\beta_!(E))
  \]
  is an equivalence.	
\end{cor}
\begin{proof}
  By Lemma~\ref{lem:fbeta}(ii) we have $T' \simeq \colim_{E\in \el(T)}
  \beta_!(E)$, and the result follows from
  Corollary~\ref{cor:generalcolimoftrees}.
\end{proof}

\begin{lemma}\label{lem:colimPhiT'}
  For $\Phi\colon\bbOint^{\op}\to\mathcal{S}$ a Segal presheaf
  and $T$ a tree, there is a natural equivalence
  \[
  \colim_{T\to T'\in\Act(T)} \Phi(T') 
  \stackrel\sim\longrightarrow 
  \lim_{E\in\el(T)}\colim_{E\to S'\in\Act(E)} \Phi(S') .
  \]
\end{lemma}
\begin{proof}
  Given an active map $\beta\colon T \to T'$ and an elementary subtree $f\colon 
  E \to T$, we can active-inert factor the composite as in 
  Lemma~\ref{lem:fbeta}:
  \[
  \begin{tikzcd}
    T \arrow{r}{\beta} & T' \\
    E \arrow{u}{f} \arrow[swap]{r}{\alpha} & \beta_! E  .
    \arrow[swap]{u}{f'}
  \end{tikzcd}
  \]
  We now have the map
  $$
  \Phi(T') \stackrel{f'{}^*}\longrightarrow \Phi(\beta_! E) \longrightarrow 
  \colim_{E\to S'\in\Act(E)} \Phi(S') ,
  $$
  and letting $\beta$ and $f$ vary, we get altogether the map of the statement.
  By construction we have a
  commutative square of $\infty$-groupoids
  \[
  \begin{tikzcd}
	  \displaystyle\colim_{T\to T'\in\Act(T)} \Phi(T') 
	  \arrow{r}{} \arrow{d}{} & \displaystyle\lim_{E\in\el(T)}\colim_{E\to 
	  S'\in\Act(E)} \Phi(S') \arrow{d}{} \\
	  \Act(T) \arrow{r}{\sim} & 
	  {\displaystyle\lim_{E\in\el(T)}\Act(E)} .
  \end{tikzcd}
  \]
  Since the bottom horizontal map is an equivalence by 
  Lemma~\ref{lem:actsegal}, 	
  to conclude that the top horizontal map is an equivalence, it suffices to
  show that, for any given basepoint $\beta\colon T\to T'$ in $\Act(T)$, the map on
  fibres
  \[
  \Phi(T')\longrightarrow\lim_{E\in\el(T)}\Phi(S')
  \]
  is an equivalence.
  But this is Corollary~\ref{cor:Phiact}, since $S' \simeq \beta_!(E)$.
\end{proof}

\begin{lemma}\label{lem:colimcolim}
  For $P$ an analytic endofunctor, the natural transformation
  $$
  j^* j_! N_\xint P \to N_\xint UF P
  $$
  is an equivalence on elementary trees.
\end{lemma}

\begin{proof}
  The statement is easily seen to be true for the trivial tree 
  $\eta$.  Consider now a corolla $C$.
  On the left side we compute (using the colimit formula for $j_!$ 
  of Lemma~\ref{lem:colimlem})
  $$ 
  (j^* j_! N_\xint P)(C) \simeq \colim_{C \to T'\in\Act(C)} N_\xint P(T') \simeq 
  \colim_{C \to T'\in\Act(C)} \Map(T',P)  .
  $$
  On the right we compute (using the colimit formula for the 
  free-monad monad of Proposition~\ref{prop:colimPbar})
  $$ (N_\xint UF P)(C) \simeq \Map(C,UFP) \simeq (UFP)(C) \simeq
  \colim_{C \to T' \in \Act(C)} P(T') \simeq  \colim_{C \to T'\in\Act(C)}
  \Map(T',P) . \qedhere
  $$
\end{proof}

\begin{propn}\label{propn:leftadjcommute}
  The mate square
  \[
  \begin{tikzcd}
	  \AnMnd 
	  \arrow{r}{N}  & \mathcal{P}(\bbO)  \\
	  \AnEnd \arrow{u}{F}\arrow[swap]{r}{N_{\xint}} & 
	  \mathcal{P}(\bbOint)  \arrow{u}{j_!}
  \end{tikzcd}
  \]
  commutes.
  In other words, the natural transformation
  \[
  j_! \circ N_\xint  \to N \circ F
  \]
  is an equivalence.
\end{propn}

\begin{proof}
  Since $j^*$ is conservative, it is enough to check that 
  $j^* j_! N_\xint  \to  j^* N F \simeq  N_\xint UF$ is an 
  equivalence.   Let $P$ be an 
  analytic endofunctor, and put $\Phi:=N_\xint P$ and $\freemonad{P} := UFP$.
  For $T$ a tree, we have:
  
  \begin{alignat*}{2}
	  j^* j_!  \Phi (T) &\ \simeq\ \colim_{T\to T' \in \Act(T)} \Phi ( T') 
	  &\qquad& \text{by Lemma~\ref{lem:colimlem}} \\
	  &\ \simeq \ \lim_{E \in \el(T)} \colim_{E\to S' \in \Act(E)} \Phi (S')
	  &\qquad& \text{by Lemma~\ref{lem:colimPhiT'}} \\
	  &\ \simeq \ \lim_{E \in \el(T)} j^* j_!  \Phi (E)
	  &\qquad& \text{by Lemma~\ref{lem:colimlem}} \\
	  &\ \simeq \ \lim_{E \in \el(T)} N_\xint \freemonad{P}(E)
	  &\qquad& \text{by Lemma~\ref{lem:colimcolim}} \\
	  &\ \simeq \ N_\xint \freemonad{P}(T)
	  &\qquad& \text{since $N_\xint$ of anything is Segal}
  \end{alignat*}
  as required.
\end{proof}

\begin{proof}[Proof of Theorem~\ref{thm:AnMnd=PSeg}]
  Proposition~\ref{propn:leftadjcommute} tells us that the square of left
  adjoints commutes.  Since the adjunction $F \dashv U$ is monadic by
  Corollary~\ref{cor:freeanalyticmonadic} and $N_\xint$ is fully
  faithful (by Proposition~\ref{prop:ff_int}), we are in position to
  apply Proposition~\ref{propn:monadff}, which now tells us that the
  square of right adjoints is a pullback.  In particular, the nerve
  functor $N \colon \AnMnd \to \mathcal{P}(\bbO)$ is fully faithful.
  Furthermore, $N$ factors through $\PSeg(\bbO)$, since this was
  defined as a pullback (\ref{def:PSeg}), as in this diagram:
  \[
  \begin{tikzcd}
	  \AnMnd \arrow{r}{N} \arrow{d}{U} & \PSeg(\bbO) \arrow{d} 
	  \arrow{r}{} & \mathcal{P}(\bbO) \arrow{d}{j^*} \\
	  \AnEnd \arrow{r}{\sim}[swap]{N_{\xint}} & \PSeg(\bbOint) \arrow{r}& 
	  \mathcal{P}(\bbOint) .
  \end{tikzcd}
  \]
  Since the composite square is a pullback, and the
  right square is a pullback, also
  the left square is a pullback, whence the result.
\end{proof}

\appendix

\section{Mates and Monads}
\label{AppA}
In this appendix we discuss some $(\infty,2)$-categorical 
results needed in order to set up the \icats{}
of polynomial functors and polynomial
monads. For most of these the proofs can be found in the companion
paper \cite{adjmnd}.

\subsection{$(\infty,2)$-Categories and Lax Transformations}
We write $\CatIT$ for the \icat{} of \itcats{}. We will not need to
use any specific model for these objects, but we will need to make use
of the \emph{lax Gray tensor product} of \itcats{}; several versions
of this have recently been constructed
\cite{GagnaHarpazLanariGray,MaeharaGray,OzornovaRovelliVerityGray} in
different models. On the \icatl{} level all produce functors
\[ \otimes^{\lax} \colon \CatIT \times \CatIT \to \CatIT \]
that preserve colimits in each variable.
\begin{remark}
  Let $\bTt \subseteq \CatIT$ denote Joyal's category of
  2-dimensional pasting diagrams. Rezk's presentation
  \cite{RezkThetaN} of \itcats{}
  as complete Segal $\bTt$-spaces implies that the Gray tensor
  product is uniquely determined by its restriction to a functor
  \[ \bTt \times \bTt \to \CatIT. \]
  This is given by the classical Gray tensor products of
  pasting diagrams in all the models, and hence they all produce the
  same functor of \icats{}.\footnote{At one point in \cite{adjmnd} we also need to use the
  further assumption that this Gray tensor product restricts to the classical
  one when applied to ordinary 2-categories, or at least for \emph{gaunt}
  2-categories, which are those with no non-identity invertible
  morphisms and 2-morphisms.}
\end{remark}
Our convention is that a \emph{lax natural transformation} $\eta$ between functors
  $F,G \colon \mathcal{X} \to \mathcal{Y}$ assigns to every morphism
  $f \colon X \to X'$ in $\mathcal{X}$ a lax square
  \[
    \begin{tikzcd}
      F(X) \arrow{r}{\eta_{X}} \arrow[swap]{d}{F(f)} & G(X) \arrow{d}{G(f)}
      \arrow[dl,Rightarrow,inner sep=2pt]\\
      F(X') \arrow[swap]{r}{\eta_{X'}} & G(X'),
    \end{tikzcd}
  \]
and that this is given by a functor $\mathcal{X} \otimes^{\lax}
\Delta^{1} \to \mathcal{Y}$. Similarly, a \emph{colax natural
  transformation} between the same functors assigns to
every morphism
  $f \colon X \to X'$ in $\mathcal{X}$ a colax square
  \[
    \begin{tikzcd}
      F(X) \arrow{r}{\eta_{X}} \arrow[swap]{d}{F(f)} & G(X) \arrow{d}{G(f)}
      \arrow[dl,Leftarrow,inner sep=2pt]\\
      F(X') \arrow[swap]{r}{\eta_{X'}} & G(X'),
    \end{tikzcd}
  \]
and this is given by a functor $\mathcal{X} \otimes^{\colax}
\Delta^{1} \to \mathcal{Y}$, where
\[ \mathcal{X} \otimes^{\colax} \mathcal{Y} := \mathcal{Y}
  \otimes^{\lax} \mathcal{X}.\]
Since $\otimes^{\lax}$ preserves colimits in each variable, we get by adjunction
natural \itcats{} $\FUN(\mathcal{X},\mathcal{Y})_{\pcolax}$ determined
by natural equivalences
\[ \Map_{\CatIT}(\mathcal{W}, \FUN(\mathcal{X}, \mathcal{Y})_{\lax})
  \simeq \Map_{\CatIT}(\mathcal{X} \otimes^{\lax} \mathcal{W},
  \mathcal{Y}),\]
\[ \Map_{\CatIT}(\mathcal{W}, \FUN(\mathcal{X}, \mathcal{Y})_{\colax})
  \simeq \Map_{\CatIT}(\mathcal{X} \otimes^{\colax} \mathcal{W},
  \mathcal{Y}) \simeq \Map_{\CatIT}(\mathcal{W} \otimes^{\lax} \mathcal{X},
  \mathcal{Y}).\]
We write $\Fun(\mathcal{X},\mathcal{Y})_{\pcolax}$ for the underlying
\icat{} of $\FUN(\mathcal{X},\mathcal{Y})_{\pcolax}$; this has
functors $\mathcal{X} \to \mathcal{Y}$ as objects and (co)lax natural
transformations between them as morphisms.

We also write $\FUN(\mathcal{X},\mathcal{Y})$ for the ordinary
internal hom in $\CatIT$ (adjoint to the cartesian product) and
$\Fun(\mathcal{X},\mathcal{Y})$ for its underlying \icat{}. By
\cite{adjmnd}*{Corollary 3.15}, we can identify this with the wide
sub-\itcat{} of $\FUN(\mathcal{X}, \mathcal{Y})_{\pcolax}$ with
morphisms those (co)lax transformations whose (co)lax naturality
squares actually commute.

\subsection{Double $\infty$-Categories of Squares}\label{subsec:SqL}
We think of double \icats{} as simplicial \icats{} that satisfy the
Segal condition, or (essentially equivalently, using the description
of \icats{} as complete Segal spaces) bisimplicial spaces that satisfy
the Segal condition in each variable. If $\mathcal{K}_{\bullet} \colon
\Dop \to \CatI$ is a double \icat{}, we think of
\begin{itemize}
\item the objects of $\mathcal{K}_{0}$ as the objects of the double
  \icat{},
\item the morphisms of $\mathcal{K}_{0}$ as the vertical morphisms,
\item the objects of $\mathcal{K}_{1}$ as the horizontal morphisms,
\item the morphisms of $\mathcal{K}_{1}$ as the squares in the double \icat{}.
\end{itemize}
The corresponding bisimplicial space is given by
\[ \mathcal{K}_{n,m} := \Map(\Delta^{m}, \mathcal{K}_{n}),\]
so that
\begin{itemize}
\item the space of objects of $\mathcal{K}$ is $\mathcal{K}_{0,0}$,
\item the space of vertical morphisms is $\mathcal{K}_{0,1}$,
\item the space of horizontal morphisms is $\mathcal{K}_{1,0}$,
\item the space of squares is $\mathcal{K}_{1,1}$.
\end{itemize}

\begin{defn}
For any \itcat{} $\mathcal{X}$ we will define double \icats{}
$\SqL(\mathcal{X})$, $\SqCL(\mathcal{X})$ and $\Sq(\mathcal{X})$ where
\begin{itemize}
\item the objects are the objects of $\mathcal{X}$,
\item both the horizontal and vertical morphisms are the morphisms of
  $\mathcal{X}$,
\item the squares are, respectively, lax squares, colax squares, and
  commuting squares in $\mathcal{X}$.
\end{itemize}
These can be defined using the (co)lax Gray tensor product and the
cartesian product as the simplicial \icats{}
\begin{eqnarray*}
 \SqL(\mathcal{X})_{\bullet} &:=& \Fun(\Delta^{\bullet},
  \mathcal{X})_{\colax} \ \footnote{With our convention for (co)lax
    natural transformations and vertical/horizontal morphisms a
    morphism in $\SqL(\mathcal{X})_{1}$ is given by a \emph{colax}
    natural transformation of functors $\Delta^{1} \to \mathcal{X}$.}
	\\
 \SqCL(\mathcal{X})_{\bullet} &:=& \Fun(\Delta^{\bullet},
  \mathcal{X})_{\lax}
  \\
  \Sq(\mathcal{X})_{\bullet} &:=& \Fun(\Delta^{\bullet}, \mathcal{X}) . 
\end{eqnarray*}
The Segal condition follows immediately from the assumption that the
Gray tensor product preserves colimits in each variable.
\end{defn}

\begin{remark}
Equivalently, these double \icats{} are given by the bisimplicial spaces
  \begin{enumerate}[(1)]
\item $\SqL(\mathcal{X})_{n,m} := \Map(\Delta^{m}, \Fun(\Delta^{n},
  \mathcal{X})_{\colax}) \simeq \Map( \Delta^{n} \otimes^{\colax}
  \Delta^{m}, \mathcal{X}) \simeq \Map(\Delta^{m} \otimes^{\lax}
  \Delta^{n}, \mathcal{X}), $
\item $\SqCL(\mathcal{X})_{n,m} := \Map(\Delta^{m}, \Fun(\Delta^{n},
  \mathcal{X})_{\lax}) \simeq \Map( \Delta^{n} \otimes^{\lax}
  \Delta^{m}, \mathcal{X}) \simeq \Map(\Delta^{m} \otimes^{\colax}
  \Delta^{n}, \mathcal{X}), $
\item $\Sq(\mathcal{X})_{n,m} :=  \Map(\Delta^{m}, \Fun(\Delta^{n},
  \mathcal{X})) \simeq \Map( \Delta^{n} \times
  \Delta^{m}, \mathcal{X}) \simeq \Map( \Delta^{m} \times
  \Delta^{n}, \mathcal{X})$.
\end{enumerate}
\end{remark}

\begin{notation}
  If $\mathcal{K}$ is a double \icat{}, regarded as a bisimplicial space, we write
  $\mathcal{K}^{\txt{h-op}}$ for the double \icat{} obtained by
  reversing direction in the first coordinate, and
  $\mathcal{K}^{\txt{v-op}}$ for that obtained by reversing direction
  in the second coordinate. We also write $\mathcal{K}^{\txt{rev}}$
  for the double \icat{} obtained by reversing the order of the
  coordinates.
\end{notation}

\begin{remark}\label{rmk:LrevCL}
  Since there is by definition a natural equivalence $\mathcal{X}
  \otimes^{\lax} \mathcal{Y} \simeq \mathcal{Y} \otimes^{\colax} \mathcal{X}$,
  we have a natural equivalence
  \[ \SqL(\mathcal{X})^{\txt{rev}} \simeq \SqCL(\mathcal{X}).\]
  There is also obviously an equivalence $\Sq(\mathcal{X})^{\txt{rev}}
  \simeq \Sq(\mathcal{X})$.
\end{remark}

\begin{defn}\label{defn:sqlaxladj}
  Let $\SqL(\mathcal{X})^{\txt{v=ladj}}$ denote the sub-double \icat{}
  of $\SqL(\mathcal{X})$ containing only the squares where the vertical maps
  are left adjoints. (We will apply similar notations with right
  adjoints and horizontal morphisms, and other types of squares,
  without further comment.)
\end{defn}

\subsection{Naturality of Mates}
Given a diagram of \icats{}
\[
\begin{tikzcd}
  \mathcal{C} \arrow[swap]{d}{\gamma} \arrow{r}{R} & \mathcal{D}
  \arrow{d}{\delta} \arrow[Rightarrow, shorten >= 6pt,shorten <= 6pt]{dl}{\alpha}\\
  \mathcal{C}' \arrow[swap]{r}{R'} & \mathcal{D}',
\end{tikzcd}
\]
where $\alpha$ is a natural transformation $\delta R \to R' \gamma$,
and the functors $R$ and $R'$ have left adjoints $L$ and $L'$,
respectively, then the \emph{mate} of $\alpha$ is the natural
transformation
\[ L' \delta \to L' \delta R L \to L' R' \gamma L \to \gamma L,\]
which we can depict as
\[
\begin{tikzcd}
  \mathcal{C} \arrow[swap]{d}{\gamma}  & \mathcal{D}
  \arrow{d}{\delta} \arrow[swap]{l}{L} \\
  \mathcal{C}'  & \mathcal{D}' \arrow{l}{L'} 
  \arrow[Rightarrow, shorten >= 6pt,shorten <= 6pt]{ul}
\end{tikzcd}
\quad
\txt{or}
\quad
\begin{tikzcd}
  \mathcal{D} \arrow[swap]{d}{\delta} \arrow{r}{L} & \mathcal{C}
  \arrow{d}{\gamma} \\
  \mathcal{D}' \arrow[swap]{r}{L'} 
  \arrow[Rightarrow, shorten >= 6pt,shorten <= 6pt]{ur} & \mathcal{C}'.
\end{tikzcd}
\]
We thus pass from a lax square where the horizontal morphisms are
right adjoints to a colax square where they are left adjoints. A dual
version of this construction takes a colax square where the horizontal
morphisms are left adjoints to a lax square, and doing both gives back
the original square.

We would like to know that the process of taking mates is natural. The
most general form of this statement would be that for any
$(\infty,2)$-category $\mathcal{X}$, taking mates gives a natural
equivalence of double \icats{}
\[ \SqL(\mathcal{X})^{\txt{h=radj}} \isoto
\SqCL(\mathcal{X})^{\txt{h=ladj,h-op}}.\]
We will not establish such an equivalence here; instead, we will
observe that the following weaker statement, where the squares in the
source are required to commute, follows from the results of \cite{adjmnd}:
\begin{propn}\label{propn:mateftr}
  There are morphisms of double \icats{}
  \[ \Sq(\mathcal{X})^{\txt{h=ladj}} \to
  \SqL(\mathcal{X})^{\txt{h=radj,h-op}}, \]
  \[ \Sq(\mathcal{X})^{\txt{h=radj}} \to
  \SqCL(\mathcal{X})^{\txt{h=ladj,h-op}}, \]
  given by taking mates in the horizontal direction.
\end{propn}
\begin{remark}
  Using the equivalence of Remark~\ref{rmk:LrevCL}, we can also
  interpret these as maps
  \[ \Sq(\mathcal{X})^{\txt{v=ladj}} \to
  \SqCL(\mathcal{X})^{\txt{v=radj,v-op}}, \]
  \[ \Sq(\mathcal{X})^{\txt{v=radj}} \to \SqL(\mathcal{X})^{\txt{v=ladj,v-op}}, \]
  given by taking mates in the vertical direction. 
\end{remark}
\begin{proof}
  By \cite{adjmnd}*{Remark 4.11} taking mates gives a natural functor
  \[ \FUN(\Delta^{\bullet}, \mathcal{X})_{\radj} \to
    \FUN((\Delta^{\bullet})^{\op}, \mathcal{X})_{\ladj,\lax}.\]
  The underlying functor of simplicial \icats{} gives precisely a
  functor of double \icats{}
  $\Sq(\mathcal{X})^{\txt{h=radj}} \to
  \SqCL(\mathcal{X})^{\txt{h=ladj,h-op}}$. The other functor is
  defined in the same way by reversing 2-morphisms (which swaps lax
  and colax transformations and left and right adjoints).  
\end{proof}

\subsection{Framed Double $\infty$-Categories}
We will need to know that the source-and-target projection for the
double \icat{} $\SqCL(\CATI)^{\txt{v=radj}}$ is a cartesian and
cocartesian fibration. In order to show this, we will now prove an
\icatl{} version of a result of Shulman~\cite{ShulmanFramed} on double
categories.  To state this we first introduce some terminology:
\begin{defn}[Shulman~\cite{ShulmanFramed}]\label{defn:hframed}
  A double category is \emph{framed} if for every vertical edge
  $f \colon a \to b$, there exist horizontal edges $f_! \colon a\to b$
  and $f^* \colon b \to a$ together with four squares ($2$-cells)
  \[
  \begin{tikzcd}
	  a\arrow{r}{f_!} \arrow[swap]{d}{f} & \arrow[equal]{d} b\\
	  b \arrow[equal]{r} & b
  \end{tikzcd}
  \qquad
  \begin{tikzcd}
	  b\arrow{r}{f^*} \arrow[equal]{d} & \arrow{d}{f} a\\
	  b \arrow[equal]{r} & b
  \end{tikzcd}
  \qquad
  \begin{tikzcd}
	  a\arrow[equal]{r} \arrow[swap]{d}{f} & \arrow[equal]{d} a\\
	  b \arrow[swap]{r}{f^*} & a
  \end{tikzcd}
  \qquad
  \begin{tikzcd}
	  a\arrow[equal]{r} \arrow[equal]{d} & \arrow{d}{f} a\\
	  a \arrow[swap]{r}{f_!} & b
  \end{tikzcd}
  \]
  such that the following four equations hold:
  \[
  \begin{tikzcd}
	  a\arrow[equal]{r} \arrow[equal]{d} & \arrow{d}{f} a\\
	  a \arrow[swap]{r}{f_!} \arrow[swap]{d}{f} & b \arrow[equal]{d} \\
	  b \arrow[equal]{r} & b
  \end{tikzcd}
  =
  \begin{tikzcd}
	  a\arrow[equal]{r} \arrow[swap]{d}{f} & \arrow{d}{f} a\\
	  b \arrow[equal]{r} & b
  \end{tikzcd}
\qquad\qquad
\begin{tikzcd}
	  a\arrow[equal]{r} \arrow[swap]{d}{f} & \arrow[equal]{d} a\\
	  b \arrow[swap]{r}{f^*} \arrow[equal]{d} & a \arrow{d}{f}\\
	  b \arrow[equal]{r} & b
  \end{tikzcd}
  =
  \begin{tikzcd}
	  a\arrow[equal]{r} \arrow[swap]{d}{f} & \arrow{d}{f} a\\
	  b \arrow[equal]{r} & b
  \end{tikzcd}
\]
\vspace*{1ex}
\[
  \begin{tikzcd}
	  a\arrow[equal]{r} \arrow[equal]{d} & \arrow{d}{f} a 
	  \arrow{r}{f_!} & b \arrow[equal]{d}\\
	  a \arrow[swap]{r}{f_!} & b \arrow[equal]{r} & b
  \end{tikzcd}
  =
  \begin{tikzcd}
	  a\arrow{r}{f_!} \arrow[equal]{d} & \arrow[equal]{d} b\\
	  a \arrow[swap]{r}{f_!} & b
  \end{tikzcd}
\qquad
\begin{tikzcd}
	  b\arrow{r}{f^*} \arrow[equal]{d} & \arrow{d}{f} a 
	  \arrow[equal]{r} & a \arrow[equal]{d} \\
	  b \arrow[equal]{r} & b \arrow[swap]{r}{f^*} & a
  \end{tikzcd}
  =
  \begin{tikzcd}
	  b\arrow{r}{f^*} \arrow[equal]{d} & \arrow[equal]{d} a \\
	  b \arrow[swap]{r}{f^*} & a,
  \end{tikzcd}
  \]
  where on the right-hand side we have the horizontal and vertical
  identity squares for $f$, $f^{*}$, and $f_{!}$.
\end{defn}

\begin{remark}
  In \cite{ShulmanFramed}, this structure is called a \emph{framed
    bicategory} rather than a framed double category.
\end{remark}

\begin{defn}\label{defn:framed}
  We say a double \icat{} is \emph{framed} if
  its homotopy double category is framed.
\end{defn}

We have the following \icatl{} version of \cite{ShulmanFramed}*{Thm.4.1}:
\begin{propn}\label{propn:equipmt}
  Let $\mathcal{X}$ be a double \icat{}, viewed as a functor $\Dop \to
  \CatI$ satisfying the Segal condition
  \[ \mathcal{X}_{n} \isoto \mathcal{X}_{1}
  \times_{\mathcal{X}_{0}}\cdots\times_{\mathcal{X}_{0}}
  \mathcal{X}_{1}.\]
  Put
  $\pi := (d_1,d_0) \colon \mathcal{X}_{1}\to \mathcal{X}_{0}\times
  \mathcal{X}_{0}$. Then the following are equivalent:
  \begin{enumerate}[(i)]
  \item The double \icat{} $\mathcal{X}$ is framed.
  \item The functor $\pi$ is a cocartesian fibration.
  \item The functor $\pi$ is a cartesian fibration.
  \end{enumerate}
\end{propn}
\begin{proof}
  The proof that (ii) and (iii) imply (i) is exactly as in the case of
  ordinary double categories, since cartesian (or cocartesian)
  fibrations induce Grothendieck (op)fibrations on the level of
  homotopy categories, and condition (i) is a statement about the
  homotopy double category.  The more interesting direction (which is
  the one we are going to need) is that (i) implies (iii).  So assume
  given (for each vertical edge) the four squares, and assume given
  homotopy equivalences representing the four equations.  Given an
  object in $\mathcal{X}_1$, that is, a horizontal edge $M:b\to d$,
  and an arrow downstairs in $\mathcal{X}_0 \times \mathcal{X}_0$ with
  codomain $(b,d)$, that is altogether a configuration
  \[
  \begin{tikzcd}
	  a \arrow[swap]{d}{f} & c \arrow{d}{g} \\
	  b \arrow[swap]{r}{M} & d  ,
  \end{tikzcd}
  \]
  we claim that
  \begin{equation*}
    \alpha: \qquad
  \begin{tikzcd}
	  a \arrow[swap]{d}{f} \arrow{r}{f_!} & b \arrow[equal]{d} 
	  \arrow{r}{M} \arrow[equal]{d} & d \arrow[equal]{d} \arrow{r}{g^*}& c \arrow{d}{g} \\
	  b \arrow[equal]{r} & b \arrow[swap]{r}{M} & d \arrow[equal]{r} & d
  \end{tikzcd}
  \end{equation*}
  is a cartesian lift. Given vertical edges $\begin{tikzcd} x \arrow{d}{u} 
  \\ a \end{tikzcd}$ and $\begin{tikzcd} y \arrow{d}{v} 
  \\ c \end{tikzcd}$,
  and a horizontal edge $N: x \to y$,
  the claim is that the natural map
  \[
  \Map_{\mathcal{X}_1}(N, g^* M f_!)_{(u,v)} \longrightarrow 
  \Map_{\mathcal{X}_1}(N,M)_{(fu,gv)}
  \]
  given by pasting the square $\alpha$ to the bottom edge
  is a homotopy equivalence.  But we can construct a homotopy inverse 
  by sending a square 
  \[
  \begin{tikzcd}
	  x \arrow[swap]{d}{uf} \arrow{r}{N}& y \arrow{d}{vg} \\
	  b \arrow[swap]{r}{M} & d
  \end{tikzcd}
  \]
  to the pasting
  \[
  \begin{tikzcd}
	  x \arrow[equal]{r} \arrow[swap]{d}{u} & x \arrow{r}{N} 
	  \arrow{d}{u} & y \arrow[swap]{d}{v} \arrow[equal]{r} & y \arrow{d}{v} 
	  \\
	  a \arrow[equal]{r} \arrow[equal]{d} & a \arrow{d}{f} & c 
	  \arrow[equal]{r} \arrow[swap]{d}{g} & c \arrow[equal]{d} \\
	  a \arrow[swap]{r}{f_!} & b \arrow[swap]{r}{M} & d \arrow[swap]{r}{g^*} & c
	  \end{tikzcd}
	  \]
  These two assignments are homotopy inverses: explicit homotopies are
  easily constructed from the homotopy equivalences stipulated in (i).
  
  The proof that (i) implies (ii) is similar.  For reference, we note
  that the cocartesian lifts (of $(f,g)$ to $N$) can be taken to be of
  the form
  \[
    \begin{tikzcd}
	  a \arrow[swap]{d}{f} \arrow[equal]{r} & a \arrow[equal]{d} 
	  \arrow{r}{N} \arrow[equal]{d} & c \arrow[equal]{d} \arrow[equal]{r}& c \arrow{d}{g} \\
	  b \arrow[swap]{r}{f^*} & a \arrow[swap]{r}{N} & c
          \arrow[swap]{r}{g_!} & d. 
  \end{tikzcd}
  \qedhere
\]
\end{proof}

\begin{propn}\label{propn:SqFramed}
  The double \icats{} $\SqL(\CATI)^{\txt{v=ladj}}$ and
  $\SqCL(\CATI)^{\txt{v=radj}}$ are framed.
\end{propn}
\begin{proof}
  We give the proof for $\SqL(\CATI)^{\txt{v=ladj}}$, the other case
  is essentially the same. For each vertical edge, that is a left
  adjoint functor $\ell \colon X \to Y$, with right adjoint
  $r \colon Y \to X$, we have the lax squares
  \[
  \begin{tikzcd}
	  X\arrow{r}{\ell} \arrow[swap]{d}{\ell} & \arrow[equal]{d} Y
	  \arrow[Rightarrow,shorten >= 8pt,shorten <= 8pt]{dl} \\
	  Y \arrow[equal]{r} & Y
  \end{tikzcd}
  \qquad
  \begin{tikzcd}
	  Y\arrow{r}{r} \arrow[equal]{d} & \arrow{d}{\ell} 
	  X \arrow[Rightarrow,shorten >= 8pt,shorten <= 8pt]{dl} \\
	  Y \arrow[equal]{r} & Y
  \end{tikzcd}
  \qquad
  \begin{tikzcd}
	  X\arrow[equal]{r} \arrow[swap]{d}{\ell} & \arrow[equal]{d} X 
	  \arrow[Rightarrow,shorten >= 8pt,shorten <= 8pt]{dl} \\
	  Y \arrow[swap]{r}{r} & X
  \end{tikzcd}
  \qquad
  \begin{tikzcd}
	  X\arrow[equal]{r} \arrow[equal]{d} & \arrow{d}{\ell} X
	  \arrow[Rightarrow,shorten >= 8pt,shorten <= 8pt]{dl} \\
	  X \arrow[swap]{r}{\ell} & Y
  \end{tikzcd}
  \]
  where the second square is the counit $\epsilon_Y$, the third square
  is the unit $\eta_X$, and the two other squares are trivial.  The
  four equations required are two trivial ones, and the triangle laws
  for adjunctions.
\end{proof}

Combining Proposition~\ref{propn:SqFramed} with
Proposition~\ref{propn:equipmt}, we get:
\begin{cor}
  The source-and-target projections
  \[ \SqL(\CATI)^{\txt{v=ladj}}_{1} \to (\CatI^{\txt{ladj}})^{\times
    2}, \qquad \SqCL(\CATI)^{\txt{v=radj}}_{1} \to (\CatI^{\txt{radj}})^{\times
    2},\]
  are cartesian and cocartesian fibrations.
\end{cor}

\subsection{Monads}\label{subsec:monad}
To define the \icat{} of polynomial monads we need to have a suitable
\icat{} of monads on varying base \icats{}. We can define this in
terms of lax natural transformations:
\begin{defn}
  Let $\fmnd$ denote the universal 2-category containing a
  monad.\footnote{This is defined in \cite{RiehlVerityAdj} as a full
    sub-2-category of the universal 2-category containing an adjunction;
    it can also be described as the one-object 2-category corresponding
    to the monoidal envelope of the non-symmetric associative operad.} A
  \emph{monad} in an \itcat{} $\mathcal{X}$ is a functor $\fmnd \to
  \mathcal{X}$. A \emph{lax morphism of monads} (or \emph{monad functor}
  in the terminology of \cite{StreetFormalMonad}) is a lax natural
  transformation of monads, \ie{} a functor $\fmnd \otimes^{\lax}
  \Delta^{1} \to
  \mathcal{X}$. Similarly, a \emph{colax morphism of monads} (or
  \emph{monad opfunctor}) is a colax natural transformation $\fmnd
  \otimes^{\colax} \Delta^{1} \to \mathcal{X}$. We then have \itcats{}
  of monads and
  (co)lax morphisms defined as
  \[ \MND(\mathcal{X})_{\pcolax} := \FUN(\fmnd,
    \mathcal{X})_{\pcolax};\]
  we denote the underlying \icats{} by $\Mnd(\mathcal{X})_{\pcolax}$.
\end{defn}

\begin{remark}
  If $T$ is a monad on $X \in \mathcal{X}$ and $S$ is a monad on
$Y \in \mathcal{X}$,  then
  \begin{enumerate}[(1)]
\item a \emph{lax morphism} $T \to S$ consists of a morphism
  $F \colon X \to Y$ and a natural transformation $SF \to FT$ --- in
  other words, a lax square
  \[
    \begin{tikzcd}
      X \arrow[swap]{d}{T} \arrow{r}{F} & Y \arrow{d}{S}
      \arrow[Rightarrow,shorten >= 8pt,shorten <= 8pt]{dl} \\
      X \arrow[swap]{r}{F} & Y,
  \end{tikzcd}
  \]
  compatible with multiplication and units through commutative diagrams
  \[ 
  \begin{tikzcd}
    F \arrow{rr} \arrow{dr} & & SF \arrow{dl} \\
     & FT,
  \end{tikzcd}
  \qquad
  \begin{tikzcd}
    SSF \arrow{r} \arrow{d} & SF \arrow{dd} \\
    SFT \arrow{d} \\
    FTT \arrow{r} & FT ,
  \end{tikzcd}
  \]
  and so on for iterated composites of $S$ and $T$.
\item a \emph{colax morphism} $T \to S$ consists of a morphism
  $F \colon X \to Y$ and a natural transformation $FT \to SF$ --- in
  other words, a colax square
  \[
    \begin{tikzcd}
      X \arrow[swap]{d}{T} \arrow{r}{F} & Y \arrow{d}{S}
      \arrow[Leftarrow,shorten >= 8pt,shorten <= 8pt]{dl} \\
      X \arrow[swap]{r}{F} & Y,
  \end{tikzcd}
  \]
  compatible with multiplication and units through commutative diagrams
  \[ 
  \begin{tikzcd}
    F \arrow{rr} \arrow{dr} & & FT \arrow{dl} \\
     & SF,
  \end{tikzcd}
  \qquad
  \begin{tikzcd}
    FTT \arrow{r} \arrow{d} & FT \arrow{dd} \\
    SFT \arrow{d} \\
    SSF \arrow{r} & SF ,
  \end{tikzcd}
  \]
  and so on for iterated composites of $S$ and $T$.
\end{enumerate}
\end{remark}
For ordinary 2-categories, Street~\cite{StreetFormalMonad} showed that
the (2-)category of monads and lax morphisms is equivalent to that of
monadic right adjoints and commutative squares between them. One of
the main results of \cite{adjmnd} uses work of
Riehl--Verity~\cite{RiehlVerityAdj} and Zaganidis~\cite{Zaganidis} to
upgrade this to an equivalence of \itcats{}, in the case of monads in
the \itcat{}
$\CATI$ of \icats{}:
\begin{thm}[\cite{adjmnd}*{Corollary 5.7}]\label{thm:Mndlax}
  Let $\FUN(\Delta^{1}, \CATI)_{\mndradj}$ denote the full
  sub-\itcat{} of $\FUN(\Delta^{1}, \CATI)$ spanned by the monadic
  right adjoints. There is an equivalence of \itcats{}
  \[ \MND(\CATI)_{\lax} \to \FUN(\Delta^{1}, \CATI)_{\mndradj}\]
  taking a monad to the right adjoint of its monadic adjunction.
\end{thm}

\begin{cor}[\cite{adjmnd}*{Corollary 5.10}]\label{cor:mndlocradj}
  The inclusion \[ \Fun(\Delta^{1}, \CatI)_{\mndradj} \hookrightarrow
  \Fun(\Delta^{1}, \CatI)_{\radj}\] of the full subcategory of monadic
right adjoints into that of all right adjoints in $\Fun(\Delta^{1},
\CatI)$, has a left adjoint, which takes a right adjoint functor to the right
adjoint of the associated monadic adjunction.
\end{cor}

\begin{defn}
  Let $\fend$ denote the universal category containing an
  endomorphism, \ie{} the pushout
  \[
    \begin{tikzcd}
      \partial \Delta^{1} \arrow{d} \arrow[hookrightarrow]{r} &
      \Delta^{1} \arrow{d} \\
      \Delta^{0} \arrow{r} & \fend,
    \end{tikzcd}
  \]
  or the delooping $B\mathbb{N}$ of the natural numbers under
  addition. If $\mathcal{X}$ is an \itcat{}, we write
  \[ \END(\mathcal{X})_{\pcolax} :=
    \FUN(\fend,\mathcal{X})_{\pcolax} \]
  for the \itcat{} of endomorphisms and (co)lax transformations
  between them; we denote the underlying \icats{} by
  $\End(\mathcal{X})_{\pcolax}$.
\end{defn}

\begin{remark}
  If $T \colon X \to X$ and $S \colon Y \to Y$ are endomorphisms in
  $\mathcal{X}$, then a \emph{lax} morphism from $T$ to $S$ is given
  by a morphism $F \colon X \to Y$ and a lax square
  \[
    \begin{tikzcd}
      X \arrow[swap]{d}{T} \arrow{r}{F} & Y \arrow{d}{S}
      \arrow[Rightarrow,shorten >= 8pt,shorten <= 8pt]{dl} \\
      X \arrow[swap]{r}{F} & Y,
    \end{tikzcd}
  \]
  while a \emph{colax} morphism is again given by a morphism $F \colon
  X \to Y$ but now with a colax square
  \[
    \begin{tikzcd}
      X \arrow[swap]{d}{T} \arrow{r}{F} & Y \arrow{d}{S}
      \arrow[Leftarrow,shorten >= 8pt,shorten <= 8pt]{dl} \\
      X \arrow[swap]{r}{F} & Y.
  \end{tikzcd}
  \]
\end{remark}
\begin{remark}
  There is an inclusion $\fend \to \fmnd$ picking out the underlying
  endomorphism of the universal monad, which induces natural functors
  of \itcats{}
  \[ \MND(\mathcal{X})_{\pcolax} \to \END(\mathcal{X})_{\pcolax}.\]
\end{remark}

\begin{remark}\label{rmk:mndcolaxsql}
  To reduce confusion regarding our conventions for lax vs.~colax, let
  us point out explicitly that there is a functor
  \[ \End(\mathcal{X})_{\lax} \to \Fun(\Delta^{1},
    \mathcal{X})_{\lax} = \SqCL(\mathcal{X})_{1},\]
  and hence a functor
  \[ \Mnd(\mathcal{X})_{\lax} \to \SqCL(\mathcal{X})_{1}.\]
\end{remark}

Another key result from \cite{adjmnd} identifies the fibres of the
underlying functors of \icats{}:
\begin{thm}[\cite{adjmnd}*{Corollary 8.9}]\label{thm:MndAlg}
  For $\mathcal{X}$ an \itcat{} and $X$ an object of $\mathcal{X}$,
  there  are natural identifications
  \[
    \begin{tikzcd}
      \Alg(\End_{\mathcal{X}}(X)) \arrow{r}{\sim} \arrow{d} &
      \Mnd(\mathcal{X})_{\colax, X} \arrow{d} \\
      \End_{\mathcal{X}}(X) \arrow{r}{\sim} & \End(\mathcal{X})_{\colax,X},
    \end{tikzcd}
    \qquad
    \begin{tikzcd}
      \Alg(\End_{\mathcal{X}}(X))^{\op} \arrow{r}{\sim} \arrow{d} &
      \Mnd(\mathcal{X})_{\lax, X} \arrow{d} \\
      \End_{\mathcal{X}}(X)^{\op} \arrow{r}{\sim} & \End(\mathcal{X})_{\lax,X},
    \end{tikzcd}    
  \]
  where $\End_{\mathcal{X}}(X)$ is the monoidal \icat{} of
  endomorphisms of $X$ in $\mathcal{X}$ under composition.
\end{thm}
\begin{remark}
  Combining this with \cref{thm:Mndlax}, we get equivalences
  \[ \Alg(\End(\mathcal{C}))^{\op} \isoto
    \Cat_{\infty/\mathcal{C}}^{\mndradj},\]
  where $\Cat_{\infty/\mathcal{C}}^{\mndradj}$ denotes the full
  subcategory of $\Cat_{\infty/\mathcal{C}}$ spanned by the monadic
  right adjoints. This equivalence has also been obtained by
  Heine~\cite{Heine} by a different method.
\end{remark}
Together with \cref{cor:mndlocradj}, we get:
\begin{cor}\label{cor:mndeq}
  The functor
  \[
  \Mnd(\mathcal{C})^{\op} \to \Cat_{\infty/\mathcal{C}}^{\radj}
  \] 
  that takes a monad to the
  associated right adjoint, is fully faithful, with image the monadic
  right adjoints. (Here $\Cat_{\infty/\mathcal{C}}^{\radj}$ denotes the full
  subcategory of $\Cat_{\infty/\mathcal{C}}$ spanned by the right adjoints.)
\end{cor}

We end by recalling two further results from \cite{adjmnd} that we
will make use of:
\begin{propn}[\cite{adjmnd}*{Proposition 6.4}]\label{propn:EndMndloccoC}\ 
  \begin{enumerate}[(i)]
  \item The projection $\End(\CATI)_{\lax} \to \CatI$ has locally
    cocartesian morphisms and locally cartesian morphisms over
    functors that are right adjoints.
  \item The projection
  $\Mnd(\CATI)_{\lax} \to \CatI$ has locally cocartesian morphisms over
  functors that are right adjoints.
\item The forgetful functor $\Mnd(\CATI)_{\lax} 
  \to \End(\CATI)_{\lax}$ preserves these locally cocartesian morphisms.
  \end{enumerate}
\end{propn}

\begin{defn}
  Let $\CatI^{\txt{radj}}$ denote the subcategory of $\CatI$
  containing only the morphisms that are right adjoints. Then we
  define $\Mnd(\CATI)_{\lax}^{\radj}$ and $\End(\CATI)_{\lax}^{\radj}$
  by pulling back $\Mnd(\CATI)_{\lax}$ and $\End(\CATI)_{\lax}$ along
  the inclusion $\CatI^{\radj} \to \CatI$.
\end{defn}

\begin{cor}[\cite{adjmnd}*{Corollary 6.6}]\label{cor:Mndradjcocart}
  There is a commuting diagram
  \opctriangle{\Mnd(\CATI)_{\lax}^{\radj}}{\End(\CATI)_{\lax}^{\radj}}{\CatI^{\txt{radj}},}{}{}{}
  where the two downward functors are cocartesian fibrations, and the
  horizontal functor preserves cocartesian morphisms. Moreover, the
  right-hand functor is also a cartesian fibration.
\end{cor}

\end{document}